\documentclass[11pt]{amsart}
\usepackage{fullpage}
\usepackage{amscd,amsmath,amsthm,amssymb}
\usepackage{comment}
\usepackage[all,cmtip]{xy}
\usepackage{color}

\usepackage[lined,algonl,boxed,norelsize]{algorithm2e}
\let\chapter\undefined

\usepackage[
backref,
]{hyperref}
\usepackage[
alphabetic,
backrefs,
msc-links,
nobysame,
lite,
]{amsrefs} 

\usepackage{tikz, float} \usetikzlibrary {positioning}
\def\NN{{\mathbb N}}
\def\QQ{{\mathbb Q}}
\def\ZZ{{\mathbb Z}}
\def\RR{{\mathbb R}}

\def\EE{{\mathbb E}}

\def\aa{{\mathfrak a}}
\def\bff{{\mathfrak b}}

\def\II{{\mathbf I}}
\def\MM{{\mathbf M}}
\def\JJ{{\mathbf J}}
\def\OO{{\mathbf O}}
\def\UU{{\mathbf U}}
\def\Mat {{\mathbf{Mat}}}

\def\Rb{{\mathbf R}}
\def\Sb{{\mathbf S}}
\def\Tb{{\mathbf T}}

\def\ab{{\mathbf a}}
\def\bb{{\mathbf b}}
\def\ub{{\mathbf u}}
\def\vb{{\mathbf v}}
\def\wb{{\mathbf w}}
\def\xb{{\mathbf x}}
\def\yb{{\mathbf y}}
\def\zb{{\mathbf z}}
\def\pb{{\mathbf p}}
\def\mb{{\mathbf m}}
\def\eb{{\mathbf e}}
\def\sb{{\mathbf s}}

\def\B{{\mathcal B}}
\def\Oc{{\mathcal O}}

\def\A{{\mathcal A}}

\def\Lc{{\mathcal L}}
\def\M{{\mathcal M}}
\def\H{{\mathcal H}}
\def\Fc{{\mathcal F}}

\def\K{{\mathcal K}}

\def\P{{\mathcal P}}
\def\Ec{{\mathcal E}}
\def\C{{\mathcal C}}
\def\G{{\mathcal G}}
\def\Dc{{\mathcal D}}

\def\opn#1#2{\def#1{\operatorname{#2}}} 
\opn\depth{depth} 
\opn\codim{codim}
\opn\ini{in} 
\opn\LM{LM}
\opn\LC{LC}
\opn\NF{NF}
\opn\Merge{Merge}
\opn\sgn{sgn}
\opn\div{div} 
\opn\Div{Div} 
\opn\Pic{Pic}
\opn\Prin{Prin}
\opn\Del{Del}
\opn\op{op}
\opn\indeg{indeg} 
\opn\outdeg{outdeg}
\opn\red{red}
\opn\Spec{Spec} 
\opn\Supp{Supp} 
\opn\supp{supp} 
\opn\Ker{Ker} 
\opn\Coker{Coker} 
\opn\Hom{Hom}
\opn\Tor{Tor} 
\opn\id{id}
\opn\span{span}
\opn\Image{Image}
\opn\con{conv} 
\opn\relint{rel.int} 
\opn\vol{vol}
\opn\syz{{\rm syz}}
\opn\spoly{{\rm spoly}}
\opn\LM{{\rm LM}}
\opn\lm{{\rm lm}}
\opn\lcm{{\rm lcm}} 
\opn\A{\mathcal A}
\opn\dist{dist}
\opn\pd{pd}
\opn\en{en}
\def\Implies{\ifmmode\Longrightarrow \else
        \unskip${}\Longrightarrow{}$\ignorespaces\fi}
\def\implies{\ifmmode\Rightarrow \else
        \unskip${}\Rightarrow{}$\ignorespaces\fi}
\def\iff{\ifmmode\Longleftrightarrow \else
        \unskip${}\Longleftrightarrow{}$\ignorespaces\fi}
\newtheorem{Theorem}{Theorem}[section]
\newtheorem{Lemma}[Theorem]{Lemma}
\newtheorem{Corollary}[Theorem]{Corollary}
\newtheorem{Proposition}[Theorem]{Proposition}

\theoremstyle{remark}
\newtheorem{Remark}[Theorem]{Remark}

\theoremstyle{definition}
\newtheorem{Example}[Theorem]{Example}

\newtheorem{Definition}[Theorem]{Definition}

\def\qed{\ifhmode\textqed\fi
      \ifmmode\ifinner\quad\qedsymbol\else\dispqed\fi\fi}
\def\textqed{\unskip\nobreak\penalty50
       \hskip2em\hbox{}\nobreak\hfil\qedsymbol
       \parfillskip=0pt \finalhyphendemerits=0}
\def\dispqed{\rlap{\qquad\qedsymbol}}

\tikzstyle{Cwhite}=[scale = .8,circle, fill = white, minimum size=3mm] 
\tikzstyle{Cgray}=[scale = .4,circle, fill = gray, minimum size=3mm] 
\tikzstyle{Cblack2}=[scale = .4,circle, fill = black, minimum size=5mm] 
\tikzstyle{Cblack}=[scale = .7,circle, fill = black, minimum size=3mm]
\tikzstyle{C0}=[scale = .9,circle, fill = black!0, inner sep = 0pt, minimum size=3mm]
\tikzstyle{C1}=[scale = .7,circle, fill = black!0, inner sep = 0pt, minimum size=3mm]
\tikzstyle{Cred}=[scale = .4,circle, fill = red, minimum size=3mm] 

\begin{document}


\title[Divisors on graphs, binomial and monomial ideals]{Divisors on graphs, binomial and monomial ideals, and cellular resolutions}

\author {Fatemeh Mohammadi}
\address{Institut f\"ur Mathematik\\ Technische Universit\"at Berlin\\ 10623 Berlin, Germany}
\email{mohammad@math.tu-berlin.de}

\author {Farbod Shokrieh}
\address{Cornell University\\
Ithaca, New York 14853-4201\\
USA}
\email{farbod@math.cornell.edu}

\thanks{Fatemeh Mohammadi was supported by  the
Alexander von Humboldt Foundation. Farbod Shokrieh was partially supported by NSF grant DMS-1201473.}


\date{\today}

\begin{abstract}
We study various binomial and monomial ideals arising in the theory of divisors, orientations, and matroids on graphs. 
We use ideas from potential theory on graphs and from the theory of Delaunay decompositions for lattices to describe their minimal polyhedral cellular free resolutions. We show that the resolutions of all these ideals are closely related and that their $\ZZ$-graded Betti tables coincide. 
As corollaries, we give conceptual proofs of conjectures and questions posed by Postnikov and Shapiro, by Manjunath and Sturmfels, and by Perkinson, Perlman, and Wilmes. Various other results related to the theory of chip-firing games on graphs also follow from our general techniques and results.

\end{abstract}

\maketitle



\section{Introduction}

This work is concerned with the development of new connections between the theory of divisors on graphs, potential theory, the theory of lattices, Delaunay decompositions, and commutative algebra.

\subsection{Divisors on graphs}

Let $G$ be a finite graph. Let $\Div(G)$ be the free abelian group generated by $V(G)$. An element of $\Div(G)$ is a formal sum of vertices with integer coefficients and is called a {\em divisor} on $G$. 

We denote by $\M(G)$ the group of integer-valued functions on the vertices. The {\em Laplacian operator} $\Delta : \M(G) \to \Div(G)$ is defined by

\[\Delta(f) = \sum_{v \in V(G)} \sum_{\{v,w\} \in E(G)} (f(v) - f(w)) (v) \  .\]

The group of {\em principal divisors} is defined as the image of the Laplacian operator and is denoted by $\Prin(G)$. Two divisors $D_1$ and $D_2$ are called {\em linearly equivalent} if their difference is a principal divisor. This gives an equivalence relation on the set of divisors. The set of equivalence classes forms a finitely generated abelian group which is called the {\em Picard group} of $G$. If $G$ is connected, then the finite (torsion) part of the Picard group has cardinality equal to the number of spanning trees of $G$. This group has appeared in the literature under many different names; in theoretical physics and in probability it was first introduced as the ``abelian sandpile group'' or ``abelian avalanche group'' in the context of self-organized critical phenomena \cite{Bak, Dhar1, Gabrielov}. In arithmetic geometry, it appears implicitly in the study of component groups of N\'eron models of Jacobians of algebraic curves \cite{Raynaud, Lorenzini1}. In algebraic graph theory this group appeared under the name ``Jacobian group'' or ``Picard group'' in the study of flows and cuts in graphs \cite{Bacher}. The study of a certain chip-firing game on graphs led to the definition of this group under the name ``critical group'' \cite{Biggs97, Biggs99}. We recommend the recent survey article \cite{Levine1} for a short but more detailed overview of the subject.

\medskip

The theory of divisors on graphs closely mirrors the theory of divisors on algebraic curves. In fact, Baker and Norine in \cite{BN1} prove a version of Riemann-Roch theorem in this setting via a combinatorial argument. It was immediately realized (in \cite{Gathmann, MZ}) that this divisor theory has a natural extension to {\em metric graphs} (or {\em abstract tropical curves}). This theory, however, has resisted a more conceptual and cohomological interpretation.

\medskip

Associated to $G$ there is a canonical ideal which encodes the equivalences of divisors on $G$. This ideal is already implicitly defined in Dhar's seminal paper \cite{Dhar1}, but it was first introduced in \cite{CoriRossinSalvy02}. Let $K$ be a field and let $\Rb=K[\xb]$ be the polynomial ring in variables $\{x_v: v \in V(G)\}$. The canonical binomial ideal is defined as
$\II_G:= \langle \xb^{D_1} - \xb^{D_2} : \, D_1 \sim D_2 \text{ both nonnegative divisors}\rangle$. 
A related monomial ideal, which we denote by $\MM_G^q$, is a certain initial ideal of $\II_G$ which is defined after fixing a vertex $q \in V(G)$ (see \S\ref{sec:divcomm}). This ideal, for the case of complete graphs, was extensively studied in \cite{PostnikovShapiro04}. In \cite{MadhuBernd}, Riemann-Roch theory for graphs is linked to Alexander duality (see \S\ref{sec:alex}) for the ideal $\MM_G^q$.

\subsection{Minimal polyhedral cellular free resolutions}
\label{sec:res}

There is a standard way to write down a complex of graded modules from a cell complex $\C$ (\cite{BayerSturmfels}). Namely, one can {\em label} $0$-dimensional cells of $\C$ by monomials, and then extend the labeling to arbitrary faces by labeling each face $F$ with the {\em least common multiple} of the monomial labels on the vertices in $F$. The resulting {\em labeled cell complex} leads to a complex of free graded $R$-modules
\[
\Fc_{\C} = \bigoplus_{\emptyset \ne F \in \C}{R(-\mb_F)}
\]
where $\mb_F$ denotes the monomial label of the face $F$. The differential of $\Fc=\Fc_{\C}$ is the homogenized differential of the cell complex $\C$; if $[F]$ denotes the generator of ${R(-\mb_F)}$ we have
\[
\partial([F]) = \sum_{{\rm codim(F,F')=1} \atop F' \subset F}{\varepsilon (F,F') \frac{\mb_F}{\mb_{F'}} \ [F']} 
\]
where $\varepsilon (F,F') \in \{-1, +1\}$ denotes the incidence function indicating the orientation of $F'$ in the boundary of $F$ (see \cite[IX.5]{Massey} or \cite[Section 6.2]{Bruns}). Note that the length of $(\Fc, \partial)$ is the dimension of $\C$. 

\medskip

This construction is so general that the resulting complex is expected not to be exact. It is shown in \cite[Proposition~1.2]{BayerSturmfels} that the complex $(\Fc, \partial)$ is exact if and only if every subcomplex $\C_{\leq \mb}$ (i.e. the subcomplex of $\C$ consisting of all cells whose labels divide the monomial $\mb$) is acyclic over $K$ (i.e. its homology with $K$ coefficients is only in degree $0$). In the rare case that we do get an exact sequence, the pair $(\Fc, \partial)$ is called a {\em cellular free resolution}. If the cell complex is polyhedral, $(\Fc, \partial)$ is called a {\em polyhedral cellular free resolution}. If moreover all ${\mb_F}/{\mb_{F'}}$ appearing in the differential maps are non-units in $R$, then we have a {\em minimal polyhedral cellular free resolution}.

\subsection{Outline and our results}

Our first goal is to give a minimal polyhedral cellular free resolution for the ideal $\II_G$. Quite surprisingly, many ideas from potential theory on graphs, from lattices and Delaunay decomposition, and from (a generalized version of) the notion of total unimodularity (developed in \S\ref{sec:divisors} and \S\ref{sec:latticesec}) fit together nicely to give a direct and self-contained solution to this problem. This is worked out in \S\ref{sec:IGresol}. Note that as a result we obtain a whole family (as $G$ varies) of ideals with minimal polyhedral cellular free resolution. For complete graphs this is the Scarf complex and for trees this is the Koszul complex. 

\medskip

We then step back and define two more ideals; the {\em graphic Lawrence ideal} $\JJ_G$ and one of its initial ideals $\OO_G^q$ (defined after fixing a vertex), which we call the {\em graphic oriented matroid ideal}. These are special classes of more general ideals studied in \cite{Popescu} and \cite{novik2002syzygies}. They are intimately related to {\em graphic hyperplane arrangements} and to {\em Delaunay decomposition of cut lattices} reviewed in \S\ref{sec:arrg}. In \S\ref{sec:matroidlawrence} we take a close look at these ideals, review some general known results, and prove some new results for our special situation.

Roughly speaking, the ideals $\JJ_G$ and $\OO_G^q$ can be thought of as ``orientation'' variants of the ``divisor'' ideals $\II_G$ and $\MM_G^q$. A powerful technique in the theory of divisors on graphs and chip-firing games is to relate divisors to orientations. Given an orientation, one can form a divisor by reading off the associated indegrees or outdegrees (see, e.g., \cite[Theorem~2.3]{Lovasz91}, \cite[Theorem~3.3]{BN1}, \cite{HopPerk}, \cite{FarbodFatemeh}, \cite{ABKS}, and \cite{Fatemehreliability}). Our next main result shows that, algebraically, there is a good justification for the strength of this method. We show that the relation between the ideals $\JJ_G$ and $\II_G$ (and similarly $\OO_G^q$ and $\MM_G^q$) can be understood via {\em regular sequences}.  This is the content of \S\ref{sec:main}. 

\medskip

  \[
 \xymatrixcolsep{5pc} \xymatrix{
    {\JJ_G} \ar@{~>}[r]^{\text{regular sequence}} \ar@{~>}[d]_{\text{initial ideal}} & {\II_G} \ar@{~>}[d]^{\text{initial ideal}} \\
    {\OO_G^q} \ar@{~>}[r]_{\text{regular sequence}}  & {\MM_G^q} 
    }
  \]

\medskip

These regular sequences allow us to compare many algebraic properties and constructions for the ideals $\JJ_G$ and $\II_G$ (and similarly $\OO_G^q$ and $\MM_G^q$). For example, one immediate corollary is to obtain a minimal polyhedral cellular free resolution for the ideal $\II_G$ from a  minimal polyhedral cellular free resolution for the ideal $\JJ_G$. This resolution is essentially equivalent to the one obtained by our potential theoretic considerations (see Remark~\ref{rmk:ResolRelation}). We also obtain a minimal polyhedral cellular free resolution for the ideal $\MM_G^q$ from a minimal polyhedral cellular free resolution for the ideal $\OO_G^q$. It follows that all these resolutions are closely related to Delaunay decompositions of the lattice of integral coboundaries (which we call the {\em integral cut lattice}) and to the graphic hyperplane arrangement. Moreover, the $\ZZ$-graded Betti numbers of all these ideals coincide. So $\MM_G^q$ and $\OO_G^q$ are examples of ``nice'' initial ideals in the sense of \cite{conca}, meaning that one can read the Betti numbers of the original ideal from the initial ideal (see \cite{boocher,Fatemeh} for other such examples). Also, we obtain, automatically, an interpretation of the Betti numbers in terms of the number of faces of various dimensions in the graphic hyperplane arrangement, or equivalently, the number of orbits of the Delaunay cells of various dimensions in the cut or principal lattice. These interpretations also imply that Betti numbers can be read from the number of {\em acyclic partial orientations} of $G$ (see Remark~\ref{sec:nopartorient}, Example~\ref{exam:1}, and Theorem~\ref{thm:betti_coincide}). As a corollary, it follows that the Betti table of all these ideals are independent of the base field $K$. 

\medskip

For complete graphs, a minimal polyhedral cellular free resolutions for $\MM_G^q$ and $\II_G$ was given in \cite{PostnikovShapiro04} and \cite{MadhuBernd}, respectively. The case of general graphs was left open in both works. Our work generalizes these constructions to arbitrary graphs, puts their constructions into a larger context, and resolves several questions and conjectures from these papers. We should mention that minimal free resolutions and the Betti numbers for both $\MM_G^q$ and $\II_G$ were first established in \cite{FarbodFatemeh} and independently in \cite{Madhu}. The first Betti number for $I_G$ was computed in \cite{horia}. A minimal {\em cellular} resolution for $\MM_G^q$ was given in \cite{Anton}. The Betti numbers for $\MM_G^q$ was also computed in \cite{hopkins}.

\medskip

We also remark that it is possible to directly give a minimal polyhedral cellular free resolution for the ideal $\MM_G^q$ by our potential theoretic techniques in \S\ref{sec:IGresol}, but we have chosen to skip the details of this construction here as all the main ideas appear elsewhere in this writing (see Remark~\ref{rmk:Mg}). Moreover, an essentially equivalent (see Remark~\ref{rmk:cells}(ii)) solution for $\MM_G^q$ has recently (and independently) appeared in \cite{Anton}, where they leave the solution for $\II_G$ as an open problem.

\medskip

Our techniques allow us to revisit some of the foundational results on {\em chip-firing} games and related fields. For example, we remark that our potential theoretic interpretation of Gr\"obner weights relating $\II_G$ to $\MM_G^q$ gives a new proof of the result in \cite{FarbodMatt12} interpreting $q$-reduced divisors as divisors of {\em minimum total potential} (see Remark~\ref{rmk:reducedpotential}). A related problem is to describe the whole Gr\"obner cone of the initial ideal $\MM_G^q$. This was a question of Bernd Sturmfels which we completely answer in \S\ref{sec:grocone}. We show that the rays of the Gr\"obner cone associated to $\MM_G^q$ correspond, in a precise sense, to Green's functions. 

\medskip

The equality of the Betti tables of all of our ideals allows one to prove many numerical facts about one ideal by looking instead at another ideal in this family. We consider a few of such examples in \S\ref{sec:conseq}. One example is the computation of multiplicities. Another example of this observation is that we can reprove some results expressing the $h$-vectors of $\II_G$ and $\MM_G^q$ in terms of the Tutte polynomial. These results were originally proved by Merino in \cite{Merino} and by Postnikov and Shapiro in \cite{PostnikovShapiro04} using direct combinatorial methods. In our approach, we show that there is a fifth ideal $\Mat_G$, directly related to the cographic matroid of $G$, with the same Betti table. This observation gives a direct and conceptual proof of the connection with the Tutte polynomial.  
Furthermore, the Hilbert function of $\MM_G^q$ is applied in system reliability theory \cite{Fatemehreliability}, percolation on trees \cite{mohammadi2015algebraic} and signature analysis of networks \cite{mohammadi2015types}.


\section{Notation and background}
\label{sec:Background}

Throughout, we assume $\NN$ contains zero. All rings are commutative with $1$. 

\medskip

A {\em graph} means a finite, connected, unweighted multigraph with no loops. As usual, the set of vertices and edges of a graph $G$ are denoted by $V(G)$ and $E(G)$. For $A \subseteq V(G)$, we denote by $A^c$ the complement of $A$ in $V(G)$. We set $n=|V(G)|$ and $m=|E(G)|$. For a set of vertices $S$, the induced subgraph of $G$ with the vertex set $S$ is denoted by $G[S]$.

\medskip

Let $\EE(G)$ denote the set of oriented edges of $G$; for each edge in $E(G)$ there are two edges $e$ and $\bar{e}$ in $\EE(G)$. So we have $|\EE(G)|=2m$. An element $e$ of $\EE(G)$ is called an {\em oriented} edge, and $\bar{e}$ is called the {\em inverse} of $e$. We have a map 
\[
\begin{aligned}
\EE(G) &\rightarrow V(G) \times V(G) \\
e &\mapsto (e_{+}, e_{-})
\end{aligned}
\]
sending an oriented edge $e$ to its head (or its terminal vertex) $e_{+}$ and its tail (or its initial vertex) $e_{-}$. Note that $\bar{e}_{+}=e_{-}$ and $\bar{e}_{-}=e_{+}$. Given disjoint nonempty subsets $A,B$ of $V(G)$ we define 
\[
\EE(A,B) = \{e \in \EE(G): e_+ \in A ,e_- \in B\} \ .
\]

\medskip


An {\em orientation} of $G$ is a choice of subset $\Oc \subset \EE(G)$ such that $\EE(G)$ is the disjoint union of $\Oc$ and $\bar{\Oc}=\{\bar{e}: \ e \in \Oc \}$. An orientation is called {\em acyclic} if it contains no directed cycle.
A {\em partial orientation} of $G$ is a choice of subset $\P \subset \EE(G)$ that strictly contains an orientation $\Oc$ of $G$. For a partial orientation $\P$, the associated (connected) {\em partition} is the partition of $G$ into totally cyclic subgraphs with edges $\{e , \bar{e} \in \P\}$. A partial orientation is called {\em acyclic} if the induced orientation on the graph obtained by contracting all its totally cyclic components is acyclic.

\medskip

Let $\Oc$ be an orientation of $G$. A vertex $q$ is called a {\em source} for $\Oc$ if $q=e_-$ for every $e \in \Oc$ which is incident to $q$.
Let $\P$ be a partial orientation of $G$. Let $H$ be the associated connected component containing the vertex $q$. Then $q$ is called a {\em source} for $\P$ if $H$ corresponds to a source in the graph obtained by contracting all components of $\P$ (see Example~\ref{exam:1}).

\medskip

For an abelian group $A$, we let $C^0(G,A)$ denote the set of all $A$-valued functions on $V(G)$. It is endowed with the bilinear form 
\[
\langle f_1, f_2\rangle = \sum_{v\in V(G)}{f_1(v)f_2(v)} \ .
\]

Also, $C^1(G, A)$ will denote the space of all $A$-valued functions $g$ on $\EE(G)$ such that $g(\bar{e})=-g(e)$ for all $e \in \EE(G)$. After fixing an orientation $\Oc \subset \EE(G)$ we have $C^1(G, A) = C_{\Oc}^1(G, A) \oplus C_{\bar{\Oc}}^1(G, A)$, where $C_{\Oc}^1(G, A)$ denotes the space of all $A$-valued functions on $\Oc$. 
The group $C^1(G, A)$ (and therefore $C_{\Oc}^1(G, A)$) is endowed with the bilinear form 
\begin{equation}
\langle g_1, g_2 \rangle = \sum_{e \in \Oc}{g_1(e)g_2(e)} = \frac{1}{2} \sum_{e \in \EE(G)}{g_1(e)g_2(e)}
\end{equation}

The usual coboundary map $d \colon C^0(G,A) \rightarrow C^1(G,A)$ is defined by
\[
(d f) (e) = f(e_+) - f(e_-) = -(d f) (\bar{e}) \ .
\]

After fixing an orientation $\Oc \subset \EE(G)$, we obtain the restricted coboundary map $d_\Oc \colon C^0(G,A) \to C_{\Oc}^1(G,A)$.

\medskip

Let $R$ be a commutative ring with $1$. We let $C_0(G, R)$ denote the free $R$-module generated by $V(G)$. Elements of $C_0(G, R)$ are of the form $\sum_{v \in V(G)} {a_v (v)}$ for $a_v \in R$. It is endowed with a bilinear form induced by $\langle (u) , (v)\rangle =\delta_{v}(u)$ for $u,v \in V(G)$. Here $\delta_{v}(u)$ denotes the usual Kronecker delta function. 

\medskip

Likewise, we let $C_1(G, R)$ denote the free $R$-module generated by $\EE(G)$. Elements of $C_1(G, R)$ are of the form $\sum_{e \in \EE(G)} {a_e (e)}$ for $a_e \in R$. It is endowed with a bilinear form induced by 
\[
\langle (e) , (e')\rangle =
\begin{cases}
1, &\text{if $e'=e$}\\
-1, &\text{if $e'=\bar{e}$}\\
0, &\text{otherwise}
\end{cases}
\]
for $e,e' \in \EE(G)$. The usual boundary map $\partial\colon C_1(G, R) \rightarrow C_0(G, R)$ is defined by 
\[
\partial(e)=(e_+)-(e_-) \ . 
\]

The bilinear forms defined above provide canonical isomorphisms $C_0(G,R) \cong C^0(G,R)$ and $C_1(G,R) \cong C^1(G,R)$. Then the maps $\partial$ and $d$ are adjoint with respect to these bilinear forms. We let $e^\ast \in C^1(G,R)$ denote the image of $(e) \in C_1(G, R)$ under this isomorphism, i.e.
\[
e^\ast := \langle (e) , \cdot \rangle \ .  
\]
The characteristic function of $v$ or $\chi_v= \delta_{v} \in C^0(G, R)$ is the image of $(v) \in C_0(G,R)$ under the canonical isomorphism.

\medskip

\medskip

Let $K$ be a field. Associated to $G$ we define two polynomial rings:
\begin{itemize}
\item Let $\Rb=K[\xb]$ denote the polynomial ring in $n$ variables $\{x_v: v \in V(G)\}$.
\item Let $\Sb=K[\yb]$ denote the polynomial ring in $2m$ variables $\{y_e: e \in \EE(G)\}$ or $\{y_e, y_{\bar{e}}: e \in \Oc\}$ (for any orientation $\Oc$). 
\end{itemize}

\section{Divisors and potential theory on graphs}
\label{sec:divisors}

Following \cite{BN1}, we let $\Div(G)$ be the free abelian group generated by $V(G)$. Equivalently, $\Div(G) = C_0(G, \ZZ)$. An element  of $\Div(G)$ is written as $\sum_{v \in V(G)} a_v (v)$ for $a_v \in \ZZ$ and is called a {\em divisor} on $G$. The coefficient $a_v$ in $D$ is denoted by $D(v)$.  A divisor $D$ is called {\em effective} if $D(v) \geq 0$ for all $v\in V(G)$. The set of effective divisors is denoted by $\Div_{+}(G)$. We write $D \leq E$ if $E-D \in \Div_{+}(G)$. For $D \in \Div(G)$, let $\deg(D) = \sum_{v \in V(G)} D(v)$. Given disjoint nonempty subsets $A,B$ of $V(G)$ one can assign a divisor $D(A,B)= \sum_{v \in A}  |\{ w \in B: \{v,w\} \in E(G)\}| \ (v)$.

\medskip

We denote by $\M(G)$ the group of integer-valued functions on the vertices. Equivalently, $\M(G)=C^0(G, \ZZ)$. For $A \subseteq V(G)$, $\chi_A \in \M(G)$ denotes the $\{0,1\}$-valued characteristic function of $A$. The {\em Laplacian operator} $\Delta \colon \M(G) \to \Div(G)$ is defined by

\[\Delta(f) = \sum_{v \in V(G)} \sum_{\{v,w\} \in E(G)} (f(v) - f(w)) (v) \  .\]

\begin{Remark} \label{rmk:selfadjoint}
With the identification $\M(G)=C^0(G, \ZZ)$ and $\Div(G)=C_0(G, \ZZ)$ and the canonical isomorphism $C_1(G,R) \cong C^1(G,R)$, the operator $\Delta$ is identified with $\partial_\Oc d_\Oc \colon C^0(G, \ZZ) \rightarrow C_0(G,\ZZ)$, where $\partial_\Oc$ and $d_\Oc$ denote the usual (restricted) boundary and coboundary maps for an arbitrary orientation $\Oc$. Somewhat more canonically, $\Delta = \frac{1}{2}\partial d$. It follows that $\Delta$ is a self-adjoint operator, which means 
\[\sum_{v}{f(v) \Delta(g)(v)}=\sum_{v}{g(v) \Delta(f)(v)}\] for all $f, g \in \M(G)$.
\end{Remark}

\medskip

The group of {\em principal divisors} is defined as the image of the Laplacian operator and is denoted by $\Prin(G)$. It is easy to check that $\Prin(G) \subseteq \Div^0(G)$ where $\Div^0(G)$ denotes the set consisting of divisors of degree
zero. The quotient
$\Pic^0(G) = \Div^0(G) / \Prin(G)$
is a finite group whose cardinality is the number of spanning trees of $G$ (see, e.g., \cite{FarbodMatt12} and references therein). The full {\em Picard group} of $G$ is defined as
\[
\Pic(G) = \Div(G) / \Prin(G)
\]
which is isomorphic to $\ZZ \oplus \Pic^0(G)$. Since $\Prin(G) \subseteq \Div^0(G)$, the map $\deg \colon \Div(G) \rightarrow \ZZ$ descends to a well-defined map $\deg \colon \Pic(G) \rightarrow \ZZ$. Two divisors $D_1$ and $D_2$ are called {\em linearly equivalent} if they become equal in  $\Pic(G)$. In this case we write $D_1 \sim D_2$.

\subsection{Divisors and potential theory} \label{sec:potential}

For $p,q \in V(G)$ let the {\em Green's function} $j_q(p,\cdot)$ denote the unique ($\QQ$-valued) solution to the Laplace equation $\Delta f = (p)-(q)$ satisfying $f(q)=0$. If we think of graph $G$ as an electrical network (in which each edge is a resistor having unit resistance) then $j_q(p,v)$ denotes the electric potential at $v$ if one unit of current enters the network at $p$ and exits at $q$, with $q$ grounded (i.e., zero potential). It is easy to check that $j_q(p,q)=0$, $j_q(p,v)=j_q(v,p)$, and $0 \leq j_q(p,v) \leq j_q(p,p)$ (see \cite{ChinburgRumely,BX06}). \cite[Construction~3.1]{FarbodMatt12} explains how to compute these functions using basic linear algebra. For all $f \in \M(G)$ we have 
\begin{equation}\label{eq:useful}
\begin{aligned}
\sum_{v}{j_q(p,v)\Delta(f)(v)} &= \sum_{v}{f(v)\Delta(j_q(p, \cdot))(v)} = 
f(p)-f(q) \ .
\end{aligned}
\end{equation}

\medskip

There exists a positive definite, symmetric bilinear form 
\[
\langle  \cdot \, , \cdot \rangle_{\en} \colon \; \Div^0(G) \times \Div^0(G) \rightarrow \QQ
\] 
\[
\langle D_1 , D_2 \rangle_{\en} = \sum_{u,v \in V(G)}{D_1(u)j_q(u,v)D_2(v)}
\]
which is a canonical (i.e. independent of the choice of $q$) pairing on $\Div^0(G)$ (see \cite{FarbodMonodromy,FarbodMatt12}). It is called the {\em energy pairing} on $\Div^0(G)$.

\medskip

Let $\mathbf{1}$ denote the all-$1$'s divisor. For $D \in \Div(G)$ and $q\in V(G)$, following \cite{FarbodMatt12}, the {\em total potential functional} is defined as 
\[
b_q(D)=\langle \mathbf{1} - n (q), D - \deg(D)(q)\rangle_{\en} = \sum_{v}\sum_{p}{j_q(p,v)D(v)} \ . 
\]

\subsection{Divisors and commutative algebra}
\label{sec:divcomm}
Any effective divisor $D$ gives rise to a monomial
\[
\xb^D := \prod_{v \in V(G)}{x_{v}^{D(v)}} \in \Rb\ .
\]

Associated to every graph $G$ there is a canonical ideal in $\Rb$ which encodes the linear equivalences of divisors on $G$:
\[
\begin{aligned}
\II_G &:= \langle \xb^{D_1} - \xb^{D_2} : \, D_1 \sim D_2 \text{ both effective divisors}\rangle \\
&= \span_K \{ \xb^{D_1} - \xb^{D_2} : \, D_1 \sim D_2 \text{ both effective divisors}\}
\end{aligned}
\]
which was first introduced in \cite{CoriRossinSalvy02}. This ideal is graded by both $\Pic(G)$ and $\ZZ$. 

\begin{Remark}
It is shown in \cite{FarbodFatemeh} that, although $\Pic(G)$ has torsion elements, it provides a ``nice'' grading in the sense that Nakayama's lemma holds with respect to this grading and the concept of $\Pic(G)$-graded minimal free resolution makes sense in this context.
\end{Remark}

\medskip

Once we fix a vertex $q$, there is a natural term order that gives rise to a particularly nice  Gr\"obner basis for $\II_G$. This term order was also introduced in \cite{CoriRossinSalvy02}. Consider a total ordering of the set of variables $\{x_v: v \in V(G)\}$ compatible with the distances of vertices from $q$ in $G$:
\begin{equation}\label{eq:dist}
\dist(w,q) < \dist(v,q) \, \implies \,  x_w < x_v   \ .
\end{equation}
Here, the distance between two vertices in a graph is the number of edges in a shortest path connecting them. This ordering can be thought of as an ordering on the vertices induced by running the breadth-first search (BFS) algorithm starting at the root vertex $q$.
The term order $<_q$ will denote the graded reverse lexicographic ordering (grevlex) on $\Rb$ induced by the total ordering on the variables given in \eqref{eq:dist}.

\medskip

The initial ideal $\MM_G^q:=\ini_{<_q}(\II_G)$ for $(\II_G, <_q)$ is canonically defined (up to the choice of the distinguished vertex $q$). This ideal is extensively studied in \cite{PostnikovShapiro04}, where it is denoted by $M_G$. This ideal is naturally equipped with $\Div(G)$ (fine) and $\ZZ$ (coarse) gradings.

One of the main results of \cite{CoriRossinSalvy02} is the following theorem -- see also \cite[Section~5]{FarbodFatemeh} where this result is reproved and generalized to higher syzygy modules.

\begin{Theorem}
\label{thm:Cori}
A Gr\"obner basis of $(\II_G, <_q)$ is
\[
\left\{\xb^{D(A^c, A)}-\xb^{D(A, A^c)} : A \subsetneq V(G) , q\in A \right\}  \ .
\]

Moreover,
\begin{itemize}
\item[(i)] $\LM(\xb^{D(A^c, A)}-\xb^{D(A, A^c)})=\xb^{D(A^c, A)}$.
\item[(ii)] It suffices to consider only those subsets $A$ of $V(G)$ such that both $G[A]$ and $G[A^c]$ are connected. In this case we obtain a {\em minimal Gr\"obner basis} of $(\II_G, <_q)$.  
\end{itemize}
\end{Theorem}

As we will see, the minimal Gr\"obner basis described in part (ii) is also a minimal generating set (see also \cite{FarbodFatemeh}).

\medskip

\subsection{Potential theory and Gr\"obner weight functionals for $\II_G$}\label{sec:wt1}
Let $\vartheta \in C^0(G, \RR)$ and think of it as a linear functional $\vartheta \colon \Div(G) \rightarrow \RR$. For $f =\sum{c_i \xb^{D_i}} \in \Rb$ the $\vartheta$-degree of $f$, denoted by $\deg_\vartheta(f)$, is the maximum value of $\vartheta(D_i)$. The $\vartheta$-initial form of $f$ is the sum of all terms $c_i \xb^{D_i}$ such that $\vartheta(D_i)$ is maximum. For an ideal $I \subset \Rb$, the $\vartheta$-initial ideal $\ini_\vartheta(I)$ is the ideal generated by all $\vartheta$-initial forms.

\medskip

Fix a term order $<$ for $\Rb$. The functional $\vartheta$ is said to {\em represent $<$ for $I$} if $\ini_\vartheta(I)=\ini_{<}(I)$. It is known that for any term order $<$ and any ideal $I$, there is a {\em non-negative} and {\em integer-valued} functional representing $<$ for $I$ (\cite[Proposition~1.11]{SturmfelsGrobnerConvex}). 

\medskip

In our situation there is a nice and direct interaction between Gr\"obner theory and potential theory.
\begin{Lemma} \label{lem:bq}
 $b_q\colon \Div(G) \rightarrow \QQ$ is a non-negative rational-valued functional representing $<_q$ for $\II_G$. 
\end{Lemma}
The proof is a straightforward variation of the discussion in \S~\ref{sec:grocone}, and is left for the reader.

\begin{Remark} \label{rmk:reducedpotential}
$q$-reduced divisors (or $G$-parking functions with respect to $q$) can be defined as the normal forms of $\Rb / \II_G$ with respect to the Gr\"obner basis described in Theorem~\ref{thm:Cori}. It easily follows from Lemma~\ref{lem:bq} that a $q$-reduced divisor is precisely the unique (in each equivalence class) minimizer of the $b_q$ functional. See \cite{FarbodMatt12} for a precise statement and a different proof of this fact.
\end{Remark}
\medskip

\begin{Definition}\label{def:M_intwt} 
We let $\vartheta_q$ denote the non-negative, integral functional associated to $b_q$ (i.e. obtained from $b_q$ by clearing the denominators). Clearly, $\vartheta_q$ will also represent $<_q$ for $\II_G$. 
\end{Definition}

\subsection{Gr\"obner cone of $\MM_G^q$} \label{sec:grocone}
It is possible to give a more precise statement than Lemma~\ref{lem:bq}. We show that the rays of the Gr\"obner cone associated to $\MM_G^q$, in a precise sense, correspond to Green's functions.

The weight functional $\eta \in C^0(G, \RR)$ defined by $\eta(D)=\sum_{v \in V(G)}{\eta(v) (v)}$ is in the Gr\"obner cone if and only if for any set $B \ne \emptyset$ with $q\not\in B$ we have 
\begin{equation}\label{eq:wt1char}
\eta(\Delta(\chi_B)) = \sum_{v \in V(G)}{\eta(v) \Delta(\chi_B)(v)} = \sum_{v \in V(G)}{\chi_B(v) \Delta(\eta)(v)} >0 \ . 
\end{equation}
In particular, for each vertex $p \ne q$, setting $B=\{p\}$ we must have:
\begin{equation} \label{eq:MConeCond}
\gamma_p : =\Delta(\eta)(p) >0 \ . 
\end{equation}
This condition is also sufficient because for all $B \ne \emptyset$ with $q\not\in B$ we have 
\[
\eta(\Delta(\chi_B))=\sum_{v \in V(G)}{\chi_B(v)\gamma_v}=\sum_{v \in B}{\gamma_v} \ .
\]

\medskip

It follows that $\eta \in \M(G)$ is a solution to $\Delta(\eta) = \gamma$ for the degree zero divisor $\gamma := \sum_{p \in V(G)} \gamma_p (p)$.  From the definition of the Green's function $j_q(p, v)$, and the fact that the Laplacian operator has a $1$-dimensional zero-eigenspace generated by the all-$1$ function $\mathbf{1}$, we obtain:
\begin{equation}\label{eq:wt1}
\eta=\sum_{p \in V(G)}\gamma_p j_q(p, \cdot ) + k \cdot \mathbf{1}
\end{equation}
for some constant $k \in \RR$. We summarize these observations in the following theorem.

\begin{Theorem}
The weight functional $\eta \in C^0(G, \RR)$ represents $<_q$ for $\II_G$ if and only if there exist $k \in \RR$ and real numbers $\gamma_p >0$ (for $p \in V(G)$) such that 
\[
\eta=\sum_{p \in V(G)}\gamma_p j_q(p, \cdot ) + k \cdot \mathbf{1} \ .
\]
\end{Theorem}

\medskip

In other words $\eta$, up to constant functions, is in the interior of the cone generated by the vectors $(j_q(p,v))_{v \in V(G)}$ for various $p \in V(G)$. Note that these vectors are independent because the matrix $(j_q(p,v))_{p,v \in V(G) \backslash \{q\}}$ is invertible (see \cite[Construction~3.1]{FarbodMatt12}). 
The question of describing this Gr\"obner cone was asked by Bernd Sturmfels.

\medskip


\section{Lattices, Delaunay decompositions, total unimodularity, and infinite arrangements} 
\label{sec:latticesec}
\subsection{Lattices and Delaunay decompositions}
\label{sec:lattice}

Let $\Lambda$ be a free $\ZZ$-module (abelian group), endowed with a positive definite symmetric bilinear pairing $\beta \colon \Lambda \times \Lambda \rightarrow \ZZ$.
The pair $(\Lambda, \beta)$ (or just $\Lambda$, when $\beta$ is understood) is called a {\em free bilinear form space over $\ZZ$} or, more concisely, an {\em abstract $\ZZ$-lattice}.

\medskip

Let $(\Lambda, \beta)$ be an abstract $\ZZ$-lattice. We let $\Lambda_{\RR} := \Lambda \otimes \RR$. The bilinear pairing $\beta$ naturally extends to a bilinear pairing $\beta_\RR$ on $\Lambda_\RR$ by $\beta_\RR(a \otimes \ub,b \otimes \vb )=ab \ \beta(\ub, \vb)$.
  
The dual $\ZZ$-module $\Lambda^{\vee} := \Hom_{\ZZ}(\Lambda , \ZZ)$ is contained (via extension of scalars) in the dual real vector space $\Lambda_{\RR}^{\vee} := \Hom_{\ZZ}(\Lambda , \RR) = \Hom_{\RR}(\Lambda_\RR, \RR)= \Lambda^\vee \otimes \RR$. The {\em non-degeneracy} of $\beta$ is the statement that the homomorphism
\[
\begin{aligned}
\Psi\colon \Lambda &\rightarrow \Lambda^\vee \\
\vb &\mapsto \beta(\vb , \cdot)
\end{aligned}
\]
is injective. Clearly every positive definite bilinear pairing is automatically non-degenerate. Therefore the natural extension $\Psi_\RR \colon \Lambda_\RR \rightarrow \Lambda_\RR^\vee$ is also injective (e.g., because $\RR$ is a flat $\RR$-module). Since these vector spaces have the same dimension, it follows that $\Psi_\RR$ is indeed an isomorphism. In other words, in the language of bilinear forms, $\beta_\RR$ is a {\em perfect pairing}\footnote{A perfect pairing is sometimes called a {\em unimodular pairing} in the literature. We will avoid this terminology to avoid confusion.} on $\Lambda_\RR$. So, in this situation, any $\varphi \in \Lambda_\RR^\vee$ is of the form $\varphi (\cdot) = \beta_\RR (\ab , \cdot)$ for some $\ab \in \Lambda_\RR$. 

\medskip

Let $d \colon \Lambda_\RR \times \Lambda_\RR \rightarrow \RR$ be any distance function on $\Lambda_\RR$. The {\em Delaunay decomposition} of $\Lambda_\RR$ with respect to the lattice $\Lambda$ and the distance function $d$ (not necessarily induced by the bilinear form) is defined as the collection of cells
\[
A_{\pb} = {\rm conv.hull} \{ \sb \in \Lambda: d(\pb,\sb) \text{ is minimal}\} 
\]
as $\pb$ varies in $\Lambda_\RR$. It is a classical fact (essentially due to Voronoi and Delaunay) that the collection of Delaunay cells $\{A_{\pb}\}$ gives a locally finite, cellular decomposition (face to face tiling) of $\Lambda_\RR$ which is invariant under the action of $\Lambda$ (see, e.g., \cite{Conway}).

\medskip

\subsection{Total unimodularity}
Consider a (not necessarily minimal) finite set $\{\varphi_i\}_{i\in I}$ of generators for the free $\ZZ$-module $\Lambda^\vee$. Extension of scalars gives an inclusion $\Lambda^\vee \hookrightarrow \Lambda_\RR^\vee$. Clearly, for any subset $J \subseteq I$ such that $\{\varphi_i\}_{i\in J}$ generates $\Lambda^\vee$ as a $\ZZ$-module, we have $\{\varphi_i \}_{i\in J}$ spans $\Lambda_\RR^\vee$ as a real vector space (here we have identified $\varphi_i \otimes 1$ with $\varphi_i$). The converse is, of course, not true in general.

\medskip

\begin{Definition}
Let $(\Lambda, \beta)$ be an abstract $\ZZ$-lattice. A finite set $\{\varphi_i\}_{i\in I}$ of generators for $\Lambda^\vee$ is called {\em totally unimodular} if for any subset $J \subseteq I$ such that the collection $\{\varphi_i\}_{i\in J}$ spans $\Lambda_\RR^\vee$ as a real vector space,  the collection $\{\varphi_i\}_{i\in J}$ generates $\Lambda^\vee$ as a $\ZZ$-module.
\end{Definition}

\begin{Example}
Let $\Lambda=\ZZ^2$, generated by $\eb_1$ and $\eb_2$, endowed with the obvious bilinear pairing induced by $\langle \eb_i, \eb_j \rangle = \delta_{i}(j)$. Let $\eb_i^\ast \in (\ZZ^2)^\vee$ denote the dual basis element $\eb_i^\ast(\eb_j)=\delta_{i}(j)$. Then 
\begin{itemize}
\item $\{\eb_1^\ast, \eb_2^\ast, \eb_1^\ast+\eb_2^\ast\}$ generates $\Lambda^\vee$ and is totally unimodular.
\item $\{\eb_1^\ast, \eb_2^\ast, \eb_1^\ast+2\eb_2^\ast\}$ generates $\Lambda^\vee$ but is not totally unimodular. 

The subcollection $\{\eb_1^\ast, \eb_1^\ast+2\eb_2^\ast\}$ spans $(\RR^2)^\vee$ as a real vector space, but it does not generate $(\ZZ^2)^\vee$. For example, $\eb_2^\ast$ will not be in the $\ZZ$-module it generates.
\end{itemize}
\end{Example}
\begin{Example}\label{ex:WU}
The primary examples of total unimodularity and the most well-known examples arise from totally unimodular matrices or, more generally, weakly unimodular matrices. An $r \times m$ ($r \leq m$) integer matrix $A=(a_{ij})$ is called {\em weakly unimodular} if every $r \times r$ square submatrix of $A$ has determinant in the set $\{-1, 0 , 1\}$. If every square submatrix of $A$ has determinant in the set $\{-1, 0 , 1\}$, then $A$ is called a {\em totally unimodular} matrix. Any totally unimodular matrix is weakly unimodular. A weakly unimodular matrix which contains the identity matrix of size $r$ is automatically totally unimodular.

\medskip

Let $A$ be a weakly unimodular matrix. Let $\Lambda$ denote the row space $\Image(A^T) \hookrightarrow \ZZ^m$ with the bilinear pairing induced by the natural bilinear pairing on $\ZZ^m$. For $1 \leq j \leq m$ let $\varphi_j \in \Lambda^\vee$ denote the restriction of $\eb_j^\ast \in (\ZZ^m)^\vee$ to $\Lambda$. Concretely, if we denote the $i$-th row ($1 \leq i \leq r$) of $A$ by $\vb_i$, then each $\varphi_j$ is defined by $\varphi_j(\vb_i)=a_{ij}$. By Cramer's rule, the collection $\{\varphi_1, \ldots, \varphi_m\}$ is totally unimodular precisely because $A$ is weakly unimodular.
\end{Example}

\subsection{Infinite hyperplane arrangements}
\label{sec:infarrag}

Consider a finite collection $\{\varphi_i\}_{i\in I} \subset \Lambda_\RR^\vee$ spanning $\Lambda_\RR^\vee$ as a vector space over $\RR$. For each $\pb \in \Lambda_\RR$ we denote by $C_\pb$ the polyhedron in $\Lambda_\RR$ defined by 
\[
C_\pb = \{\sb \in \Lambda_\RR: \lfloor \varphi_i(\pb) \rfloor  \leq \varphi_i(\sb) \leq \lceil \varphi_i(\pb) \rceil  \text{ for all } i \in I \} \ .
\]
As usual, $\lfloor x \rfloor$ denotes the largest integer $n \leq x$, and  $\lceil x \rceil$ denotes the smallest integer $n \geq x$. Clearly $C_\sb=C_\pb$ for all $\sb \in \relint(C_\pb)$. We denote by $\H(\Lambda_\RR, \{\varphi_i\}_{i\in I})$ the collection of all polyhedra $C_\pb$ for $\pb \in \Lambda_\RR$.
 
\medskip

The following result is well known for the case of totally unimodular matrices (Example~\ref{ex:WU}) (see, e.g., \cite{OdaSeshadri, Erdahl}). It can be proved similar to \cite[Corollary 3.2]{OdaSeshadri}.

\begin{Theorem}\label{thm:totunim}
Fix a finite collection $\{\varphi_i\}_{i\in I} \subset \Lambda_\RR^\vee$ which spans $\Lambda_\RR^\vee$ as a vector space over $\RR$.
\begin{itemize}
\item[(i)] $\H(\Lambda_\RR, \{\varphi_i\}_{i\in I})$ is a polyhedral cell decomposition of $\Lambda_\RR$ by bounded convex polyhedra. This cell decomposition is invariant under the translation by 
\[\{\sb \in \Lambda_\RR:  \varphi_i(\sb)  \in \ZZ \text{ for all } i \in I \}\]
 which is contained in the set of $0$-dimensional polyhedra in $\H(\Lambda_\RR, \{\varphi_i\}_{i\in I})$.
\item[(ii)] If, further, $\{\varphi_i\}_{i\in I} \subset \Lambda^\vee$ and it generates $\Lambda^\vee$, then $\H(\Lambda_\RR, \{\varphi_i\}_{i\in I})$ is invariant under the translation action by elements of $\Lambda$ which is contained in the set of $0$-dimensional polyhedra in $\H(\Lambda_\RR, \{\varphi_i\}_{i\in I})$.
\item[(iii)] If, further, $\{\varphi_i\}_{i\in I}$ is totally unimodular, then $\Lambda$ coincides with the set of $0$-dimensional polyhedra in $\H(\Lambda_\RR, \{\varphi_i\}_{i\in I})$. Moreover, $\H(\Lambda_\RR, \{\varphi_i\}_{i\in I})$ coincides with the Delaunay decomposition of $\Lambda_\RR$ with respect to the lattice $\Lambda$ and the metric induced by
\begin{equation} \label{eq:norm}
\lVert \pb\rVert^2 = \sum_{i \in I}{\lvert \varphi_i(\pb) \rvert^2} \ .
\end{equation}
\end{itemize}
\end{Theorem}

\begin{Remark} \label{rmk:delon}
\begin{itemize}
\item[]
\item[(i)] Under the total unimodularity assumption, by Theorem~\ref{thm:totunim}(iii), we obtain a finite polyhedral cell decomposition of the quotient torus $\Lambda_\RR / \Lambda$. This cell decomposition is essential in the study of our binomial ideals.
\item[(ii)] If the totally unimodular collection is coming from a weakly unimodular matrix as in Example~\ref{ex:WU}, then the norm in \eqref{eq:norm} coincides with the standard norm induced by the bilinear form $\beta_\RR$. This is because the $\varphi_j$'s are precisely the restriction of the $\eb_j^\ast$'s to $\Lambda_\RR$.
\end{itemize}
\end{Remark}

\section{Potential theory and the cellular free resolution of $\II_G$}
\label{sec:IGresol}
Here we use potential theory and the energy pairing to give a self-contained and direct solution to the problem of finding a minimal polyhedral cellular free resolution of the ideal $\II_G$.

\subsection{Principal lattice with the energy pairing}

Recall the $\ZZ$-module $\Prin(G)$ is defined as the image of the Laplacian operator $\Delta \colon \M(G) \to \Div(G)$. We have introduced two different canonical bilinear forms on this group. One is the bilinear form induced from the bilinear form on $C_0(G,\ZZ)=\Div(G)$ defined in \S\ref{sec:Background}. The bilinear form that is most relevant in this section is the one induced from the energy pairing defined in \S\ref{sec:potential}. 

\begin{Definition}
By a {\em principal lattice} we will mean the pair $(\Prin(G), \langle \cdot , \cdot \rangle_{\en})$ where
\[
\langle \cdot , \cdot \rangle_{\en} \colon \Prin(G) \times \Prin(G) \rightarrow \ZZ
\]
is the restriction of the energy pairing to $\Prin(G) \subseteq \Div^0(G)$. 
\end{Definition}

\begin{Remark}
It is easy to see (using \eqref{eq:useful}) that if $D \in \Prin(G)$ then for all $E \in \Div^0(G)$ we have $\langle E , D \rangle_{\en} \in \ZZ$ and therefore 
\begin{itemize}
\item[(i)] The restriction of the energy pairing to  $\Prin(G)$ is $\ZZ$-valued.
\item[(ii)] The energy pairing descends to a well-defined pairing on $\Pic^0(G)$, which is shown to be non-degenerate in \cite{FarbodMonodromy}. 
\end{itemize}
\end{Remark}
The principal lattice is an abstract $\ZZ$-lattice in the sense of \S\ref{sec:lattice}. Its ambient vector space $\Prin(G)_\RR = \Prin(G) \otimes \RR$ coincides with $\Div_{\RR}^0(G) = \Div^0(G) \otimes \RR \subset C_1(G,\RR)$. 

Our next goal is to find a nice collection of functionals for this lattice. For each $e \in \EE(G)$ we define the functional $\zeta_e \in \Div_{\RR}^0(G)^{\vee}$ by
\[
\zeta_e(\cdot) = \langle \partial(e) , \cdot \rangle_{\en} \ .
\]

The following lemma follows from an easy computation.

\begin{Lemma} \label{lem:lookgraphical}
\begin{itemize}
\item[]
\item[(i)] Any $D \in \Div_{\RR}^0(G)$ is of the form $D=\Delta(f)$ for some $f \in C^1(G, \RR)$. 
\item[(ii)] For $D =  \Delta(f) \in \Div_{\RR}^0(G)$ we have $\zeta_e(D)=(d f)(e)$. 
\end{itemize}
\end{Lemma}

\begin{Proposition}\label{prop:rightfunctionals}
\begin{itemize}
\item[]
\item[(i)] $\{\zeta_e\}_{e \in \EE(G)} \subset \Prin(G)^\vee$.
\item[(ii)] $\{\zeta_e\}_{e \in \EE(G)}$ generates $\Prin(G)^\vee$.
\item[(iii)] $\{\zeta_e\}_{e \in \EE(G)}$ is totally unimodular for the principal lattice.
\end{itemize}
\end{Proposition}
\begin{proof}
(i) Let $D = \Delta(f)$ for $f \in \M(G)$. 
By Lemma~\ref{lem:lookgraphical}(ii) $\zeta_e(D) = (d f)(e)$
which is an integer because $f$ is integer-valued.

\medskip

(ii) Let $\zeta$ be an arbitrary element of $\Prin(G)^\vee$. We need to show that $\zeta=\sum_{e \in \EE(G)}{a_e\zeta_e}$ for some integers $a_e$. Since $\zeta \in \Div_{\RR}^0(G)^\vee$ and $\langle \cdot , \cdot \rangle_{\en}$ is positive definite (and therefore non-degenerate), we must have  $\zeta(\cdot) = \langle \ab , \cdot \rangle_{\en}$ for some $\ab \in \Div_{\RR}^0(G)$ (see \S\ref{sec:lattice}). For all $ p \in V(G) \backslash \{q\}$ we have (see \eqref{eq:useful})
\begin{equation} \label{eq:integervalues}
\langle \ab , \Delta(\chi_p) \rangle_{\en} 
= \ab(p) \ .
\end{equation}
Since $\zeta \in \Prin(G)^\vee$ we must have $\ab(p) = \langle \ab , \Delta(\chi_p) \rangle_{\en} \in \ZZ$ for all $p \in V(G) \backslash \{q\}$. Since $\ab(q)=-\sum_{p \ne q}{\ab(p)}$ we obtain $\ab \in \Div^0(G)$. Let 
\begin{equation}\label{eq:telescope}
\ab = \sum_{p \in V(G)}{\ab(p)(p)}=\sum_{p \ne q}{\ab(p)((p)-(q))} \ .
\end{equation}
 Since $G$ is connected, for each $p \ne q$ there is a directed path from $q$ to $p$ consisting of some oriented edges $\{e^{(i)}\}_{1 \leq i \leq \ell}$ such that $e^{(1)}_-=q$, $e^{(\ell)}_+=p$, and $e^{(i)}_+ = e^{(i+1)}_-$ for $ 1 \leq i \leq \ell-1$. We may write 
\[
(p)-(q)=\sum_{i=1}^{\ell}{(e^{(i)}_+ -e^{(i)}_-)}=\sum_{i=1}^{\ell}{\partial (e^{(i)})} \ .
\]
Substituting this in \eqref{eq:telescope}, we conclude that $\ab=\sum_{e \in \EE(G)}{a_e \partial(e)}$ for some integers $a_e$. Therefore $\zeta=\sum_{e \in \EE(G)}{a_e\zeta_e}$ as we want.

\medskip

(iii) Assume $J \subseteq \EE(G)$ is such that the collection $\{\zeta_e\}_{e \in J}$ spans $\Div_{\RR}^0(G)^\vee$ as a real vector space. We need to show that $\{\zeta_e\}_{e \in J}$ also generates $\Prin(G)^\vee$ as a $\ZZ$-module. Let $\zeta\in \Prin(G)^\vee$. Then $\zeta = \sum_{e \in J}{b_e\zeta_e}$ for some $b_e \in \RR$ because $\{\zeta_e\}_{e \in J}$ spans $\Div_{\RR}^0(G)^\vee$. In other words
\[
\zeta(\cdot) = \langle \bb , \cdot \rangle_{\en}
\quad \text{ with }\quad \bb = \sum_{e \in J}{b_e\partial(e)}
\]
for some $b_e \in \RR$. We need to show that $b_e \in \ZZ$ for all $e \in J$. A computation similar to \eqref{eq:integervalues} shows that we have $\bb \in \Div^0(G)$. It is a well-known classical fact (due to Poincar\'e) that the incidence matrix of $G$ is totally unimodular (see, e.g., \cite[Proposition~5.3]{BiggsBook} and \S\ref{sec:cutlattice}). So $\sum_{e \in J}{b_e\partial(e)} \in \Div^0(G)$ will automatically imply that all $b_e$'s must be integers. 
\end{proof}
\begin{Remark}
It also follows from the proof of Proposition~\ref{prop:rightfunctionals}(ii) that
\begin{itemize}
\item[(i)] $\Prin(G)^\vee \cong \Div^0(G)$ and a canonical isomorphism is furnished by the energy pairing.
\item[(ii)] $C_1(G, \ZZ) \xrightarrow{\partial} C_0(G, \ZZ) \xrightarrow{\deg} \ZZ \rightarrow 0$ is an exact sequence. This statement, when $\ZZ$ is replaced with $\RR$ is classical (see, e.g., \cite[Proposition~12.1 and Proposition~28.1]{Biggs97}). 
\end{itemize}
\end{Remark}

\medskip

We are now ready to apply the results in \S\ref{sec:infarrag} to this setting.

\begin{Theorem} \label{thm:PrinDel}
Let $\H(\Div^0_\RR(G), \{\zeta_e\}_{e\in \EE(G)}) = \{C_{\ab}\}$ be the collection of all polyhedra
\begin{equation} \label{eq:Ca}
C_\ab = \{\bb \in \Div^0_\RR(G): \lfloor \zeta_e(\ab) \rfloor  \leq \zeta_e(\bb) \leq \lceil \zeta_e(\ab) \rceil  \text{ for all } e \in \EE(G) \} \ .
\end{equation}
as $\ab$ varies in $\Div^0_\RR(G)$. Then 
\begin{itemize}
\item[(i)] $\{C_{\ab}\}$ is a polyhedral cell decomposition of $\Div^0_\RR(G)$ by bounded convex polyhedra.  
\item[(ii)] The cell decomposition $\{C_{\ab}\}$ is invariant under the translation by the lattice $\Prin(G)$.
\item[(iii)] The set of $0$-dimensional cells in $\{C_{\ab}\}$ coincides with $\Prin(G)$. 
\item[(iv)] $\{C_{\ab}\}$ is the same as the Delaunay cell decomposition of $\Div^0_\RR(G)$ with respect to the lattice $\Prin(G)$ and the metric induced by the norm
\begin{equation} \label{eq:norm2}
\lVert \pb \rVert = \sqrt{\langle \pb, \pb \rangle_{\en}} = \sqrt{\Ec(\pb)} \ .
\end{equation}
\item[(v)] $\{C_{\ab}\}$ descends to a finite polyhedral cell decomposition of $\Div^0_\RR(G) / \Prin(G)$. 
\end{itemize}
\end{Theorem}
\begin{proof}
This result follows from Proposition~\ref{prop:rightfunctionals}, Theorem~\ref{thm:totunim}, and Remark~\ref{rmk:delon}(i). We only need to show that the norm defined in \eqref{eq:norm2} is compatible with the one considered in \eqref{eq:norm}. By Lemma~\ref{lem:lookgraphical}(i) any $\pb \in \Div^0_\RR(G)$ is of the form $\Delta(f)$ for some $f \in C^1(G, \RR)$. By \eqref{eq:useful}, Lemma~\ref{lem:lookgraphical}(ii), and Remark~\ref{rmk:selfadjoint} we have
\[
\begin{aligned}
\Ec(\pb) &= \langle \Delta(f), \Delta(f) \rangle_{\en} = \sum_{v \in V(G)}{f(u) \Delta(f) (u)} \\
&= \frac{1}{2}\sum_{v \in V(G)}{f(u) (\partial d f) (u)} =\frac{1}{2}\sum_{e \in \EE(G)}{(d f)(e)(d f)(e)}\\
&=\frac{1}{2}\sum_{e \in \EE(G)}{\lvert \zeta_e(\pb) \rvert ^2}\ .
\end{aligned}
\]
So the norm defined in \eqref{eq:norm2} is proportional to the norm defined in \eqref{eq:norm} and they induce the same Delaunay cell decomposition. 
\end{proof}

\medskip

The Delaunay cell decomposition $\{C_{\ab}\}$ of Theorem~\ref{thm:PrinDel} will be denoted by ${\Del}(\Prin(G))$. The induced finite cell decomposition of $\Div^0_\RR(G) / \Prin(G)$ will be denoted by ${\Del}(\Prin(G)) / \Prin(G)$. 

\medskip

\begin{Remark} \label{rmk:cells}
\begin{itemize}
\item[]
\item[(i)] Since $\zeta_{\bar{e}}=-\zeta_e$ for all $e \in \EE(G)$ we could alternatively define $C_\ab$ in \eqref{eq:Ca} as
\[
\{\bb \in \Div^0_\RR(G): \zeta_e(\bb) \leq \lceil \zeta_e(\ab) \rceil  \text{ for all } e \in \EE(G) \} \ .
\]
It follows that open cells in this cell complex correspond precisely to equivalence classes of points, where $\ab \sim \bb$ if and only if $\lceil \zeta_e(\ab) \rceil = \lceil \zeta_e(\bb) \rceil$ for all $e \in \EE(G)$.

\item[(ii)] By Lemma~\ref{lem:lookgraphical}(ii) the local picture at the origin is the image of the graphic arrangement defined in \S\ref{sec:BG} under the map $\Delta$. 
\item[(iii)] The cell complexes ${\Del}(\Prin(G))$ and ${\Del}(\Prin(G)) / \Prin(G)$ are related to the cell complexes ${\Del}(L(G))$ and ${\Del}(L(G)) / L(G)$ (defined in \S\ref{sec:cutlattice}) by the (restricted) boundary map (see Remark~\ref{rmk:isometry} and Remark~\ref{rmk:ResolRelation}). The finite cell complex ${\Del}(L(G)) / L(G)$ and the finite cell complex ${\Del}(\Prin(G)) / \Prin(G)$  have the same $f$-vector (i.e. the same number of $i$-dimensional faces for all $i$).
\end{itemize}
\end{Remark}

\medskip

The following lemma will be used in the proof of Theorem~\ref{thm:Ures}.
\begin{Lemma}\label{lem:contract}
Fix a divisor $E \in \Div(G)$. The subcomplex of $\Del(\Prin(G))$ on the lattice points $P(E)=\{D \in \Prin(G):  D \leq E\}$ is a polyhedral subdivision of a contractible space.
\end{Lemma}
\begin{proof}
$P(E)$ is precisely the set of lattice points inside the closed convex polytope $Q(E)=\{\ab \in \Div^0_\RR(G):  \ab \leq E\}$. The subcomplex of $\Del(\Prin(G))$ consisting of cells on the lattice points $P(E)$ consists of all Delaunay cells on these lattice points. Recall $\Del(\Prin(G))$ is a tiling of the ambient space. Therefore this subcomplex forms a space which is homotopy equivalent to the polytope $Q(E)$ itself, and therefore is contractible.
\end{proof}

\subsection{Labeling $\Del(\Prin(G))$ and the minimal free resolution of $\II_G$} \label{sec:labelDelPrin}

Let $\Tb = K[\xb,\xb^{-1}]$ denote the Laurent polynomial ring in variables $\{x_v: v \in V(G)\}$. Clearly $\Tb$ is a module over $\Rb$. Consider the $\Rb$-submodule $\UU_G \subset \Tb$ generated by Laurent monomials $\{\xb^D: D \in \Prin(G)\}$. This Laurent monomial module $\UU_G$ may be thought of as the ``universal cover'' of $\II_G$ and many question about $\II_G$ can be reduced to questions about $\UU_G$. For example, the free resolutions of $\UU_G$ and $\II_G$ are closely related. See \cite{BayerSturmfels} for an extensive study of this relation. Since the only effective divisor in $\Prin(G)$ is the all-$0$ divisor, the results of \cite{BayerSturmfels} apply to our situation.

\medskip

Consider the cell decomposition $\Del(\Prin(G))$. By Theorem~\ref{thm:PrinDel} the set of $0$-dimensional cells in $\Del(\Prin(G))$ is precisely $\Prin(G)$. We will label each $0$-cell $D \in \Prin(G)$ by the Laurent monomials $\xb^D$. As usual, we let the label of any other cell to be the least common multiple of the labels of its vertices. This labeled cell complex leads to a complex of free $\Div(G)$-graded $\Rb$-modules
\[
\Fc_G:= \Fc_{\Del(\Prin(G))} = \bigoplus_{\emptyset \ne F \in \Del(\Prin(G))}{\Rb(-\mb_F)}
\]
where $\mb_F$ denotes the monomial label of the face $F$. Let $[F]$ denote the generator of $\Rb(-\mb_F)$. The differential of $\Fc_G$ is the homogenized differential (boundary) operator of the cell complex $\Del(\Prin(G))$:
\begin{equation} \label{eq:differntials}
\partial([F]) = \sum_{{\rm codim(F,F')=1} \atop F' \subset F}{\varepsilon (F,F') \frac{\mb_F}{\mb_{F'}} \ [F']} 
\end{equation}
where $\varepsilon (F,F') \in \{-1, +1\}$ denotes the incidence function indicating the orientation of $F'$ in the boundary of $F$. 

\medskip

\begin{Lemma}\label{lem:nounit}
\begin{itemize}
\item[]
\item[(i)] Let $\ab \in \Div^0_\RR(G)$. Then $\ab(v)=\sum_{e_+=v}\zeta_e(\ab)$. 
\item[(ii)] Let $F=C_\ab$ be a cell in $\Del(\Prin(G))$ corresponding to a point $\ab \in \Div^0_\RR(G)$ (i.e. $\ab \in \relint(F)$). Then $\mb_F= \xb^E$ where $E \in \Div(G)$ is defined by
\begin{equation} \label{eq:prinlabels}
E(v)= \sum_{e_+=v} \lceil \zeta_e(\ab) \rceil \ .
\end{equation}
\item[(iii)] For distinct faces $F' \subsetneq F$ of $\Del(\Prin(G))$ we have $\mb_F \ne \mb_{F'}$.
\end{itemize}
\end{Lemma}
\begin{proof}
(i) By Lemma~\ref{lem:lookgraphical}(i) we may write $\ab = \Delta(f)$ for some $f \in C^1(G, \RR)$. By definition we have $\Delta(f) = \sum_{v}\sum_{e_+=v}(f(e_+)-f(e_-)) (v)$. Therefore, it follows from Lemma~\ref{lem:lookgraphical}(ii) that $\ab(v)=\sum_{e_+=v}\zeta_e(\ab)$.  

(ii) follows from (i) and the fact that open cells in $\Del(\Prin(G))$ correspond precisely to equivalence classes of points, where $\ab \sim \bb$ if and only if $\lceil \zeta_e(\ab) \rceil = \lceil \zeta_e(\bb) \rceil$ for all $e \in \EE(G)$ (Remark~\ref{rmk:cells}(ii)).

\medskip

(iii) Let $F=C_\ab$ for $\ab \in \relint(F)$ and $F'=C_{\ab'}$ for $\ab' \in \relint(F')$. Since $\ab'$ is in $F$ as well, it satisfies $\zeta_e(\ab') \leq \lceil \zeta_e(\ab) \rceil$ for all $e \in \EE(G)$. Therefore we have $\lceil \zeta_e(\ab')\rceil \leq \lceil \zeta_e(\ab) \rceil$. But since $F' \ne F$ there must exist some $e$ such that $\zeta_e(\ab') \in \ZZ$ but $\zeta_e(\ab) \not \in \ZZ$ and therefore $\lceil \zeta_e(\ab')\rceil < \lceil \zeta_e(\ab) \rceil$. The result now follows from part (ii) because for this edge, by \eqref{eq:prinlabels}, the exponent of $x_{e_+}$ in $\mb_{F'}$ must be strictly less than the exponent of $x_{e_+}$ in $\mb_{F}$.
\end{proof}

\begin{Theorem}\label{thm:Ures}
The complex $(\Fc_G, \partial)$ is a minimal $\Div(G)$-graded free resolution of the module $\UU_G$ over $\Rb$. 
\end{Theorem}
\begin{proof}
We need to show two things:
\begin{itemize}
\item[(i)] $(\Fc_G, \partial)$ is exact, i.e. $(\Fc_G, \partial)$ is a cellular free resolution of $\UU_G$.
\item[(ii)] For distinct faces $F' \subsetneq F$ of $\Del(\Prin(G))$ with $\codim(F,F')=1$ we have $\mb_F \ne \mb_{F'}$, i.e. no unit of $\Rb$ appears in differential maps and the resolution $(\Fc_G, \partial)$ is minimal.
\end{itemize}

By \cite[Proposition~1.2]{BayerSturmfels}, we know (i) is equivalent to 
\begin{itemize}
\item[(i')] For each $E \in \Div(G)$, the subcomplex of $\Del(\Prin(G))$ on the lattice points $\{D \in \Prin(G):  D \leq E\}$ is acyclic over the field $K$, i.e. its reduced homology $\widetilde{H}_i$ with $K$ coefficients vanishes for all $i \geq 0$. 
\end{itemize}
(i') follows from Lemma~\ref{lem:contract} and (ii) follows from Lemma~\ref{lem:nounit}(iii). 
\end{proof}

From Theorem~\ref{thm:Ures} and \cite[Corollary~3.7]{BayerSturmfels} we immediately obtain the following theorem. 
\begin{Theorem}\label{thm:Ires}
The quotient cell complex $\Del(\Prin(G)) / \Prin(G)$ supports a $\Pic(G)$-graded minimal free resolution for $I_G$. 
\end{Theorem}

\begin{Example} \label{ex:PrinResol}

Consider the graph $K_3$ with a fixed orientation as in Figure~\ref{fig:K30}. 

\begin{figure}[h!]
\begin{center}

\begin{tikzpicture} [scale = .30, very thick = 20mm]

\node (n1) at (34,12) [Cwhite] {$u_1$};
  \node (n2) at (30,6)  [Cwhite] {$u_3$};
  \node (n3) at (38,6)  [Cwhite] {$u_2$};

\node (n1) at (34,5.2) [Cwhite] {$e_1$};
  \node (n2) at (31.7,9)  [Cwhite] {$e_3$};
  \node (n3) at (36.3,9)  [Cwhite] {$e_2$};

  \node (n1) at (34,11) [Cgray] {};
  \node (n2) at (31,6)  [Cgray] {};
  \node (n3) at (37,6)  [Cgray] {};
 \foreach \from/\to in {n3/n1,n1/n2,n2/n3}
  \draw[black][->] (\from) -- (\to);

\end{tikzpicture}

\caption{Graph $K_3$ and a fixed orientation $\Oc$}
\label{fig:K30}
\end{center}
\end{figure}
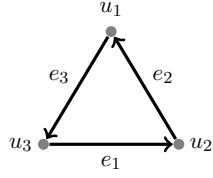

The lattice $\Prin(G)$ is two dimensional and is depicted in Figure~\ref{fig:PrinLattice}. This lattice ``lives in'' $C_0(G,\RR)=\span\{(u_1),(u_2),(u_3)\} \cong \RR^3$. In the picture $c_1=\Delta(\chi_{u_1})=2(u_1)-(u_2)-(u_3)$, $c_2=\Delta(\chi_{u_2})=-(u_1)+2(u_2)-(u_3)$, and $c_3=\Delta(\chi_{u_3})=-(u_1)-(u_2)+2(u_3)$.

\begin{figure}[h!]

\begin{center}
\begin{tikzpicture} [scale = .30, very thick = 20mm]

\node (n42) at (20,1)  [Cblack] {};
\node (n41) at (14,1)  [Cblack] {}; 
\node (n43) at (26,1)  [Cblack] {};
\node (n44) at (32,1)  [Cblack] {};
\node (n11) at (14,11) [Cblack] {};  
\node (n12) at (20,11) [Cblack] {};  
\node (n13) at (26,11) [Cblack] {};  
\node (n14) at (32,11) [Cblack] {};
\node (n21) at (11,6)  [Cblack] {}; 
\node (n22) at (17,6)  [Cblack] {}; 
\node (n23) at (23,6)  [Cblack] {}; 
\node (n24) at (29,6)  [Cblack] {};  
\node (n25) at (35,6)  [Cblack] {};
\node (n4) at (14,1)  [Cblack] {};
\node (n1) at (14,11) [Cblack] {};
\node (n2) at (11,6)  [Cblack] {};
\node (n3) at (17,6)  [Cblack] {};
\node (n51) at (23,-4) [Cblack] {}; 
\node (n52) at (29,-4) [Cblack] {};   
\node (n70) at (20,-8.3) [C0] {}; 

\node (n18) at (35,16) [C0] {$\zeta_{e_1}=0$};
\node (n40) at (6.7,1)  [C0] {$\zeta_{e_2}=0$};
\node (n10) at (17,16) [C0] {$\zeta_{e_3}=0$}; 
\node (n181) at (38,11) [C0] {$\zeta_{e_1}=1$};
\node (n401) at (3.7,6)  [C0] {$\zeta_{e_2}=1$};
\node (n101) at (11,16) [C0] {$\zeta_{e_3}=1$};

\node (m24) at (30.5,6.8)  [C0] {$c_1$};
\node (mm) at (30.5,-4.5) [C0] {$c_2$};  
\node (m42) at (19,0.17)  [C0] {$c_3$};
\node (m42) at (25,0.17)  [C0] {$0$};

\foreach \from/\to in {n40/n41, n401/n21}
    \draw[black][dashed] (\from) -- (\to);
\foreach \from/\to in {n14/n18, n70/n51,n25/n181}
    \draw[green][dashed] (\from) -- (\to);
\foreach \from/\to in {n10/n12,n1/n3, n101/n11}
    \draw[red][dashed] (\from) -- (\to);

\foreach \from/\to in {n11/n14,n21/n25,n41/n44,n51/n52}
  \draw[black][] (\from) -- (\to);
\foreach \from/\to in {n11/n22,n22/n42,n42/n51,n12/n23,n23/n43,n43/n52,n13/n24,n24/n44,n14/n25,n21/n41}
  \draw[red][] (\from) -- (\to);
\foreach \from/\to in {n11/n21,n12/n22,n22/n41,n13/n23,n23/n42,n14/n24,n24/n43,n43/n51,n25/n44,n44/n52}
  \draw[green][] (\from) -- (\to);

\end{tikzpicture}
\caption{The lattice $(\Prin(G), \langle \cdot , \cdot \rangle_{\en})$ and the associated cellular decomposition of the ambient space $\Div_{\RR}^0(G)$}
\label{fig:PrinLattice}
\end{center}
\end{figure}
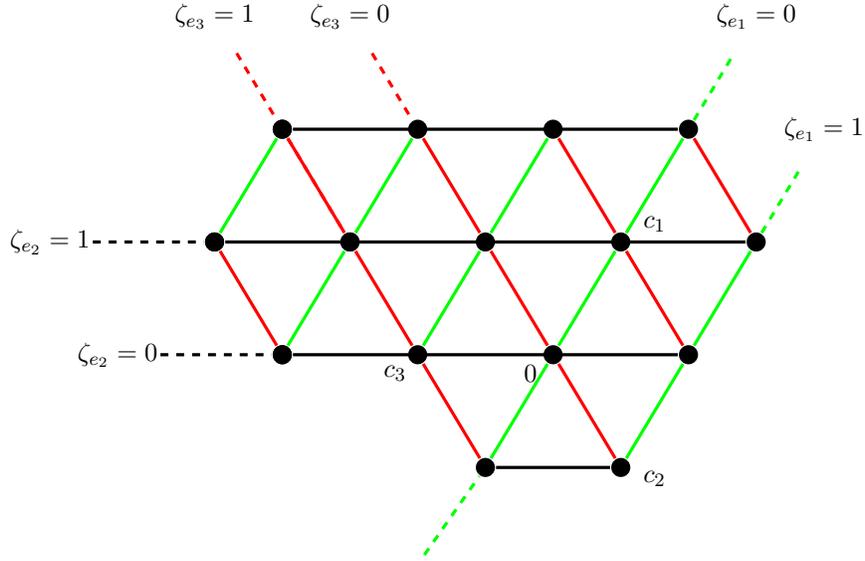

The cell decomposition $\Del(\Prin(G))$ is the Delaunay decomposition of $\Div^0_{\RR}(G)$ with respect to the principal lattice and the energy distance (Theorem~\ref{thm:PrinDel}(iv)) which coincides with the infinite hyperplane arrangement \eqref{eq:Ca}. 
The quotient cell complex $\Del(\Prin(G))/\Prin(G)$ of the torus has one $0$-cell $\{v\}$ (orbit of the origin), three $1$-cells $\{e,e',e''\}$ (orbits of green, red, and black edges), and two $2$-cells $\{f,f'\}$ (orbits of upward and downward triangles). 

In Figure~\ref{fig:fundPrin} we have chosen a fundamental domain for the lattice, and have labeled all cells of this fundamental domain according to the recipe described in the beginning of \S\ref{sec:labelDelPrin} or, equivalently, in Lemma~\ref{lem:nounit}(ii). For simplicity we have used $x_i$ instead of $x_{u_i}$. The labeled cell complex in Figure~\ref{fig:fundPrin} is enough to completely describe a minimal free resolution for both $\II_G$ and $\UU_G$. Concretely, the minimal resolution of $\II_G$ is as follows:
\[
 0 \rightarrow \Rb(-\mb_{f}) \oplus \Rb(-\mb_{f'}) \xrightarrow{\partial_2} \Rb(-\mb_{e}) \oplus \Rb(-\mb_{e'})\oplus \Rb(-\mb_{e''}) \xrightarrow{\partial_1} \Rb(-\mb_v) \ .
\]
As usual, assume $[F]$ denotes the generator of $\Rb(-\mb_F)$. Let
\[
\mb_e = x_1^2 \ ,
\quad
\mb_{e'} = x_1 x_2 \ ,
\quad
\mb_{e''} = x_2^2 \ ,
\]
\[
\mb_f = x_1^2 x_2 \ ,
\quad
\mb_{f'}=x_1 x_2^2 \ .
\]
The homogenized differential operator (see \eqref{eq:differntials}) $(\partial_1, \partial_2)$ of the cell complex is described as follows:

\[
\partial_1([e])= \frac{x_1^2}{1}[v] -\frac{x_1^2}{\frac{x_1^2}{x_2x_3}}[v] =  ({x_1^2} - {x_2x_3})[v]\ ,
\]
\[
\partial_1([e'])= \frac{x_1 x_2}{\frac{x_1x_2}{x_3^2}}[v] - \frac{x_1x_2}{1}[v] = ({x_3^2} - {x_1x_2})[v]\ , 
\]
\[
\partial_1([e''])= \frac{x_2^2}{\frac{x_2^2}{x_1x_3}}[v] - \frac{x_2^2}{1}[v] = ({x_1x_3} - {x_2^2})[v]\ , 
\]
\[
\partial_2([f])= \frac{x_1^2 x_2}{x_1^2}[e] -  \frac{x_1^2 x_2}{\frac{x_1^2 x_2}{x_3}}[e''] + \frac{x_1^2 x_2}{x_1 x_2}[e'] = {x_2}[e] - {x_3}[e'']  + {x_1}[e']   \ ,
\]
\[
\partial_2([f'])= \frac{x_1 x_2^2}{\frac{x_1x_2^2}{x_3}}[e] -  \frac{x_1 x_2^2}{x_2^2}[e''] + \frac{x_1 x_2^2}{x_1 x_2}[e'] = {x_3}[e] - {x_1}[e'']  + {x_2}[e']   \ .
\]
Clearly $\II_G$ is the image of $\partial_1$ after identifying $[v]$ with $1 \in \Rb$ (see Theorem~\ref{thm:Cori}).
Note that, since the labeling is compatible with the action of the lattice, any other fundamental domain would give rise to the exact same description of the differential maps.

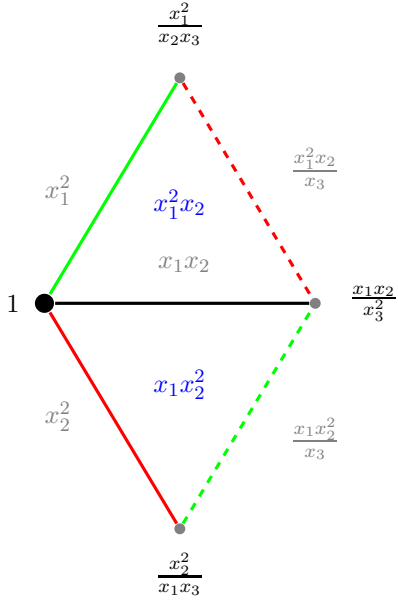
\begin{figure}[h!]

\begin{center}

\begin{tikzpicture} [scale = .60, very thick = 20mm]

\node (n4) at (14,1)  [Cgray] {};
\node (n1) at (14,11) [Cgray] {}; 
\node (n2) at (11,6)  [Cblack] {}; 
\node (n3) at (17,6)  [Cgray] {}; 

\node (m1) at (14,0) [C0] {$\frac{x_2^2}{x_1x_3}$};
\node (m1) at (14,12.2) [C0] {$\frac{x_1^2}{x_2x_3}$};
\node (m3) at (18.3,6) [C0] {$\frac{x_1 x_2}{x_3^2}$};
\node (m) at (10.3,6)  [C0] {$1$};

\foreach \from/\to in {n4/n2}
  \draw[red][] (\from) -- (\to);
\foreach \from/\to in {n1/n2}
  \draw[green][] (\from) -- (\to);
\foreach \from/\to in {n2/n3}
  \draw[black][] (\from) -- (\to);
\foreach \from/\to in {n1/n3}
  \draw[red][dashed] (\from) -- (\to);
\foreach \from/\to in {n4/n3}
  \draw[green][dashed] (\from) -- (\to);

\node (m10) at (11.3,8.5) [C0] {\textcolor{gray}{$x_1^2$}};
\node (m14) at (11.3,3.5) [C0] {\textcolor{gray}{$x_2^2$}};
\node (m23) at (14.1,6.86) [C0] {\textcolor{gray}{$x_1x_2$}};
\node (m023) at (17,9.0) [C0] {\textcolor{gray}{$\frac{x_1^2x_2}{x_3}$}};
\node (m023) at (17,3.0) [C0] {\textcolor{gray}{$\frac{x_1 x_2^2}{x_3}$}};

\node (m123) at (14,8.2) [C0] {\textcolor{blue}{$x_1^2 x_2$}};
\node (m423) at (14,4.2) [C0] {\textcolor{blue}{$x_1x_2^2$}};

\end{tikzpicture}
\caption{A choice of fundamental domain with labels}
\label{fig:fundPrin}
\end{center}

\end{figure}
\end{Example}

\medskip

\begin{Remark} \label{rmk:isometry}
It follows from the computation
\[
\begin{aligned}
\langle \Delta(f), \Delta(g) \rangle_{\en} &= \sum_{v \in V(G)}{f(u) \Delta(g) (u)} = \sum_{v \in V(G)}{f(u) (\partial_{\Oc} d_{\Oc} g) (u)} \\
&=\sum_{e \in \EE(G)}{(d_{\Oc} f)(e)(d_{\Oc} g)(e)} 
\end{aligned}
\]
that there is an isometry between the principal lattice $(\Prin(G) , \langle \cdot, \cdot \rangle_{\en})$ and the {\em cut lattice} (lattice of integral cocyles) $(L(G) , \langle \cdot , \cdot \rangle )$ defined in \S\ref{sec:cutlattice}. It is natural to ask whether there are other ideals defined directly in terms of the cut lattice and, if so, whether there are nice relations between these ideals. These questions will be answered in this work (see \S\ref{sec:relate}, especially Remark~\ref{rmk:ResolRelation}). 
\end{Remark}

\begin{Remark}\label{rmk:Mg}
It is possible to give a polyhedral cellular free resolution of the ideal $\MM_G^q$ using the local picture at the origin of $\Del(\Prin(G))$ (or, alternatively, using the graphic hyperplane arrangement -- see Remark~\ref{rmk:cells}(ii)) and study its Gr\"obner relation with $\II_G$, similar to what we will do for $\OO_G^q$ in relation to $\JJ_G$ in \S\ref{sec:matroidlawrence}. Instead, we will show (in \S\ref{sec:relate}) that one could alternatively relate $\II_G$ to $\JJ_G$ and $\MM_G^q$ to $\OO_G^q$ via a regular sequence. As a corollary, this gives an alternate way to describe polyhedral cellular free resolutions of all these ideals and to compare their Betti numbers. 
\end{Remark}

\begin{Remark}
The minimal free resolution of $\MM_G^q$ is a Koszul complex when $G$ is a tree because $\MM_G^q$ is generated by the variables $\{x_v : v \ne q\}$ (see Theorem~\ref{thm:Cori}). When $G$ is a complete graph, the minimal free resolution of $\MM_G^q$ is given by a Scarf complex (see, e.g., \cite[Corollary~6.9]{PostnikovShapiro04}).
\end{Remark}

\section{Graphs, arrangements, and integral cuts} \label{sec:arrg}

\subsection{Graphic arrangements and connected partitions} \label{sec:BG}
Following \cite{GreenZaslavsky}, we define the {\em graphic hyperplane arrangement} as follows. An important feature that we want to emphasize in this section is that this arrangement naturally ``lives in'' the Euclidean space $C^0(G,\RR)$, i.e. the vector space of all real-valued functions on $V(G)$ endowed with the bilinear form 
\[
\langle f_1, f_2\rangle = \sum_{v\in V(G)}{f_1(v)f_2(v)} \ .
\]

Recall that $C^1(G, \RR)$ denotes the vector space of real-valued functions on $\EE(G)$ and $d \colon C^0(G,\RR) \rightarrow C^1(G,\RR)$ denotes the usual coboundary map.

For each edge $e \in \EE(G)$, let $\H_{e} \subset C^0(G,\RR)$ denote the hyperplane
\[
\H_{e} = \{f \in C^0(G,\RR): (d f) (e)=0 \} \ .
\]

Note that $\H_{\bar{e}}=\H_{e}$. Consider the arrangement 
\[
\H'_G= \{\H_e : e \in \EE(G)\} 
\]
in $C^0(G,\RR)$. Since $G$ is connected, we know $\bigcap_{e \in \EE(G)}{\H_e}$ is the $1$-dimensional space of constant functions on $V(G)$, which is the same as the kernel of $d$. We define the {\em graphic arrangement} corresponding to $G$, denoted by $\H_G$, to be the restriction of $\H'_G$ to the hyperplane 
\begin{equation}\label{eq:kerperp}
(\Ker(d))^{\perp} = \{f \in C^0(G, \RR) : \sum_{v\in V(G)}f(v)=0\} \ .
\end{equation}

\medskip

The intersection poset of $\H_G$ (i.e. the collection of nonempty intersections of hyperplanes $\H_e$ ordered by reverse inclusion) is naturally isomorphic to the poset of connected partitions of $G$ (i.e. partitions of $V(G)$ whose blocks induce connected subgraphs). See, e.g., \cite[p.112]{GreenZaslavsky}. 

\medskip

It is well-known that there is a one-to-one correspondence between acyclic orientations of $G$ and the regions of $\H_G$ (see, e.g., \cite[Lemma~7.1 and Lemma~7.2]{GreenZaslavsky}). Given any function $f \in C^0(G,\RR)$ one can label each vertex $v$ with the real number $f(v)$. In this way we obtain an acyclic partial orientation of $G$ by directing $v$ to $u$ if $f(u) < f(v)$. Recall this means we have an acyclic orientation on the graph $G/f$ obtained by contracting all unoriented edges (i.e. all edges  $\{u,v\}$ with $f(u)=f(v)$).

\medskip

We are mainly interested in acyclic orientations of $G$ with a {\em unique source} at $q \in V(G)$. For this purpose, we fix a real number $c>0$ and define

\[
\H^{q,c}=\{f\in C^0(G, \RR): f(q)=-c \}  \ .
\]
The restriction of the arrangement $\H_G$ to $\H^{q,c}$ will be denoted by $\H_G^{q,c}$. We denote the {\em bounded complex} (i.e. the polyhedral complex consisting of bounded cells) of $\H_G^{q,c}$ by $\B_G^{q,c}$. 

\begin{Remark}
\begin{itemize}
\item[]
\item[(i)] By \eqref{eq:kerperp}, the restriction of $\H_G$ to $\H^{q,c}$ coincides with the restriction of $\H_G$ to
\[
(\H^{q,c})'=\{f\in C^0(G, \RR): \sum_{v \ne q}f(v)=c \}  \ .
\]

\item[(ii)] We will see in \S\ref{sec:grobrel} (e.g. Lemma~\ref{lem:labels}(ii)) that it is most natural (although not necessary) to choose $0 <c<1$.
\end{itemize}
\end{Remark}

The following lemma relates regions of $\B_G^{q,c}$ to acyclic orientations with unique source at $q$ (see also \cite[Theorem~7.3]{GreenZaslavsky}).

\begin{Lemma} \label{lem:BG}
Each $f \in \B_G^{q,c}$ gives an acyclic partial orientation of $G$ with a unique source at $q$. {In particular $f(v) \geq f(q)$ for any edge $\{v,q\} \in E(G)$.}
\end{Lemma}

\begin{proof}
Since we are considering the orientation on $G / f$ we may assume $f(u) \ne f(v)$ for any $\{u, v\} \in E(G)$.
Since any acyclic orientation of $G$ has at least one source vertex%
\footnote{It is an elementary fact that {\em any} acyclic orientation of $G$ has at least one source and one sink.}, 
it suffices to show that no vertex $v \ne q$ can be a source in the orientation corresponding to $f$. 

Let $w$ be a vertex such that $f(w)$ is maximum (i.e. $f(w) \geq f(v)$ for all $v \in V(G)$). To obtain a contradiction, assume $s \ne q$ is a source and therefore $f(v) > f(s)$ for all $\{v, s\} \in V(G)$. 

Recall that $\chi_v$ denotes the characteristic function of $v \in V(G)$. It is straightforward to check that 
\[
f_t=f + t (\chi_w-\chi_s) \in C^0(G, \RR)
\]
 also belongs to the same cell as $f$ for {\em any} $t \geq 0$. 
However, not all $f_t$ for $t \geq 0$ can be contained in the bounded complex because they constitute a ray in $C^0(G, \RR)$ emanating from $f$.
\end{proof}

\medskip

\begin{Remark} \label{sec:nopartorient}
It follows (see also \cite[Corollary~7.3]{GreenZaslavsky}) that the number of $i$-dimensional cells in $\B_G^{q,c}$ is equal to the number of acyclic partial orientations of $G$ with $(i+2)$ (connected) components having a unique source at $q$. For an example, see Example~\ref{exam:1}.
\end{Remark}

\subsection{Lattice of integral cuts and graphic infinite arrangements}
\label{sec:cutlattice}
Fix an arbitrary orientation $\Oc \subset \EE(G)$. Consider the restricted coboundary map $d_\Oc : C^0(G,\ZZ) \to C_{\Oc}^1(G,\ZZ)$ and the usual bilinear form on $C_\Oc^1(G,\ZZ)$ defined by 
\begin{equation}\label{eq:LGpairing}
\langle g_1, g_2 \rangle = \sum_{e \in \Oc}{g_1(e)g_2(e)} \ .
\end{equation}
The {\em lattice of integral cuts} (with respect to the orientation $\Oc$) is by definition the group of integral coboundaries $\Image(d_\Oc)$ inside $C_\Oc^1(G,\ZZ)$ with its bilinear form induced from \eqref{eq:LGpairing}. It is denoted by $L(G,\Oc)$. When the orientation is clear we simply denote it by $L(G)$.

\begin{Remark} \label{rmk: orientvslattice}
Consider the (unrestricted) coboundary map 
\[
d : C^0(G,\ZZ) \to C^1(G,\ZZ) \cong C_{\Oc}^1(G,\ZZ) \oplus C_{\bar{\Oc}}^1(G,\ZZ) \ .
\]
Its image $\Lambda=\Image(d)$ is isomorphic to the lattice $\{(a,-a) : a \in L(G,\Oc)\}$. The choice of the orientation $\Oc$ gives a splitting of $C^1(G,\ZZ)$ and of $\Lambda$. 
\end{Remark}

\medskip

We may identify $C^0(G,\ZZ)$ with $\ZZ^{V(G)}$ and $C_\Oc^1(G,\ZZ)$ with $\ZZ^{\Oc}$. If we also fix a labeling on the vertices and edges of the graph, then $d_\Oc$ is represented by the matrix $B^T$, where $B$ is the $n \times m$ vertex-edge incidence matrix of $G$. In this case, the lattice of integral cuts $L(G)$ is $\Image(B^T) \hookrightarrow \ZZ^m$. It is a well-known classical fact (due to Poincar\'e) that the matrix $B$ is totally unimodular in the sense of Example~\ref{ex:WU} (see, e.g., \cite[Proposition~5.3]{BiggsBook}). Therefore Theorem~\ref{thm:totunim}(iii) and Remark~\ref{rmk:delon} apply to this situation.
The Delaunay cell decomposition corresponding to the lattice $L(G)$ will be denoted by ${\Del}(L(G))$.
It follows from the total unimodularity that every element with minimal nonempty support in $L(G)$ is an integral multiple of a {\em bonds} (i.e. minimal edge-cuts, or, equivalently, edge-cuts connecting two connected subgraphs) (see, e.g., \cite[\S{1} and \S5]{Tutte})


\section{Graphic oriented matroid ideal and Lawrence ideal}
\label{sec:matroidlawrence}
We next study some natural ideals associated to the cell complexes introduced in \S\ref{sec:arrg}. See \cite{Popescu} and \cite{novik2002syzygies} for a more general study of such constructions.

\subsection{Graphic oriented matroid ideal}

An {\em oriented hyperplane arrangement} is a real hyperplane arrangement along with a choice of a ``positive side'' for each hyperplane. Equivalently, one may fix a set of linear forms vanishing on hyperplanes to fix the ``orientation''. For any oriented hyperplane arrangement one can define (see \cite{novik2002syzygies}) the associated {\em oriented matroid ideal}: let $\{h_j\}$ be $m$ nonzero linear forms defining the hyperplane arrangement $\A$ with hyperplanes $\H_j = \{ \pb \in V : h_j(\pb)=c_j\}$ in a real affine space $V$. The oriented matroid ideal associated to $\A$ is the ideal in $2m$ variables of the form:
\[
\OO_{\A} = \langle \mb(\pb) : \pb \in V \rangle \subset K[\wb,\zb]
\]
where for each $\pb \in V$
\[
\mb(\pb)=\prod_{h_i(\pb)>c_i}w_i  \prod_{h_i(\pb)<c_i}z_i \ .
\]
Note that any two points in the relative interior of a cell will give rise to the same monomial.
\medskip

Consider the hyperplane arrangement $\H_G^{q,c}$ (defined in \S\ref{sec:BG}) which is contained in a codimension $2$ affine subspace of $C^0(G, \RR)$. Fixing an orientation $\Oc$ of the graph  $G$ will fix the linear forms $(d f)(e)=f(e_{+})-f(e_{-})$ for $e \in \Oc$ and gives an orientation to the hyperplane arrangement $\H_G^{q,c}$. The oriented matroid ideal associated to this oriented hyperplane arrangement $\H_G^{q,c}$ will be denoted by $\OO_G^q$ (instead of $\OO_{\H_G^{q,c}}$) and will be called the {\em graphic oriented matroid ideal} associated to $G$ and $q$. It follows from the discussion in \S\ref{sec:BG} that this ideal is independent of the choice of the real number $c>0$. In this situation, we may consider the variables $\wb$ as $\{y_e: e \in \Oc\}$ and the variables $\zb$ as $\{y_{\bar{e}}: e \in \Oc\}$ and then $\OO_G^q \subset \Sb$.

\subsection{Graphic Lawrence ideal}

For any embedded integral lattice $L \hookrightarrow \ZZ^m$ one can define (see \cite[Chapter 7]{SturmfelsGrobnerConvex}) a binomial ideal $\JJ_L$ in $2m$ variables, called the {\em Lawrence ideal} of $L$, by the following formula:

\[ 
\JJ_L= \langle \wb^{a^+}\zb^{a^-}-\wb^{a^-}\zb^{a^+}: a^+,a^-\in \NN^m,\ a={a^+}-{a^-}\in L \rangle \subset K[\wb,\zb] \ .
\]

When the lattice $L$ is unimodular, the Lawrence ideal $\JJ_L$ is called unimodular (\cite{Popescu}).

\medskip

For simplicity, the unimodular Lawrence ideal associated to the unimodular lattice of integral cuts $L(G)$ will be denoted by $\JJ_G$ (instead of $\JJ_{L(G)}$) and will be called the {\em graphic Lawrence ideal} of $G$. Again, we may consider the variables $\wb$ as $\{y_e: e \in \Oc\}$ and the variables $\zb$ as $\{y_{\bar{e}}: e \in \Oc\}$ and then $\JJ_G \subset \Sb$.


\medskip

\subsection{Labeling $\B_G^{q,c}$ and the minimal free resolution of $\OO_G^q$} \label{sec:labelBG}
The bounded polyhedral cell complex $\B_G^{q,c}$ (defined in \S\ref{sec:BG}) supports a minimal free resolution for the ideal $\OO_G^q$. To see this, we need to label the vertices of $\B_G^{q,c}$ appropriately:  each vertex $f \in \B_G^{q,c}$ is labeled by the monomial 

\begin{equation} \label{eq:BLabels}
\mb(f) = \prod_{e \in \EE(G) \atop (d f)(e)>0}{y_{e}} \ .
\end{equation}
\begin{Remark}
Fixing an orientation $\Oc$ will result in the factorization of $\mb(f)$ as
\[
\mb(f) =\prod_{e \in \Oc \atop  f(e_+) -f(e_-)>0}{y_{e}} \prod_{e \in \Oc \atop  f(e_-) -f(e_+)>0}{y_{\bar{e}}}  \ .
\]
\end{Remark}

In this way, we obtain a labeling of all cells by the least common multiple construction. It is easily seen that the label of any cell will be $\mb(f)$ (as in \eqref{eq:BLabels}) for any point $f$ in the relative interior of that cell. 

The following result is an application of \cite[Theorem~1.3(b)]{novik2002syzygies} for the hyperplane arrangement $\H_G^q$.

\begin{Theorem}
\label{thm:bounded}
The labeled polyhedral cell complex $\B_G^{q,c}$ gives a $C^1(G,\ZZ)$-graded minimal free resolution for $\OO_G^q$. 
In particular, $\OO_G^q$ is minimally generated by the monomials $\mb(f)$, as $f$ ranges over the vertices of $\B_G^{q,c}$. 
\end{Theorem}
The fact that there is no unit in the corresponding differential maps is immediate from the description of the labelings. All subcomplexes $(\B_G^{q,c})_{\leq \mb}$ are in fact contractible, by a result of Bj\"orner and Ziegler (\cite[Theorem~4.5.7]{orientedmatroid}). See \cite{novik2002syzygies} for more details, and Example~\ref{exam:1} and Figure~\ref{fig:arrangement} for an example.

\subsection{Labeling $\Del(L(G))$ and the minimal free resolution of $\JJ_G$} \label{sec:labelDel}

Fix an arbitrary orientation $\Oc \subset \EE(G)$ of $G$ and consider the lattice of integral cuts $L(G)$ as in \S\ref{sec:cutlattice}. As we have already discussed, it comes equipped with a canonical polyhedral cell decomposition of the ambient real vector space $L(G)_\RR=L(G) \otimes \RR = \Image (d_\Oc \colon C^0(G, \RR) \rightarrow C_{\Oc}^1(G, \RR))$. This polyhedral cell decomposition, denoted by $\Del(L(G))$, can be thought of as an infinite hyperplane arrangement (Theorem~\ref{thm:totunim}(iii)), or more naturally, as the Delaunay decomposition of the ambient space with respect to the lattice $L(G)$ and the metric induced by its natural pairing \eqref{eq:LGpairing} (See Remark~\ref{rmk:delon}(ii)). We make this a labelled cell complex by assigning the label
\begin{equation} \label{eq:blabel}
\bb(a)=\prod_{e \in \EE(G)} { {y_{e}^{a(e)}} }
\end{equation}
to each vertex $a \in L(G) \hookrightarrow C^1(G, \RR)$.

\begin{Remark}
Fixing an orientation $\Oc$ will result in the factorization of this Laurent monomial as
\[
\bb(a)=\prod_{e \in \Oc} { {y_{e}^{a(e)}} }  \prod_{e \in \bar{\Oc}} { {y_{e}^{-a(e)}} } = \prod_{e \in \Oc} { {y_{e}^{a(e)}} } / \prod_{e \in \Oc} { {y_{\bar{e}}^{a(e)}} }
\]
for $a \in L(G)$. 
\end{Remark}

As usual, we extend the labeling to all faces by the least common multiple rule. The associated complex of free $C^1(G, \ZZ)$-graded $\Sb$-modules (see \S\ref{sec:res}) is not $\Sb$-finite. By \cite[Theorem~3.1]{Popescu} this complex is a minimal cellular free resolution of the (Laurent) monomial module generated by the labels of the lattice points in $L(G)$. This Laurent monomial module can be thought of as the ``universal cover'' of $\JJ_G$; the Delaunay cell complex is invariant under the translation by $L(G)$ (Theorem~\ref{thm:totunim} and Remark~\ref{rmk:delon}), and the labeling is also compatible with this action. So we obtain a well defined finite cell complex on the quotient torus $L(G)_\RR / L(G)$, which we denote by $\Del(L(G)) / L(G)$. The following theorem is an application of \cite[Theorem~3.5]{Popescu} (or \cite[Theorem~3.2]{BayerSturmfels}) to our setting.
\begin{Theorem} \label{thm:Jresol}
The quotient cell complex $\Del(L(G)) / L(G)$ supports a $(C^1(G, \ZZ)/{\Lambda})$-graded minimal free resolution for $\JJ_G$.
\end{Theorem}
 Here $\Lambda$ is the image of the (unrestricted) coboundary map $d : C^0(G,\ZZ) \to C^1(G,\ZZ)$ (see Remark~\ref{rmk: orientvslattice}).

\medskip

\subsection{Gr\"obner relation between $\JJ_G$ and $\OO_G^q$} \label{sec:grobrel}

Recall that the hyperplane arrangement $\H_G^{q,c}$ is naturally sitting inside $C^0(G, \RR)$, and the Delaunay decomposition $\Del(L(G))$ is an infinite hyperplane arrangement naturally sitting inside $C_{\Oc}^1(G, \RR)$. The obvious map between these ambient spaces is the (restricted) coboundary map $d_\Oc \colon C^0(G, \RR) \rightarrow C_{\Oc}^1(G, \RR)$. As we will see, this map relates the corresponding hyperplane arrangements and cell complexes, and this relation translates into precise algebraic relations between $\JJ_G$ and $\OO_G^q$. 

\medskip

First note that $\Ker(d) =\Ker(d_\Oc)$ is the $1$-dimensional space of constant functions on $V(G)$, and we have
\[
L(G)_\RR = \Image (d_\Oc) \cong C^0(G, \RR) / \Ker(d)  \cong C^0(G, \RR) \cap (\Ker(d))^{\perp} \ .
\]
Let $e \in \EE(G)$. Under the induced isomorphism $d_\Oc \colon C^0(G, \RR) \cap (\Ker(d))^{\perp} \xrightarrow{\sim} L(G)_\RR$, the hyperplane 
\[
\H_{e}|_{(\Ker(d))^{\perp}} = \{f \in C^0(G,\RR): (d f) (e)=0 \} \cap (\Ker(d))^{\perp} 
\]
is mapped to the hyperplane 
\[
\G_e=\{a \in L(G)_\RR: \varphi_e(a) = 0\} \ ,
\]
where $\varphi_e$ is the restriction of the functional $e=e^{\ast\ast} \in C_1(G, \ZZ)$ to $L(G)_\RR$. 
By Example~\ref{ex:WU}, Proposition~\ref{thm:totunim}(iii), and Remark~\ref{rmk:delon}(ii), the hyperplanes $\G_e$ are precisely the hyperplanes passing through the origin in $\Del(L(G))$.  

Recall from \S\ref{sec:BG} that the hyperplane arrangement $\H_G^{q,c}$ has another hyperplane defined by 
\begin{equation} \label{eq:lastplane}
(\H^{q,c})'|_{(\Ker(d))^{\perp}} =\{f\in C^0(G, \RR): \sum_{v \ne q}{f(v)}=c \}  \cap (\Ker(d))^{\perp} \ .
\end{equation}

The real vector space $L(G)_\RR$ is spanned by $\{d_\Oc(\chi_v): v \ne q\}$. Under the induced isomorphism $d_\Oc \colon C^0(G, \RR) \cap (\Ker(d_\Oc))^{\perp} \xrightarrow{\sim} L(G)_\RR$, the hyperplane \eqref{eq:lastplane} is mapped to the affine hyperplane
\[
\G^{q,c}=\{a \in C^1(G, \RR): a = \sum_{v \ne q}{f(v) d_\Oc(\chi_v)} \text{ with } \sum_{v \ne q}{f(v)}=c\} \ .
\]
This is a hyperplane passing through all points $\{c \cdot d_\Oc(\chi_v): v \ne q\}$. 

\medskip

We denote the restriction of the arrangement $\{\G_e\}_{e \in \EE(G)}$ to the affine hyperplane $\G^{q,c}$ by $\G_G^{q,c}$. It follows that $\G_G^{q,c}$,  upto a linear transformation, coincides with the arrangement $\H_G^{q,c}$, and therefore its bounded complex, which we denote by $\A_G^{q,c}$, may be identified with $\B_G^{q,c}$.

\medskip

Next we show that these geometric considerations nicely relate the labeling of $\B_G^{q,c}$ by monomials (described in \S\ref{sec:labelBG}) with the natural  labeling of $\A_G^{q,c}$ induced by $\Del(L(G))$ (described in \S\ref{sec:labelDel}).  For this purpose, we will see that it is most natural to assume $0<c<1$. With this assumption, if the hyperplane $\G^{q,c}$ intersects a Delaunay cell $C$, then ${C}$ must contain the origin. By the least common multiple labeling rule, this means that all such cells $C$ have monomial labels in $\Sb$. 

\medskip

To concretely describe these induced monomial labels, it suffices to find the labels of the vertices in $\G_G^{q,c}$ induced from the labels of the rays in the central hyperplane arrangement  
$\{ \G_e: e \in \EE(G)\}$. These rays correspond to bonds $d_\Oc(\chi_B)$ for $B \subset V(G)$ (see \S\ref{sec:cutlattice}). Such a ray intersects $\G^{q,c}$ if and only if for some real number $t>0$ we have
\[
t d_\Oc(\chi_B) = \sum_{v \ne q}{f(v)d_\Oc(\chi_v)} \ ,
\]
or equivalently
\[
d_\Oc(t\chi_B-\sum_{v \ne q}{f(v)\chi_v})=0 \ .
\]
Since the kernel of $d_\Oc$ consists of constant functions we must have 
\begin{equation} \label{eq:eval}
t\sum_{v \in B}{\chi_v}-\sum_{v \ne q}{f(v)\chi_v} = k \sum_{v}{\chi_v}
\end{equation}
for some constant $k \in \RR$. 

\medskip

We claim that $q \not\in B$. 
Indeed, if $q \in B$, then evaluating \eqref{eq:eval} at $q$ we obtain $k=t$ and therefore 
\[t\sum_{v \in B^c}{\chi_v}=-\sum_{v \ne q}{f(v)\chi_v}  \ .\] 
This implies that $f(v)=-t <0$ for $v \in B^c$ and $f(v)=0$ for $v \in B \backslash \{q\}$. But this is impossible because $\sum_{v \ne q}{f(v)}=c$ by assumption.

\medskip

Since $q \not \in B$, by evaluating \eqref{eq:eval} at $q$ we obtain $k=0$ and therefore 
\[
t\sum_{v \in B}{\chi_v}=\sum_{v \ne q}{f(v)\chi_v}   \ ,
\]
which implies that $f(v)=t$ for $v \in B$ and $f(v) = 0$ for $v \in B^c \backslash \{q\}$. Since $\sum_{v \ne q}{f(v)}=c$, we must have $t=\frac{c}{|B|}$. 
Conversely, for any nonempty subset $B \subset V(G) \backslash \{q\}$, the ray corresponding to the simple cut $d_\Oc(\chi_B)$ intersects $\G^{q,c}$ at the point $\frac{c}{|B|}d_\Oc(\chi_B)$. If we fix $0<c<1$, then we always have $0<\frac{c}{|B|} <1$ which means that the point of intersection belongs to a cell in $\Del(L(G))$ containing the origin. We summarize these observations in the following proposition.

\begin{Proposition} \label{prop:rays}
Let $\emptyset \ne B \subset V(G)$. The ray corresponding to the bond $d_\Oc(\chi_B)$ intersects $\G^{q,c}$ if and only if $q \not \in B$. If $0<c<1$, then the point of intersection belongs to a cell in $\Del(L(G))$ containing the origin.
\end{Proposition}

The vertices of $\A_G^{q,c}$ are the points of intersections with these rays. For each vertex of $\A_G^{q,c}$ we may assign the label corresponding to the $1$-dimensional cell of $\Del(L(G))$ containing that vertex. If we assume $0<c<1$, this is a (non-Laurent) monomial label that coincides with the labeling rule for $\B_G^{q,c}$ described in \S\ref{sec:labelBG}. From this point of view, it is straightforward to describe these labels combinatorially.

\begin{Lemma} \label{lem:labels}
For any $A \subsetneq V(G)$ with $q \in A$ the following holds.
\begin{itemize}
\item[(i)] The label of the point $d_\Oc(\chi_{A^c})$ in the labeled complex $\Del(L(G))$ is
\[ \bb(d_\Oc(\chi_{A^c})) = \frac{\prod_{e \in \EE(A^c, A)}{y_e} }{\prod_{e \in \EE(A,A^c)}{y_e} } \ .\]
\item[(ii)] For $0<c<1$, the induced label on the vertex $\A_G^{q,c}$ corresponding to the bond $d_\Oc(\chi_{A^c})$ is 
\[
\prod_{e \in \EE(A^c, A)}{y_e} \ . 
\]
\end{itemize}
\end{Lemma}

\begin{proof}
(i) By \eqref{eq:blabel} we have
\[
\begin{aligned}
\bb(d_\Oc(\chi_{A^c}))&=\prod_{e \in \EE(G)} { {y_{e}^{d(\chi_{A^c})(e)}} } \\
&=\prod_{e \in \EE(G)}  { {y_{e}^{\chi_{A^c}(e_+)-\chi_{A^c}(e_-)} } } \\ 
&= \frac{\prod_{e \in \EE(A^c, A)}{y_e} }{\prod_{e \in \EE(A,A^c)}{y_e} } \ .
\end{aligned}
\]

(ii) The label of the origin is $\bb(\mathbf{0})=1$. Therefore, by the least common multiple construction, the label of the $1$-dimensional cell $\{\mathbf{0}, d_\Oc(\chi_{A^c}) \}$ in $\Del(L(G))$ is $\prod_{e \in \EE(A^c, A)}{y_e}$. The result now follows from Proposition~\ref{prop:rays}.
\end{proof}

\medskip

Since the labeled complex $\A_G^{q,c}$ (for $0 <c<1$) coincides with the labeled complex $\B_G^{q,c}$, we might as well think of the ideal $\OO_G^q$ as constructed from $\A_G^{q,c}$. The advantage of this point of view is a precise Gr\"obner relation between $\OO_G^q$ and $\JJ_G$ coming from the described relation of $\A_G^{q,c}$ and $\Del(L(G))$. 

\begin{Lemma}\label{lem:bij}
Intersection of cells in $\Del(L(G))$ with the hyperplane $\G^{q,c}$ induces a bijection between $(i+1)$-dimensional cells of $\Del(L(G)) / L(G)$ and $i$-dimensional cells of $\A_G^{q,c}$ for all $0 \leq i \leq n-2$. 
\end{Lemma}
\begin{proof}
It suffices to only consider cells in $\Del(L(G))$ containing the origin; all other cells in $\Del(L(G))$ can be obtained by translating such cells by $L(G)$. 
The primitive (or indecomposable) elements of $L(G)$ correspond to bonds (see \S\ref{sec:cutlattice}). Therefore the vertex set of any cell in $\Del(L(G))$ containing the origin is of the form $\{\mathbf{0}\} \cup P$ for some $P \subset \{d_\Oc(\chi_B) : \emptyset \ne B \subset V(G)\}$. Since $d_\Oc(\chi_{B^c})=-d_\Oc(\chi_B)$, it suffices to restrict our attention to the case where $P \subset \{d_\Oc(\chi_B) : \emptyset \ne B \subset V(G), q \not \in B\}$. 
 By Proposition~\ref{prop:rays}, these are precisely those cells that have nonempty intersection with $\G^{q,c}$.
\end{proof}

\begin{Proposition} \label{prop:grobrelation}
\begin{itemize}
\item[]
\item[(i)] A generating set for the ideal $\JJ_G$ is 
\[ \left\{ {\prod_{e \in \EE(A^c, A)}{y_e} }-{\prod_{e \in \EE(A,A^c)}{y_e} } : A \subsetneq V(G), q \in A \right\} .\]
If we consider only those subsets $A$ of $V(G)$ such that both $G[A]$ and $G[A^c]$ are connected, then we have a minimal generating set for $\JJ_G$.
\item[(ii)] The minimal generating set in part (i) is also a Gr\"obner basis with respect to {\em any} term order (i.e. is a universal Gr\"obner basis).
\item[(iii)] A minimal generating set for the ideal $\OO_G^q$ is 
\[ \left\{{\prod_{e \in \EE(A^c, A)}{y_e} } : A \subsetneq V(G), q \in A, G[A] \text{ and } G[A^c] \text{ are connected} \right\} .\]
\item[(iv)] $\OO_G^q$ is the initial ideal of $\JJ_G$ with respect to any term order $\prec_q$ with the property that 
\[
{\prod_{e \in \EE(A,A^c)}{y_e} }  \ \prec_q {\prod_{e \in \EE(A^c, A)}{y_e} } 
\]
for every $A \subsetneq V(G)$ with $q \in A$ such that both $G[A]$ and $G[A^c]$ are connected.
\end{itemize}

\end{Proposition}
\begin{proof}

(i) It follows from the discussion in \S\ref{sec:res}, Theorem~\ref{thm:Jresol} and \cite[proof of Theorem~3.2]{BayerSturmfels} that a minimal generating set for $\JJ_G$ is given by binomials 
\[
\frac{\mb_F}{\mb_{F'}} - \frac{\mb_F}{\mb_{\mathbf{0}}} \ ,
\]
where $F$ is in a fundamental set of representatives of $1$-cells in $\Del(L(G))$ connecting $\mathbf{0}$ to  $F=d_\Oc(\chi_{A^c})$ for $A \subsetneq V(G)$ and $q \in A$.

By Lemma~\ref{lem:labels}(i), we have 
\[
\mb_{F'} = \bb(d_\Oc(\chi_{A^c})) = \frac{\prod_{e \in \EE(A^c, A)}{y_e} }{\prod_{e \in \EE(A,A^c)}{y_e}} ,
\quad
\mb_{\mathbf{0}} = 1 ,
\]\[
\mb_{F} = \lcm(\mb_{F'}, \mb_{\mathbf{0}}) = \prod_{e \in \EE(A^c, A)}{y_e} 
\]
and therefore
\[
\frac{\mb_F}{\mb_{F'}} - \frac{\mb_F}{\mb_{\mathbf{0}}}= \prod_{e \in \EE(A, A^c)}{y_e}-\prod_{e \in \EE(A^c, A)}{y_e} \ .
\]
The rest of part (i) is immediate.

\medskip

(ii) follows from the general fact that in any Lawrence ideal, a minimal binomial generating set is a Gr\"obner basis with respect to any term order (\cite[Theorem~7.1]{SturmfelsGrobnerConvex}). In our concrete situation, one can also easily verify (as in the proof of Theorem~\ref{thm:Cori} given in \cite[Theorem~5.1]{FarbodFatemeh}) that the $S$-polynomial of the two binomials corresponding to the cuts $(A,A^c)$ and $(B,B^c)$ can be reduced to zero by the binomials corresponding to the cuts $(A \backslash B, (A \backslash B)^c )$
and $(B \backslash A, (B \backslash A)^c )$. 

\medskip

(iii) It follows from the discussion in \S\ref{sec:res}, Theorem~\ref{thm:bounded}, and the fact that the labeled cell complex $\A_G^{q,c}$ coincides with the labeled complex $\B_G^{q,c}$, that a minimal generating set for $\OO_G^q$ is given by the monomials $\mb_F$ as $F$ varies over the vertices of the bounded cell complex $\A_G^{q,c}$. By Proposition~\ref{prop:rays} and Lemma~\ref{lem:labels}(ii), these labels are precisely of the form
\[
\prod_{e \in \EE(A^c, A)}{y_e}
\]
for $A \subsetneq V(G)$ with  $q \in A$ such that the edges between $(A,A^c)$ form a bond.

\medskip

 (iv) follows from (ii) and (iii).
\end{proof}

\medskip

\begin{Example}\label{exam:1}

Consider the graph $G$ depicted in Figure~\ref{fig:graph} with the fixed orientation $\Oc$. Let $q$ be the distinguished (red) vertex at the bottom. 
Acyclic partial orientations of $G$ with unique source at $q$ are depicted in Figures~\ref{fig:2partition}--\ref{fig:4partition}.

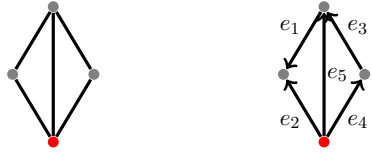
\begin{figure}[h]

\begin{center}

\begin{tikzpicture} [scale = .18, very thick = 10mm]

  \node (n4) at (4,1)  [Cred] {};
  \node (n1) at (4,11) [Cgray] {};
  \node (n2) at (1,6)  [Cgray] {};
  \node (n3) at (7,6)  [Cgray] {};
  \foreach \from/\to in {n4/n2,n1/n3}
    \draw[] (\from) -- (\to);
\foreach \from/\to in {n2/n1,n4/n3,n1/n4}
    \draw[] (\from) -- (\to);

    \node (n4) at (24,1)  [Cred] {};  \node (m4) at (25,6)  [Cwhite] {$e_5$};
  \node (n1) at (24,11) [Cgray] {};   \node (m1) at (26.5,9.5) [Cwhite] {$e_3$};
  \node (n2) at (21,6)  [Cgray] {};  \node (m1) at (21.5,9.5) [Cwhite] {$e_1$};
  \node (n3) at (27,6)  [Cgray] {};  \node (m1) at (26.5,2.5) [Cwhite] {$e_4$}; \node (m1) at (21.5,2.5) [Cwhite] {$e_2$};
 \foreach \from/\to in {n4/n2,n3/n1,n1/n2,n4/n3,n4/n1}
  \draw[black][->] (\from) -- (\to);

\end{tikzpicture}
\caption{Graph $G$ and a fixed orientation $\Oc$}
\label{fig:graph}
\end{center}

\end{figure}

\begin{figure}[h]

\begin{center}

\begin{tikzpicture} [scale = .17, very thick = 10mm]

  \node (n4) at (4,1)  [Cred] {};
  \node (n1) at (4,11) [Cgray] {};
  \node (n2) at (1,6)  [Cgray] {};
  \node (n3) at (7,6)  [Cgray] {};
  \foreach \from/\to in {n4/n2,n1/n3}
    \draw[] (\from) -- (\to);
\foreach \from/\to in {n2/n1,n4/n3,n4/n1}
    \draw[blue][->] (\from) -- (\to);

    \node (n4) at (14,1)  [Cred] {};
  \node (n1) at (14,11) [Cgray] {};
  \node (n2) at (11,6)  [Cgray] {};
  \node (n3) at (17,6)  [Cgray] {};
  \foreach \from/\to in {n4/n2,n3/n1,n4/n1}
   \draw[blue][->] (\from) -- (\to);
     \foreach \from/\to in {n1/n2,n4/n3}
    \draw[] (\from) -- (\to);

    \node (n4) at (24,1)  [Cred] {};
  \node (n1) at (24,11) [Cgray] {};
  \node (n2) at (21,6)  [Cgray] {};
  \node (n3) at (27,6)  [Cgray] {};
  \foreach \from/\to in {n1/n2,n1/n3}
    \draw[] (\from) -- (\to);
    \foreach \from/\to in {n4/n2,n4/n3,n4/n1}
    \draw[blue][->] (\from) -- (\to);

  \node (n4) at (34,1)  [Cred] {};
  \node (n1) at (34,11) [Cgray] {};
  \node (n2) at (31,6)  [Cgray] {};
  \node (n3) at (37,6)  [Cgray] {};
   \foreach \from/\to in {n4/n2,n4/n3}
    \draw[] (\from) -- (\to);
    \foreach \from/\to in {n2/n1,n3/n1,n4/n1}
    \draw[blue][->] (\from) -- (\to);

    \node (n4) at (44,1)  [Cred] {};
  \node (n1) at (44,11) [Cgray] {};
  \node (n2) at (41,6)  [Cgray] {};
  \node (n3) at (47,6)  [Cgray] {};
 \foreach \from/\to in {n1/n3,n4/n3,n4/n1}
    \draw[] (\from) -- (\to);
    \foreach \from/\to in {n1/n2,n4/n2}
    \draw[blue][->] (\from) -- (\to);

    \node (n4) at (54,1)  [Cred] {};
  \node (n1) at (54,11) [Cgray] {};
  \node (n2) at (51,6)  [Cgray] {};
  \node (n3) at (57,6)  [Cgray] {};
 \foreach \from/\to in {n1/n2,n4/n2,n4/n1}
    \draw[] (\from) -- (\to);
    \foreach \from/\to in {n1/n3,n4/n3}
    \draw[blue][->] (\from) -- (\to);

\end{tikzpicture}

\caption{Acyclic partial orientations with $2$ components}
\label{fig:2partition}
\end{center}
\medskip

\begin{center}

\begin{tikzpicture} [scale = .17, very thick = 10mm]

 \node (n4) at (-6,13)  [Cred] {};
  \node (n1) at (-6,23) [Cgray] {};
  \node (n2) at (-9,18)  [Cgray] {};
  \node (n3) at (-3,18)  [Cgray] {};
\foreach \from/\to in {n2/n4}
   \draw[] (\from) -- (\to);
 \foreach \from/\to in {n4/n3,n2/n1,n3/n1,n4/n1}
    \draw[blue][->] (\from) -- (\to);

 \node (n4) at (4,13)  [Cred] {};
  \node (n1) at (4,23) [Cgray] {};
  \node (n2) at (1,18)  [Cgray] {};
  \node (n3) at (7,18)  [Cgray] {};
 \foreach \from/\to in {n3/n4}
   \draw[] (\from) -- (\to);
 \foreach \from/\to in {n4/n2,n2/n1,n3/n1,n4/n1}
    \draw[blue][->] (\from) -- (\to);

 \node (n4) at (14,13)  [Cred] {};
  \node (n1) at (14,23) [Cgray] {};
  \node (n2) at (11,18)  [Cgray] {};
  \node (n3) at (17,18)  [Cgray] {};
  \foreach \from/\to in {n1/n3}
    \draw[] (\from) -- (\to);
 \foreach \from/\to in {n4/n2,n4/n3, n2/n1,n4/n1}
    \draw[blue][->] (\from) -- (\to);

      \node (n4) at (24,13)  [Cred] {};
  \node (n1) at (24,23) [Cgray] {};
  \node (n2) at (21,18)  [Cgray] {};
  \node (n3) at (27,18)  [Cgray] {};
   \foreach \from/\to in {n1/n2}
    \draw[] (\from) -- (\to);
    \foreach \from/\to in {n1/n3, n3/n4,n2/n4,n1/n4}
    \draw[blue][<-] (\from) -- (\to);
\end{tikzpicture}

\vspace{-.5cm}



\begin{tikzpicture}  [scale = .18, very thick = 10mm]
\node (n4) at (-16,13)  [Cred] {};
  \node (n1) at (-16,23) [Cgray] {};
  \node (n2) at (-19,18)  [Cgray] {};
  \node (n3) at (-13,18)  [Cgray] {};
\foreach \from/\to in {n2/n1,n2/n4}
    \draw[blue][<-] (\from) -- (\to);

 \foreach \from/\to in {n4/n3,n1/n3}
    \draw[blue][->] (\from) -- (\to);

\foreach \from/\to in {n1/n4}
    \draw[] (\from) -- (\to);


 \node (n4) at (-6,13)  [Cred] {};
  \node (n1) at (-6,23) [Cgray] {};
  \node (n2) at (-9,18)  [Cgray] {};
  \node (n3) at (-3,18)  [Cgray] {};
\foreach \from/\to in {n2/n1,n2/n4,n1/n4}
    \draw[blue][<-] (\from) -- (\to);

 \foreach \from/\to in {n4/n3}
    \draw[blue][->] (\from) -- (\to);

\foreach \from/\to in {n1/n3}
    \draw[] (\from) -- (\to);
 \node (n4) at (4,13)  [Cred] {};

  \node (n1) at (4,23) [Cgray] {};
  \node (n2) at (1,18)  [Cgray] {};
  \node (n3) at (7,18)  [Cgray] {};
  \foreach \from/\to in {n2/n4}
    \draw[blue][<-] (\from) -- (\to);
\foreach \from/\to in {n2/n1}
    \draw[] (\from) -- (\to);
 \foreach \from/\to in {n4/n3,n1/n3,n4/n1}
    \draw[blue][->] (\from) -- (\to);

 \node (n4) at (14,13)  [Cred] {};
  \node (n1) at (14,23) [Cgray] {};

  \node (n2) at (11,18)  [Cgray] {};
  \node (n3) at (17,18)  [Cgray] {};

  \foreach \from/\to in {n1/n2,n3/n1}
    \draw[blue][<-] (\from) -- (\to);

 \foreach \from/\to in {n4/n3,n4/n1}
    \draw[blue][->] (\from) -- (\to);
 \foreach \from/\to in {n2/n4}
    \draw[] (\from) -- (\to);

      \node (n4) at (24,13)  [Cred] {};
  \node (n1) at (24,23) [Cgray] {};

  \node (n2) at (21,18)  [Cgray] {};
  \node (n3) at (27,18)  [Cgray] {};
  \foreach \from/\to in {n2/n1,n1/n3}
    \draw[blue][<-] (\from) -- (\to);

 \foreach \from/\to in {n4/n2,n4/n1}
    \draw[blue][->] (\from) -- (\to);
 \foreach \from/\to in {n3/n4}
    \draw[] (\from) -- (\to);

\end{tikzpicture}

\caption{Acyclic partial orientations with $3$ components}
\label{fig:3partition}
\end{center}

\medskip
\begin{center}

\begin{tikzpicture}  [scale = .17, very thick = 10mm]

 \node (n4) at (4,13)  [Cred] {};
  \node (n1) at (4,23) [Cgray] {};
  \node (n2) at (1,18)  [Cgray] {};
  \node (n3) at (7,18)  [Cgray] {};
  \foreach \from/\to in {n2/n1,n2/n4}
    \draw[blue][<-] (\from) -- (\to);
 \foreach \from/\to in {n4/n3,n3/n1,n4/n1}
    \draw[blue][->] (\from) -- (\to);

 \node (n4) at (14,13)  [Cred] {};
  \node (n1) at (14,23) [Cgray] {};
  \node (n2) at (11,18)  [Cgray] {};
  \node (n3) at (17,18)  [Cgray] {};
  \foreach \from/\to in {n1/n2,n3/n1}
    \draw[blue][<-] (\from) -- (\to);
 \foreach \from/\to in {n4/n3,n4/n2,n4/n1}
    \draw[blue][->] (\from) -- (\to);

      \node (n4) at (24,13)  [Cred] {};
  \node (n1) at (24,23) [Cgray] {};
  \node (n2) at (21,18)  [Cgray] {};
  \node (n3) at (27,18)  [Cgray] {};
  \foreach \from/\to in {n1/n2,n1/n3, n3/n4,n2/n4,n1/n4}
    \draw[blue][<-] (\from) -- (\to);

      \node (n4) at (34,13)  [Cred] {};
  \node (n1) at (34,23) [Cgray] {};
  \node (n2) at (31,18)  [Cgray] {};
  \node (n3) at (37,18)  [Cgray] {};
  \foreach \from/\to in {n2/n1,n3/n1, n3/n4,n2/n4,n1/n4}
    \draw[blue][<-] (\from) -- (\to);
\end{tikzpicture}
\caption{Acyclic partial orientations with $4$ components}
\label{fig:4partition}
\end{center}


\end{figure}


\medskip

Consider the arrangement $\H'_G=\{\H_{e_1},\ldots, \H_{e_5}\}$.  The graphic arrangement $\H_G^{q,c}$ (for some $c > 0$) is two-dimensional and is depicted in Figure~\ref{fig:arrangement}. Its bounded complex $\B_G^{q,c}$ is the bounded part of this figure. Recall that the graphic arrangement ``lives in'' $C^0(G,\RR)$, which may be identified with $\RR^4$ after fixing a labeling of the vertices. For each hyperplane labeled $\H_e$, the small arrow next to it denotes the side where $(d f)(e) >0$. The hyperplane $\H_{\bar{e}}$ coincides with $\H_e$, but its arrow will be reversed. We have also labeled the $0$-cells according to \eqref{eq:BLabels}. 
\setlength{\unitlength}{1.1pt}
\begin{figure}[h!]
\begin{center}
    \begin{picture}(100,195)(0,-55)

     \thicklines
     \put(60,0){\circle*{4}}
      \put(15,-10){$y_{\bar{e}_1} y_{e_4} y_{e_5}$}

 \put(26,66){\circle*{4}}
      \put(-23,65){$y_{e_2} y_{e_3} y_{e_5}$}

      \put(160,0){\circle*{4}}
      \put(130,-10){$y_{\bar{e}_3} y_{e_4}$}

 \put(-40,0){\circle*{4}}
      \put(-95,-10){$y_{\bar{e}_1} y_{e_3} y_{e_5}$}

      \put(60,100){\circle*{4}}
      \put(24,98){$y_{e_1} y_{e_2}$}

      \put(60,50){\circle*{4}}
      \put(15,45){$y_{e_2} y_{e_4} y_{e_5}$}

      \put(60,0){\line(0,-1){30}}
\put(60,50){\line(2,-1){100}}  \put(60,50){\line(-2,1){65}}     \put(60,100){\line(0,1){30}}
      \put(60,-20){\vector(-1,0){10}}
      \put(56,-43){$H_{e_3}$}

 \put(175,0){\vector(0,1){10}}
      \put(180,10){$H_{e_2}$}

 \put(-50,-10){\vector(1,-1){10}}
      \put(-60,-43){$H_{e_4}$}

      \put(30,130){\line(1,-1){20}}
      \put(160,0){\line(1,-1){20}}
      \put(170,-10){\vector(-2,-3){7}}
      \put(193,-40){$H_{e_5}$}

  \put(14,73){\vector(2,3){7}}
      \put(-30,80){$H_{e_1}$}

      \thicklines
      \put(-40,0){\line(1,0){230}}
      \put(-40,0){\line(-1,0){30}}
      \put(60,0){\line(-1,0){80}}
      \put(60,0){\line(0,1){130}}

\put(160,0){\line(3,-2){35}}
\put(-40,0){\line(1,1){130}} \put(-40,0){\line(-1,-1){30}}
\put(160,0){\line(-1,1){130}} \put(160,0){\line(1,-1){30}}

\put(-40,-10){$\pb_1$}
\put(64,-10){$\pb_2$}
\put(155,10){$\pb_3$}
\put(63,55){$\pb_4$}
\put(21,74){$\pb_5$}
\put(65,98){$\pb_6$}

    \end{picture}
    \caption{$\H_G^q$, $\B_G^q$, and the monomial labels on the vertices}
    \label{fig:arrangement}
\end{center}
\end{figure}
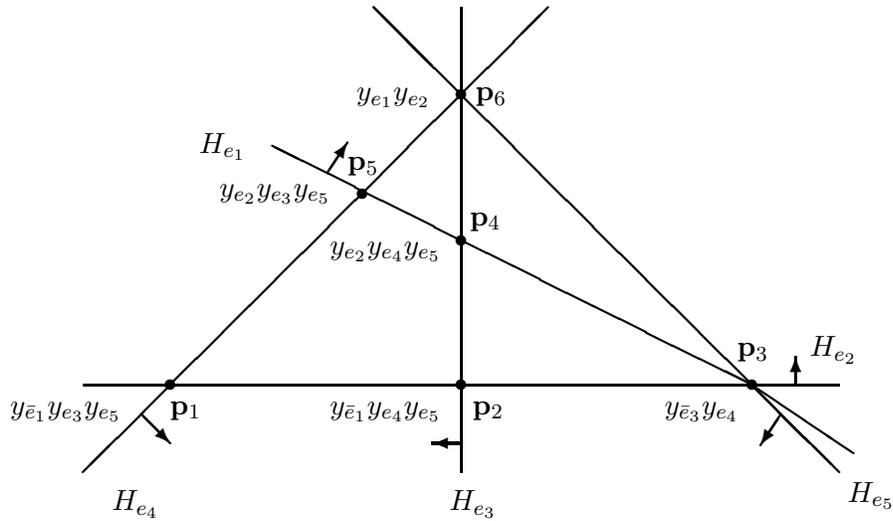

The polynomial ring $\Sb$ has $10$ variables:
\[\{y_e, y_{\bar{e}}: e \in \Oc\} = \{y_{e_1},y_{e_2},y_{e_3},y_{e_4},y_{e_5} ; y_{\bar{e}_1},y_{\bar{e}_2},y_{\bar{e}_3},y_{\bar{e}_4},y_{\bar{e}_5}\} \ .\] 

 By Theorem~\ref{thm:bounded}, the associated oriented matroid ideal $\OO_G^q$ is minimally generated by the labels of the $0$-cells:
\begin{equation} \label{eq:OGex}
\OO_G^q=\langle y_{\bar{e}_1} y_{e_4} y_{e_{5}}, y_{e_2}y_{e_3}y_{e_5}, y_{\bar{e}_3} y_{e_4}, y_{\bar{e}_1} y_{e_3}y_{e_5},
y_{e_1} y_{e_2},y_{e_2}y_{e_4}y_{e_5}\rangle \ .
\end{equation}
Note that the indices appearing in the minimal generating set correspond precisely to the oriented edges leaving the connected partition containing $q$ (i.e. the blue edges in Figure~\ref{fig:2partition}). This is what we expect by Proposition~\ref{prop:grobrelation}(iii). 

The lattice of integral cuts $L(G)$ is $3$-dimensional. Instead of drawing it, we may directly write a minimal generating set for $\JJ_G$ using Proposition~\ref{prop:grobrelation}(i):
\[
\JJ_{G} =\langle y_{\bar{e}_1} y_{e_4} y_{e_{5}}-y_{{e}_1} y_{\bar{e}_4} y_{\bar{e}_{5}}, 
y_{e_2}y_{e_3}y_{e_5}-y_{\bar{e}_2}y_{\bar{e}_3}y_{\bar{e}_5}, 
y_{\bar{e}_3}y_{e_4}-y_{{e}_3}y_{\bar{e}_4}, y_{\bar{e}_1} y_{e_3}y_{e_5}-y_{{e}_1} y_{\bar{e}_3}y_{\bar{e}_5},\]
\[
y_{e_1} y_{e_2}-y_{\bar{e}_1} y_{\bar{e}_2},y_{e_2}y_{e_4}y_{e_5}-y_{\bar{e}_2}y_{\bar{e}_4}y_{\bar{e}_5}\rangle\ .
\] 
The first term in each binomial is the dominant term for the term order $\prec_q$. The bounded complex $\B_G^q$ has six $0$-cells $\{\pb_1, \ldots , \pb_6\}$, nine $1$-cells $\{E_1, \ldots , E_9\}$, and four $2$-cells $\{F_1, \ldots , F_4\}$. These numbers correspond to the acyclic orientations of Figure~\ref{fig:2partition}, Figure~\ref{fig:3partition}, and Figure~\ref{fig:4partition}, as well as the Betti numbers of $\OO_G^q$ and $\JJ_G$.
Moreover, $\B_G^q$  supports a minimal free resolution for $\OO_G^q$. To explicitly describe this minimal resolution, let 
\[
E_1=\{\pb_1,\pb_2\} , \quad E_2=\{\pb_2,\pb_3\}, \quad E_3=\{\pb_1,\pb_5\}, \quad E_4=\{\pb_2,\pb_4\}, \quad E_5=\{\pb_3,\pb_4\} 
\]
\[
E_6=\{\pb_4,\pb_5\}, \quad E_7=\{\pb_5,\pb_6\} , \quad E_8=\{\pb_4,\pb_6\} , \quad E_9=\{\pb_3,\pb_6\} \ ,
\]
\[
F_1=\{\pb_1,\pb_2,\pb_4,\pb_5\}, \quad F_2=\{\pb_2,\pb_3,\pb_4\} , \quad F_3=\{\pb_4,\pb_5,\pb_6\}, \quad F_4=\{\pb_3,\pb_4,\pb_6\} \ .
\]
We extend the labeling on the vertices to the whole $\B_G^q$ by the least common multiple construction. For example,  
\[
\mb_{E_2}=y_{\bar{e}_1}y_{\bar{e}_3} y_{e_4} y_{e_5}, \ 
\mb_{E_4}=y_{\bar{e}_1} y_{e_2} y_{e_4} y_{e_5}, \ 
\mb_{E_5}=y_{e_2} y_{\bar{e}_3} y_{e_4} y_{e_5} , \ 
\mb_{E_6}=y_{e_2}y_{e_3} y_{e_4} y_{e_5}, \ 
\]
\[
\mb_{F_2}= y_{\bar{e}_1} y_{e_2} y_{\bar{e}_3} y_{e_4}y_{e_5}\ .
\] 
Then the minimal resolution of $\OO_G^q$ is as follows.
\[
 0 \rightarrow \bigoplus_{i=1}^4\Sb(-\mb_{F_i}) \xrightarrow{\partial_2} \bigoplus_{i=1}^9\Sb(-\mb_{E_i}) \xrightarrow{\partial_1} \bigoplus_{i=1}^6\Sb(-\mb_{\pb_i}) \xrightarrow{\partial_0} \Sb \twoheadrightarrow \Sb /\OO_G^q \ .
\]
As usual, assume $[F]$ denotes the generator of $\Sb(-\mb_F)$. The homogenized differential operator of the cell complex $(\partial_0, \partial_1, \partial_2)$ is as described in \eqref{eq:differntials}. For example 

\[
\partial_0([\pb_i])= \mb_{\pb_i} = \mb(\pb_i) \ ,
\]

\[
\partial_1([E_6])= \frac{y_{e_2}y_{e_3} y_{e_4} y_{e_5}}{y_{e_2} y_{e_4} y_{e_5}}[\pb_4] -\frac{y_{e_2}y_{e_3} y_{e_4} y_{e_5}}{y_{e_2}y_{e_3} y_{e_5}}[\pb_5] = y_{e_3}[\pb_4] - y_{e_4}[\pb_4]\ ,
\]

\medskip

\[
\begin{aligned}
\partial_2([F_2]) &= \frac{y_{\bar{e}_1} y_{e_2} y_{\bar{e}_3} y_{e_4}y_{e_5}}{y_{\bar{e}_1}y_{\bar{e}_3} y_{e_4} y_{e_5}}[E_2] -\frac{y_{\bar{e}_1} y_{e_2} y_{\bar{e}_3} y_{e_4}y_{e_5}}{y_{\bar{e}_1} y_{e_2} y_{e_4} y_{e_5}}[E_4]+\frac{y_{\bar{e}_1} y_{e_2} y_{\bar{e}_3} y_{e_4}y_{e_5}}{y_{e_2} y_{\bar{e}_3} y_{e_4} y_{e_5}}[E_5] \\
&= y_{e_2}[E_2] - y_{\bar{e}_3}[E_4] + y_{\bar{e}_1}[E_5] \ .
\end{aligned}
\]

Although $\JJ_G$ has the same Betti table as $\OO_G^q$, it is not possible to read the minimal free resolution for $\JJ_G$ directly from $\B_G^q$; one really needs to consider the cell decomposition of the torus $L(G)_{\RR}/L(G)$. 

\end{Example}

\begin{Example} \label{ex:CutResol}

Consider the graph $K_3$ with a fixed orientation as in Figure~\ref{fig:K30}.

The lattice of integral cuts $L(G)$ is two-dimensional and is depicted in Figure~\ref{fig:CutLattice}. This picture should be compared with Figure~\ref{fig:PrinLattice} (see Remark~\ref{rmk:isometry}). This lattice ``lives in'' $C_{\Oc}^1(G,\RR)=\span\{e_1^*,e_2^*,e_3^*\} \cong \RR^3$. In the picture $a_1=d_{\Oc}(\chi_{u_1})=e^\ast_2-e^\ast_3$, $a_2=d_{\Oc}(\chi_{u_2})=e^\ast_1-e^\ast_2$, and $a_3=d_{\Oc}(\chi_{u_3})=e^\ast_3-e^\ast_1$.

\begin{figure}[h!]

\begin{center}
\begin{tikzpicture} [scale = .30, very thick = 20mm]

\node (n42) at (20,1)  [Cblack] {};
\node (n41) at (14,1)  [Cblack] {}; 
\node (n43) at (26,1)  [Cblack] {};
\node (n44) at (32,1)  [Cblack] {};
\node (n11) at (14,11) [Cblack] {};  
\node (n12) at (20,11) [Cblack] {};  
\node (n13) at (26,11) [Cblack] {};  
\node (n14) at (32,11) [Cblack] {};
\node (n21) at (11,6)  [Cblack] {}; 
\node (n22) at (17,6)  [Cblack] {}; 
\node (n23) at (23,6)  [Cblack] {}; 
\node (n24) at (29,6)  [Cblack] {};  
\node (n25) at (35,6)  [Cblack] {};
\node (n4) at (14,1)  [Cblack] {};
\node (n1) at (14,11) [Cblack] {};
\node (n2) at (11,6)  [Cblack] {};
\node (n3) at (17,6)  [Cblack] {};
\node (n51) at (23,-4) [Cblack] {}; 
\node (n52) at (29,-4) [Cblack] {};   
\node (n61) at (27,2.8) [Cgray] {};  
\node (n62) at (27,-.7) [Cgray] {}; 
\node (n72) at (27,-9) [C0] {};  
\node (n70) at (20,-8.3) [C0] {}; 

\node (n18) at (35,16) [C0] {$\varphi_1=0$};
\node (n40) at (6.7,1)  [C0] {$\varphi_2=0$};
\node (n10) at (17,16) [C0] {$\varphi_3=0$}; 
\node (n181) at (38,11) [C0] {$\varphi_1=1$};
\node (n401) at (3.7,6)  [C0] {$\varphi_2=1$};
\node (n101) at (11,16) [C0] {$\varphi_3=1$};

\node (m24) at (30.5,6.8)  [C0] {$a_1$};
\node (mm) at (30.5,-4.5) [C0] {$a_2$};  
\node (m42) at (19,0.17)  [C0] {$a_3$};
\node (m42) at (25,0.17)  [C0] {$0$};

\node (n71) at (27,16) [C0] {$\G^{q,c}$};

\foreach \from/\to in {n71/n62,n72/n61}
    \draw[blue][dashed] (\from) -- (\to);
\foreach \from/\to in {n40/n41, n401/n21}
    \draw[black][dashed] (\from) -- (\to);
\foreach \from/\to in {n14/n18, n70/n51,n25/n181}
    \draw[green][dashed] (\from) -- (\to);
\foreach \from/\to in {n10/n12,n1/n3, n101/n11}
    \draw[red][dashed] (\from) -- (\to);

\foreach \from/\to in {n11/n14,n21/n25,n41/n44,n51/n52}
  \draw[black][] (\from) -- (\to);
\foreach \from/\to in {n11/n22,n22/n42,n42/n51,n12/n23,n23/n43,n43/n52,n13/n24,n24/n44,n14/n25,n21/n41}
  \draw[red][] (\from) -- (\to);
\foreach \from/\to in {n11/n21,n12/n22,n22/n41,n13/n23,n23/n42,n14/n24,n24/n43,n43/n51,n25/n44,n44/n52}
  \draw[green][] (\from) -- (\to);

\foreach \from/\to in {n62/n61}
    \draw[blue][] (\from) -- (\to);
\node (n61) at (27,2.8) [Cgray] {};  
\node (n62) at (27,-.7) [Cgray] {}; 
\node (n63) at (27,1) [Cgray] {}; 

\end{tikzpicture}
\caption{Cut lattice $L(G)$}
\label{fig:CutLattice}
\end{center}
\end{figure}
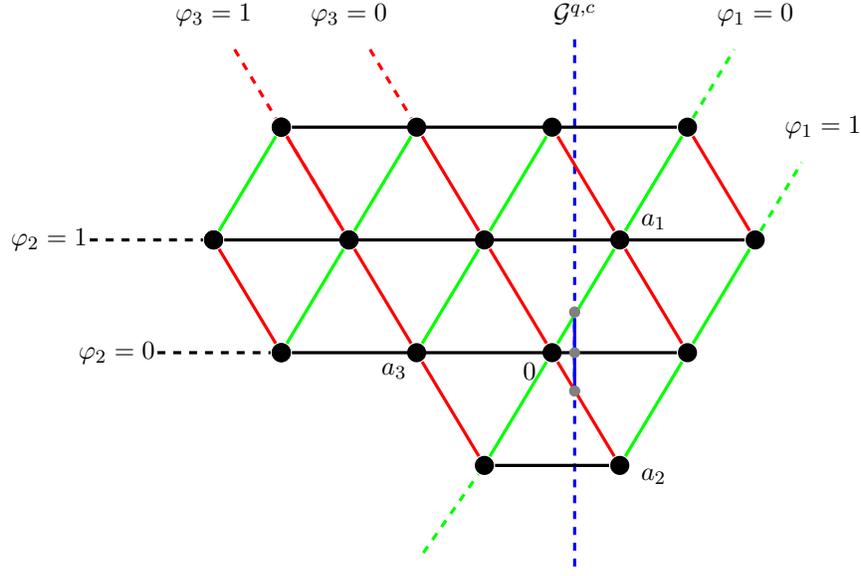

The cell decomposition $\Del(L(G))$ is the Delaunay decomposition of $L(G)_\RR$ with respect to the cut lattice and the usual Euclidean metric (cf. Remark~\ref{rmk:delon}(ii)), which coincides with an infinite hyperplane arrangement (Theorem~\ref{thm:totunim}(ii) and \S\ref{sec:cutlattice}). The hyperplanes at the origin are defined by $\varphi_i = e_i |_{L(G)_\RR} = 0$. 
The quotient cell decomposition $\Del(L(G))/L(G)$ of the torus $L(G)_\RR / L(G)$ has one $0$-cell $\{\pb\}$ (the orbit of the origin), three $1$-cells $\{E,E',E''\}$ (the orbits of the green, red, and black edges), and two $2$-cells $\{F,F'\}$ (the orbits of the upward and downward triangles). 
Assume that $q=u_3$ is the distinguished vertex. The hyperplane $\G^{q,c}$ is the hyperplane passing through points $c a_1$ and $c a_2$. In the figure $c$ is roughly $\frac{1}{3}$. The bounded complex of the intersection of this hyperplane with the arrangement at the origin is denoted by a solid blue segment. This is $\A_G^{q,c}$, which is combinatorially equivalent to $\B_G^{q,c}$ (via the coboundary map).

In Figure~\ref{fig:fund}, we have chosen a fundamental domain for the lattice, and have labeled all cells of this fundamental domain according to the recipe described in \S\ref{sec:labelDel}. This labeling induces a labeling on $\A_G^{q,c}$ (compatible with the labeling of $\B_G^{q,c}$) which is also given in the figure. The labelled cell complexes in Figure~\ref{fig:fund} are enough to completely describe minimal free resolutions for $\JJ_G$ and for $\OO_G$. Concretely, the minimal resolution of $\JJ_G$ is as follows:
\[
 0 \rightarrow \Sb(-\mb_{F}) \oplus \Sb(-\mb_{F'}) \xrightarrow{\partial_2} \Sb(-\mb_{E}) \oplus \Sb(-\mb_{E'})\oplus \Sb(-\mb_{E''}) \xrightarrow{\partial_1} \Sb(-\mb_{\pb}) \ .
\]
As usual, assume $[F]$ denotes the generator of $\Sb(-\mb_F)$. The labels of cells in $\Del(L(G)) / L(G)$ are:
\[
\mb_E = y_{e_2}y_{\bar{e}_3} \ ,
\quad
\mb_{E'} =  y_{e_1}y_{\bar{e}_3} \ ,
\quad
\mb_{E''} =  y_{e_1}y_{\bar{e}_2} \ ,
\]
\[
\mb_F =   y_{e_1} y_{e_2}y_{\bar{e}_3} \ ,
\quad
\mb_{F'}= y_{e_1}y_{\bar{e}_2}y_{\bar{e}_3} \ .
\]
The homogenized differential operator (see \eqref{eq:differntials}) of the cell complex $(\partial_1, \partial_2)$ is described as follows: 

\[
\partial_1([E])= \frac{y_{e_2}y_{\bar{e}_3}}{1}[\pb] -\frac{y_{e_2}y_{\bar{e}_3}}{\frac{y_{e_2}y_{\bar{e}_3}}{y_{\bar{e}_2}y_{e_3}}}[\pb] =  ({y_{e_2}y_{\bar{e}_3}} - {y_{\bar{e}_2}y_{e_3}})[\pb]\ ,
\]
\[
\partial_1([E'])= \frac{y_{e_1}y_{\bar{e}_3}}{\frac{y_{e_1}y_{\bar{e}_3}}{y_{\bar{e}_1}y_{e_3}}}[\pb] - \frac{y_{e_1}y_{\bar{e}_3}}{1}[\pb] = ({y_{\bar{e}_1}y_{e_3}}-{y_{e_1}y_{\bar{e}_3}})[\pb]\ , 
\]
\[
\partial_1([E''])= \frac{y_{e_1}y_{\bar{e}_2}}{\frac{y_{e_1}y_{\bar{e}_2}}{y_{\bar{e}_1}y_{e_2}}}[\pb] - \frac{y_{e_1}y_{\bar{e}_2}}{1}[\pb] = ({y_{\bar{e}_1}y_{e_2}} - {y_{e_1}y_{\bar{e}_2}})[\pb]\ , 
\]
\[
\partial_2([F])= \frac{y_{e_1}y_{e_2}y_{\bar{e}_3}}{y_{e_2}y_{\bar{e}_3}}[E] -  \frac{y_{e_1}y_{e_2}y_{\bar{e}_3}}{\frac{y_{e_1}y_{e_2}y_{\bar{e}_3}}{y_{e_3}}}[E''] + \frac{y_{e_1}y_{e_2}y_{\bar{e}_3}}{y_{e_1}y_{\bar{e}_3}}[E'] = {y_{e_1}}[E] - {y_{e_3}}[E'']  + {y_{e_2}}[E']   \ ,
\]
\[
\partial_2([F'])= \frac{y_{e_1}y_{\bar{e}_2}y_{\bar{e}_3}}{\frac{y_{e_1}y_{\bar{e}_2}y_{\bar{e}_3}}{{y_{\bar{e}_1}}}}[E] -  \frac{y_{e_1}y_{\bar{e}_2}y_{\bar{e}_3}}{y_{e_1}y_{\bar{e}_2}}[E''] + \frac{y_{e_1}y_{\bar{e}_2}y_{\bar{e}_3}}{y_{e_1}y_{\bar{e}_3}}[E'] = {y_{\bar{e}_1}}[E] - {y_{\bar{e}_3}}[E'']  + {y_{\bar{e}_2}}[E']   \ .
\]
Clearly $\JJ_G$ is the image of $\partial_1$ after identifying $[\pb]$ with $1$ (see Proposition~\ref{prop:grobrelation}). Since the labeling is compatible with the action of the lattice, any translation of this fundamental domain would give rise to the exact same description of the differential maps. 

The minimal resolution of $\OO_G^q$ can be read from the bounded complex $\A_G^{q,c}$. If we identify the name of each cell in $\A_G^{q,c}$ with the name of the associated cell in $\Del(L(G))$, we have 
\[
 0 \rightarrow \Sb(-\mb_{F}) \oplus \Sb(-\mb_{F'}) \xrightarrow{\tilde{\partial}_1} \Sb(-\mb_{E}) \oplus \Sb(-\mb_{E'})\oplus \Sb(-\mb_{E''}) \xrightarrow{\tilde{\partial}_0} \Sb \ ,
\]
where
\[
\tilde{\partial}_0([E])=\mb_E = y_{e_2}y_{\bar{e}_3} \ ,
\]
\[
\tilde{\partial}_0([E'])=\mb_{E'} =  y_{e_1}y_{\bar{e}_3} \ ,
\]
\[
\tilde{\partial}_0([E''])=\mb_{E''} =  y_{e_1}y_{\bar{e}_2} \ ,
\]
\[
\tilde{\partial}_1([F])= \frac{y_{e_1} y_{e_2}y_{\bar{e}_3}}{y_{e_2}y_{\bar{e}_3}}[E] - \frac{y_{e_1} y_{e_2}y_{\bar{e}_3}}{y_{e_1}y_{\bar{e}_3}}[E'] = y_{e_1}[E]-y_{e_2}[E']\ ,
\]
\[
\tilde{\partial}_1([F'])= \frac{y_{e_1}y_{\bar{e}_2}y_{\bar{e}_3}}{y_{e_1}y_{\bar{e}_3}}[E'] - \frac{y_{e_1}y_{\bar{e}_2}y_{\bar{e}_3}} {y_{e_1}y_{\bar{e}_2}} [E'']  = y_{\bar{e}_2}[E'] - y_{\bar{e}_3} [E''] \ .
\]
The ideal $\OO_G^q$ is the image of $\tilde{\partial}_0$ (see Proposition~\ref{prop:grobrelation}).
This example is, of course, closely related to Example~\ref{ex:PrinResol}. The general relationship between these two constructions is explained in Remark~\ref{rmk:ResolRelation}. 


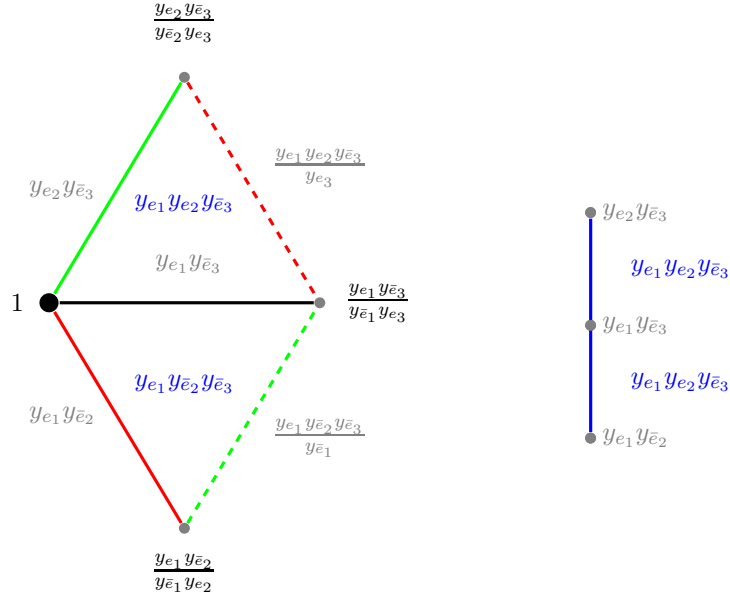
\begin{figure}[h!]

\begin{center}

\begin{tikzpicture} [scale = .60, very thick = 20mm]

\node (n4) at (14,1)  [Cgray] {};
\node (n1) at (14,11) [Cgray] {}; 
\node (n2) at (11,6)  [Cblack] {}; 
\node (n3) at (17,6)  [Cgray] {}; 

\node (m1) at (14,0) [C0] {$\frac{y_{e_1}y_{\bar{e}_2}}{y_{\bar{e}_1}y_{{e}_2}}$};
\node (m1) at (14,12.2) [C0] {$\frac{y_{e_2}y_{\bar{e}_3}}{y_{\bar{e}_2}y_{{e}_3}}$};
\node (m3) at (18.3,6) [C0] {$\frac{y_{e_1}y_{\bar{e}_3}}{y_{\bar{e}_1}y_{{e}_3}}$};
\node (m) at (10.3,6)  [C0] {$1$};

\foreach \from/\to in {n4/n2}
  \draw[red][] (\from) -- (\to);
\foreach \from/\to in {n1/n2}
  \draw[green][] (\from) -- (\to);
\foreach \from/\to in {n2/n3}
  \draw[black][] (\from) -- (\to);

\foreach \from/\to in {n1/n3}
  \draw[red][dashed] (\from) -- (\to);
\foreach \from/\to in {n4/n3}
  \draw[green][dashed] (\from) -- (\to);

\node (m10) at (11.3,8.5) [C0] {\textcolor{gray}{$y_{e_2}y_{\bar{e}_3}$}};
\node (m14) at (11.3,3.5) [C0] {\textcolor{gray}{$y_{e_1}y_{\bar{e}_2}$}};
\node (m23) at (14.1,6.86) [C0] {\textcolor{gray}{$y_{e_1}y_{\bar{e}_3}$}};
\node (m023) at (17,9.0) [C0] {\textcolor{gray}{$\frac{y_{e_1}y_{e_2}y_{\bar{e}_3}}{y_{{e}_3}}$}};
\node (m023) at (17,3.0) [C0] {\textcolor{gray}{$\frac{y_{e_1}y_{\bar{e}_2}y_{\bar{e}_3}}{y_{\bar{e}_1}}$}};

\node (m123) at (14,8.2) [C0] {\textcolor{blue}{$y_{e_1}y_{e_2}y_{\bar{e}_3}$}};
\node (m423) at (14,4.2) [C0] {\textcolor{blue}{$y_{e_1}y_{\bar{e}_2}y_{\bar{e}_3}$}};

\node (nn4) at (23,3)  [Cgray] {};
\node (nn1) at (23,8) [Cgray] {}; 
\node (nn0) at (23,5.5) [Cgray] {}; 

\node (m1) at (24,5.5) [C0] {\textcolor{gray}{$y_{e_1}y_{\bar{e}_3}$}};
\node (m1) at (24,3) [C0] {\textcolor{gray}{$y_{e_1}y_{\bar{e}_2}$}};
\node (m1) at (24,8) [C0] {\textcolor{gray}{$y_{e_2}y_{\bar{e}_3}$}};

\node (m1) at (25,4.25) [C0] {\textcolor{blue}{$y_{e_1}y_{e_2}y_{\bar{e}_3}$}};
\node (m1) at (25,6.75) [C0] {\textcolor{blue}{$y_{e_1}y_{e_2}y_{\bar{e}_3}$}};

\foreach \from/\to in {nn4/nn1}
  \draw[blue][] (\from) -- (\to);

\node (nn4) at (23,3)  [Cgray] {};
\node (nn1) at (23,8) [Cgray] {}; 
\node (nn0) at (23,5.5) [Cgray] {}; 

\end{tikzpicture}
\caption{A choice of fundamental domain with labels (left) , $\A_G^{q,c}$ with its induced labels (right)}
\label{fig:fund}
\end{center}

\end{figure}
\end{Example}

\subsection{Potential theory and Gr\"obner weight functionals for $\JJ_G$} \label{sec:PotJG}
Let $C_0(G,\RR)$ denote the real vector space spanned by $V(G)$, and let $C_1(G,\RR)$ denote the real vector space spanned by $\EE(G)$. The usual boundary operator $\partial \colon C_1(G,\RR) \rightarrow C_0(G,\RR)$ is defined by 
\[
(\partial(\sigma))(v) = \sum_{e_+=v}{\sigma(e)} - \sum_{e_-=v}{\sigma(e)} \ .
\]

An element $\sigma \in C_1(G,\RR)$ gives a map $\sigma \colon C^1(G,\ZZ) \rightarrow \RR$ by sending $f$ to $f(\sigma)$. So it may be thought of as a weight functional for the ideal $\JJ_G$. Our next goal is to study the weight functionals $\sigma \in C_1(G, \RR)$ that represent the term order $\prec_q$ in Proposition~\ref{prop:grobrelation}(iv). For our application, a very important class of examples arises from weight functionals representing $<_q$ for $\II_G$ as studied in \S\ref{sec:wt1} (see Lemma~\ref{lem:bq}, Definition~\ref{def:M_intwt}, or \eqref{eq:wt1}).

\begin{Proposition} \label{prop:plusworks}
Let $\vartheta \in C^0(G, \RR)$ be any weight functional representing $<_q$ for $\II_G$ (i.e. $\MM_G^q=\ini_{\vartheta}{(\II_G)}$). Then the $1$-chain $\sigma \in C_1(G,\RR)$ defined by 
\[
\sigma(e) = \vartheta(e_+) \quad \text{for all } e \in \EE(G)
\]
represents a term order $\prec_q$ for $\JJ_G$ with $\OO_G^q = \ini_{\sigma}(\JJ_G)$.
\end{Proposition}
\begin{proof}
By Proposition~\ref{prop:grobrelation}, the term order $\prec_q$ is characterized by requiring 
\[
{\prod_{e \in \EE(A,A^c)}{y_e} }  \ \prec_q {\prod_{e \in \EE(A^c, A)}{y_e} } 
\]
for every $A \subsetneq V(G)$, where $q \in A$ with $G[A]$ and $G[A^c]$ connected. 
Since (see Lemma~\ref{lem:labels})
\[
\frac{\prod_{e \in \EE(A^c, A)}{y_e} }{\prod_{e \in \EE(A,A^c)}{y_e} }=
\prod_{e \in \EE(G)} { {y_{e}^{d(\chi_{A^c})(e)}} } \ ,
\]


we have $\OO_G^q = \ini_{\sigma}(\JJ_G)$ if and only if 
\begin{equation}\label{eq:sigpos}
\sigma(d(\chi_{A^c}))=\sum_{e \in \EE(G)}  {\sigma(e) \cdot (d(\chi_{A^c}))(e)} >0
\end{equation}
 for all bonds $d(\chi_{A^c})(e)$ associated to $A \subsetneq V(G)$ with $q \in A$. Since $\partial$ is the adjoint to $d$, \eqref{eq:sigpos} is equivalent to

\begin{equation}\label{eq:sigmapos}
\sum_{v \in V(G)}{(\partial(\sigma))(v) \cdot \chi_{A^c}(v) } > 0 \ .
\end{equation}
Since $\sigma(e) = \vartheta(e_+)$, we have 
\[
\begin{aligned}
(\partial(\sigma))(v) &= \sum_{e_+=v}{\sigma(e)} - \sum_{e_-=v}{\sigma(e)} = \sum_{e_+=v}{\vartheta(e_+)} - \sum_{e_-=v}{\vartheta(e_+)} \\
&= \deg(v){\vartheta(v)} - \sum_{\{u,v\} \in E(G)}{\vartheta(u)} = \Delta(\vartheta)(v) \ .
\end{aligned}
\]
Therefore (see \eqref{eq:wt1char})
\[
\sum_{v \in V(G)}{(\partial(\sigma))(v) \cdot \chi_{A^c}(v) } = 
\sum_{v \in V(G)}{\Delta(\vartheta)(v) \cdot \chi_{A^c}(v) } >0
\]
and \eqref{eq:sigmapos} holds.
\end{proof}

\begin{Definition} \label{def:O_intwt}
Let $\vartheta_q \in C^0(G,\ZZ)$ denote the non-negative, integral functional defined in Definition~\ref{def:M_intwt}. We denote by $\lambda_q$ the associated non-negative, integral weight functional in $C_1(G,\RR)$ defined by
\[
\lambda_q(e) = \vartheta_q(e_+) \quad \text{for all } e \in \EE(G)
\]
as in Proposition~\ref{prop:plusworks} .
\end{Definition}

\subsection{Gr\"obner cone of $\OO_G^q$} \label{sec:Ocone}
Next we will describe the Gr\"obner cone associated to $\OO_G^q$. As in \S\ref{sec:grocone}, this cone is intimately related to potential theory and Green's functions. 

The description of this cone is most elegant when $G$ does not have a cut vertex. Cut vertices introduce linear subspaces in the Gr\"obner cone and are slightly tedious (but similar) to deal with. Throughout this section, we will therefore assume that $G$ is $2$-vertex-connected. This condition is equivalent to assuming that the lattice $L(G)$ is indecomposable (\cite[Proposition~4]{Bacher}).

\begin{Proposition}
Assume $G$ is $2$-vertex-connected. Then $\sigma \in C_1(G, \RR)$ represents a term order $\prec_q$ for $\JJ_G$ with $\OO_G^q = \ini_{\sigma}(\JJ_G)$ if and only if for all $p \in V(G) \backslash \{q\}$ we have
\[
\beta_p :=(\partial(\sigma))(p) > 0 \ .
\]
\end{Proposition}
\begin{proof}
We have already seen that $\sigma \in C_1(G, \RR)$ represents a term order $\prec_q$ for $\JJ_G$ with $\OO_G^q = \ini_{\sigma}(\JJ_G)$ if and only if \eqref{eq:sigmapos} holds for all bonds $d(\chi_{A^c})(e)$ associated to $A \subsetneq V(G)$ with $q \in A$. Since we have assumed there is no cut vertex, the star of every vertex gives a bond, so it is necessary (setting $A^c=\{p\}$ for $p\ne q$ in \eqref{eq:sigmapos}) to have $\beta_p=(\partial(\sigma))(p) > 0$. This condition is also sufficient because then for any bond $d(\chi_{A^c})(e)$ associated to $A \subsetneq V(G)$ with $q \in A$, we get 
\[
\sum_{v \in V(G)}{(\partial(\sigma))(v) \cdot \chi_{A^c}(v) }=\sum_{v \in V(G)}{\beta_v \cdot \sum_{p \in A^c} \chi_p(v) } = \sum_{p \in A^c}{\beta_p} >0 
\]
and \eqref{eq:sigmapos} holds.
\end{proof}

Therefore $\sigma \in C_1(G, \RR)$ is a solution to $\partial(\sigma)=\beta$ for $\beta=\sum_{p \in V(G)}{\beta_v(v)}$ in $\Div^0(G)$ with $\beta_p >0$ for $p \ne q$.

\medskip

After identifying $C_1(G,\RR)$ with $C^1(G, \RR)$ (by sending $e$ to $e^\ast$) we have the orthogonal (``Hodge'') decomposition
\[
C_1(G,\RR) \cong \Ker(\partial) \oplus \Image(d) \ . 
\]
Let $\sigma = \sigma' + \sigma''$ for $\sigma' \in \Ker(\partial)$ and $\sigma'' = d(\psi) \in \Image(d)$ for $\psi \in C^0(G,\RR)$. Then $\partial(\sigma)=\beta$ if and only if $\partial d(\psi) = \partial(\sigma'')=\beta$. By Remark~\ref{rmk:selfadjoint} $\partial d = 2\Delta$, so 
\[
\Delta{\psi} = \frac{1}{2}\beta \ .
\]

 It follows from the definition of the Green's function $j_q(p, v)$, together with the fact that the Laplacian operator has a one dimensional zero-eigenspace generated by $\mathbf{1}$, that:
\[
\psi=\frac{1}{2}\sum_{p \in V(G)}{\beta_p j_q(p, \cdot )} + k \cdot \mathbf{1}
\]
for some constant $k \in \RR$. Therefore
\[
\sigma(e)=\sigma'(e) + \sigma''(e) =\sigma'(e) + (d(\psi))(e) = \sigma'(e) + \frac{1}{2}\sum_{p \in V(G)}{\beta_p (j_q(p, e_+ ) - j_q(p, e_- ))}\ .
\]
We summarize these observations in the following theorem.

\begin{Theorem}
Assume $G$ is $2$-vertex connected. The $1$-chain $\sigma \in C_1(G, \RR)$ represents $\prec_q$ for $\JJ_G$ if and only if there exist $\sigma'  \in \Ker(\partial)$ and real numbers $\beta'_p >0$ (for $p \in V(G)$) such that 
\[
\sigma(e)= \sigma'(e) + \sum_{p \in V(G)}{\beta'_p (j_q(p, e_+ ) - j_q(p, e_- ))}
\]
for all $e \in \EE(G)$. 
\end{Theorem}

In other words $\sigma$ (up to an element of the ``extended cycle space'' $\Ker(\partial)$) is in the interior of the cone generated by the vectors $(j_q(p, e_+ ) - j_q(p, e_- ))_{e \in \EE(G)}$ for various $p \in V(G)$. It is easy, using \cite[Construction~3.1]{FarbodMatt12}, to show that these vectors are independent.

\section{Regular sequences for $\OO_G^q$ and $\JJ_G$}
\label{sec:main}


\subsection{Linear system of parameters for $\OO_G^q$}
The ideal $\OO_G^q\subset \Sb$ is a squarefree monomial ideal. Let $\Sigma_G^q$ denote its associated simplicial complex on $2m$ vertices $\{y_e: e \in \EE(G)\}$. 

\medskip

For each spanning tree $T$ of $G$, let $\Oc_T$ denote the orientation of $T$ with a unique source at $q$ (i.e. the orientation obtained by orienting all paths away from $q$). For an example, see Figure~\ref{fig:spanning trees}.

\begin{Proposition}\label{prop:primdecom}
\begin{itemize}
\item[]
\item[(i)] The number of facets of $\Sigma_G^q$ is the same as the number of spanning trees of $G$. For each spanning tree $T$, the corresponding facet $\tau_T$ is:
\[
\tau_T = \{y_e: e \in \EE(G) \backslash \Oc_T\} \ .
\] 

\item[(ii)] For each spanning tree $T$ of $G$, let 
$P_T=\langle y_e: e \in \Oc_T \rangle$. 
The minimal prime decomposition of $\Oc_G^q$ is
\[
\OO_G^q = \bigcap_{T} P_T \ ,
\]
the intersection being over all spanning trees of $G$. 

\item[(iii)] For each facet $\tau$ of $\Sigma_G^q$ we have $|\tau|=2m-n+1$. Therefore 
\[\dim(K[\Sigma_G^q])=2m-n+1 \ .\]

\item[(iv)] $\Sigma_G^q$ is Cohen-Macaulay. 
\end{itemize}
\end{Proposition}

\begin{proof}
(i) By Proposition~\ref{prop:grobrelation}, we know that $\OO_G^q$ is generated by monomials of the form $\prod_{e \in \EE(A^c,A)}{y_e}$, where $q \in A \subsetneq V(G)$ and $\EE(A^c,A) \subset \EE(G)$ denotes the set of oriented edges from $A$ to its complement $A^c$.

First we show that for each spanning tree $T$, the monomial $m_T:=\prod_{e \in \EE(G) \backslash \Oc_T}{y_e}$ does not belong to $\OO_G^q$. Clearly $m_T \in \OO_G^q$ if and only if $m_T$ is divisible by one of the given generators $\prod_{e \in \EE(A^c,A)}{y_e}$. But  
\[
\prod_{e \in \EE(A^c,A)}{y_e} \,\,\, \mid \, \prod_{e \in \EE(G) \backslash \Oc_T}{y_e} \quad\quad \iff \quad\quad  \EE(A^c,A) \subseteq (\EE(G) \backslash \Oc_T) \ .
\]
However, it follows from the definition of $\Oc_T$ that it must contain some element of $\EE(A^c,A)$ for any $A$.
This shows that $\tau_T =\{y_e: e \in \EE(G) \backslash \Oc_T\}$ is a face in the simplicial complex $\Sigma_G^q$. 

Next we show that $\tau_T$ must be a facet; for $f \in \Oc_T$ removing $f$ from the tree gives a partition of $V(T)=V(G)$ into two connected subsets $B$ and $B^c$ with $f_- \in B$ and $f_+ \in B^c$. Then the monomial $m_T \cdot y_f$ is divisible by $\prod_{e \in \EE(B^c,B)}{y_e}$. 

It remains to show that for any monomial $m = \prod_{e \in F}{y_e}$ that does not belong to $\OO_G^q$ we have $F\subseteq (\EE(G) \backslash \Oc_T)$ for some spanning tree $T$. To show this, we repeatedly use the fact that $m$ is not divisible by generators of the form $\prod_{e \in \EE(A^c,A)}{y_e}$ for various $A$, and construct a spanning tree $T$. This procedure is explained in Algorithm~\ref{alg:FindTree}. Note that if $\prod_{e \in F}{y_e}$ is not divisible by $\prod_{e \in \EE(A^c,A)}{y_e}$ then there exists an $e \in \EE(A^c,A)$ such that $e \not \in  F$. The orientation $\Oc_T$ is also induced by Algorithm~\ref{alg:FindTree}. 

\begin{algorithm}
\caption{Finding a facet containing a given monomial not belonging to $\OO_G^q$}
\KwIn{\\ A monomial $m = \prod_{e \in F}{y_e}$ not belonging to $\OO_G^q$.}
\KwOut{\\ A spanning tree $T$ such that $F\subseteq (\EE(G) \backslash \Oc_T)$.}

\BlankLine
{\bf Initialization:}
\\$A=\{q\}$, 
\\$T=\emptyset$.

\BlankLine

\While{$A \ne V(G)$}{
Find an oriented edge $e$ such that $e \in E(A,A^c)$ and $e \not \in  F$,\\
$T=T \cup \{e\}$, \\
$A=A \cup \{e_+\}$,
}
Output $T$.
\label{alg:FindTree}
\end{algorithm}
\medskip

(ii) follows from (i) and \cite[Theorem~1.7]{MillerSturmfels}.

\medskip

(iii) follows from (i) and the fact that $\dim(K[\Sigma_G^q])$ is equal to the maximal cardinality of the faces of $\Sigma_G^q$. 

\medskip

(iv) The Krull dimension of $K[\Sigma_G^q]=\Sb / \OO_G^q$ is $2m-n+1$ by part (iii). 
By the Auslander--Buchsbaum formula (for graded rings and modules, see \cite[page~437]{Singular}),
\[
\depth(\Sb / \OO_G^q) = \depth(\Sb) - \pd_\Sb(\Sb / \OO_G^q) = 2m-n+1
\]
 because $\pd_\Sb(\Sb / \OO_G^q)=n-1$ by Theorem~\ref{thm:bounded}.
Therefore $\dim(\Sb / \OO_G^q)=\depth(\Sb / \OO_G^q)$ and $K[\Sigma_G^q]$ is Cohen-Macaulay. 
\end{proof}

\medskip

\begin{Remark}
\begin{itemize}
\item[]
\item[(i)]
Proposition~\ref{prop:primdecom}(iii) can be strengthened; the simplicial complex $\Sigma_G^q$ is in fact shellable. Since $\JJ_G$  is the lattice ideal associated to the free abelian group $\Lambda = \Image{\partial^\ast}$, it is a toric ideal (in the sense of \cite[Chapter~4]{SturmfelsGrobnerConvex}). $\Sigma_G^q$ is precisely the {\em initial complex} of $\JJ_G$ with respect to $\prec_q$ (in the sense of \cite[Chapter~8]{SturmfelsGrobnerConvex}). Let $\sigma \in C_1(G, \RR)$ be any weight functional representing the term order $\prec_q$ for $\JJ_G$ (e.g. $\vartheta_q$ of Definition~\ref{def:O_intwt} -- see also \S\ref{sec:Ocone}). By \cite[Theorem~8.3]{SturmfelsGrobnerConvex} $\sigma$ provides us with a regular triangulation of $\Sigma_G^q$. This is accomplished by ``lifting'' each point $y_e$ into the next dimension by the height $\sigma(e)$, and then projecting back the lower face of the resulting positive cone. This is a unimodular triangulation because the ideal $\OO_G^q$ is squarefree (\cite[Corollary~8.9]{SturmfelsGrobnerConvex}). The associated Gr\"obner fan studied in \S\ref{sec:PotJG} coincides with the associated secondary fan of this triangulation.

It is well-known that given any regular triangulation, one can obtain shelling orders using the {\em line shelling} technique (see, e.g., \cite[Theorem~9.5.10]{TriangulationBook}). For the shellability of some seemingly related simplicial complexes, see \cite{Kozlov, Alex, Jujic}.

\item[(ii)] A minimal free resolution of the Alexander dual of $\OO_G^q$ can be obtained by the construction given in \cite{BatziesWelker} (see also \cite{AntonFatemeh} for some related results). Following the results of \cite{van2015orlik} in an ongoing project, we further explore connections to the relation space of the hyperplane arrangements, and the Orlik-Terao ideals.
\end{itemize}
\end{Remark}



\medskip

We are now ready to give a particularly nice {\em linear system of parameters} (abbreviated as {\em l.s.o.p.}) for $\OO_G^q$. Note that since $K[\Sigma_G^q]=\Sb / \OO_G^q$ is Cohen-Macaulay, every {\em homogeneous system of parameters} (in particular every l.s.o.p.) is regular (\cite[p.35]{Stanley96}).

\medskip

First we introduce some notation. For each $v \in V(G)$ we choose a distinguished incoming edge to $v$ and denote it by $e_v$. In other words, we fix a distinguished subset $\{e_v: v \in V(G)\ \} \subset \EE(G)$ of cardinality $n$ in such a way that $(e_v)_{+}=v$.

\medskip

For each $v$ define the set of linear forms
\[
\Lc_v = \{y_e - y_{e_v}:\  e \in \EE(G) \ , e \ne e_v \ , e_{+}=(e_v)_{+}=v \}
\]
and let
\begin{equation} \label{eq:L}
\Lc=  \bigcup_{v\in V(G)}{\Lc_v} \ .
\end{equation}
We also let 
\[
\Lc^{(q)}= \Lc \cup \{y_{e_q}\} \ .
\] 
Clearly, $|\Lc_v|=\deg(v)-1$ for $v\in V(G)$, $|\Lc|=2m-n$, and $|\Lc^{(q)}|=2m-n+1$.

\begin{Proposition}
\label{prop:reg}
The set $\Lc^{(q)}$ forms an l.s.o.p. (and thus a regular sequence) for $K[\Sigma_G^q]=\Sb/\OO_G^q$. 
\end{Proposition}

\begin{proof}
We will use the criterion of Kind and Kleinschmidt (\cite[pp.81-82]{Stanley96}, \cite{KindKleinschmidt}). Note that by Proposition~\ref{prop:primdecom}(iii), $\dim K[\Sigma_G^q]=|\Lc^{(q)}|$. 
For each facet $\tau$ and each vertex $v \ne q$, by Proposition~\ref{prop:primdecom}(i), all but one variable $y_e$ with $e^+=v$ appear in $\tau$. Again by Proposition~\ref{prop:primdecom}(i), all variables $y_e$ with $e^+=q$ appear in $\tau$. 
It follows that the dimension of the vector space spanned by the restrictions of forms in $\Lc^{(q)}$ to the facet $\tau$ is equal to $\sum_{v} (\deg(v)-1)+1=2m-n+1$ which is equal to $|\tau|$ by Proposition~\ref{prop:primdecom}(iii), and the conditions in \cite[Lemma 2.4]{Stanley96} are satisfied.
\end{proof}

\medskip

\begin{Example}\label{exam:Delta}
For the graph in Example~\ref{exam:1}, $\OO_G^q$ is the Stanley-Reisner ideal of the simplicial complex $\Sigma_G^q$ given by facets 

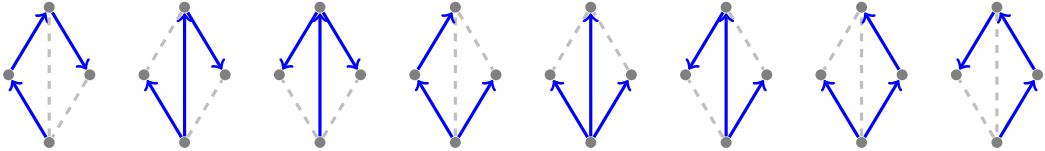
\begin{figure}[h!]

\begin{center}

\begin{tikzpicture} [scale = .18, very thick = 10mm]

  \node (n4) at (4,1)  [Cgray] {};
  \node (n1) at (4,11) [Cgray] {};
  \node (n2) at (1,6)  [Cgray] {};
  \node (n3) at (7,6)  [Cgray] {};

\foreach \from/\to in {n2/n1,n1/n3,n4/n2}
    \draw[blue][->] (\from) -- (\to);
 
\foreach \from/\to in {n4/n1,n3/n4}
    \draw[lightgray][dashed] (\from) -- (\to);

    \node (n4) at (14,1)  [Cgray] {};
  \node (n1) at (14,11) [Cgray] {};
  \node (n2) at (11,6)  [Cgray] {};
  \node (n3) at (17,6)  [Cgray] {};
  \foreach \from/\to in {n4/n2,n1/n3,n4/n1}
   \draw[blue][->] (\from) -- (\to);

     \foreach \from/\to in {n1/n2,n4/n3}
    \draw[lightgray][dashed] (\from) -- (\to);

    \node (n4) at (24,1)  [Cgray] {};
  \node (n1) at (24,11) [Cgray] {};
  \node (n2) at (21,6)  [Cgray] {};
  \node (n3) at (27,6)  [Cgray] {};
  \foreach \from/\to in {n4/n2,n4/n3}
    \draw[lightgray][dashed] (\from) -- (\to);

    \foreach \from/\to in {n4/n1,n1/n3,n1/n2}
    \draw[blue][->] (\from) -- (\to);

  \node (n4) at (34,1)  [Cgray] {};
  \node (n1) at (34,11) [Cgray] {};
  \node (n2) at (31,6)  [Cgray] {};
  \node (n3) at (37,6)  [Cgray] {};
   \foreach \from/\to in {n4/n1,n1/n3}
    \draw[lightgray][dashed] (\from) -- (\to);

    \foreach \from/\to in {n2/n1,n4/n2,n4/n3}
    \draw[blue][->] (\from) -- (\to);

    \node (n4) at (44,1)  [Cgray] {};
  \node (n1) at (44,11) [Cgray] {};
  \node (n2) at (41,6)  [Cgray] {};
  \node (n3) at (47,6)  [Cgray] {};
 \foreach \from/\to in {n1/n2,n3/n1}
    \draw[lightgray][dashed] (\from) -- (\to);

    \foreach \from/\to in {n4/n2,n4/n3,n4/n1}
    \draw[blue][->] (\from) -- (\to);

    \node (n4) at (54,1)  [Cgray] {};
  \node (n1) at (54,11) [Cgray] {};
  \node (n2) at (51,6)  [Cgray] {};
  \node (n3) at (57,6)  [Cgray] {};
 \foreach \from/\to in {n1/n3,n4/n2}
    \draw[lightgray][dashed] (\from) -- (\to);

    \foreach \from/\to in {n4/n3,n4/n1,n1/n2}
    \draw[blue][->] (\from) -- (\to);


   \node (n4) at (64,1)  [Cgray] {};
  \node (n1) at (64,11) [Cgray] {};
  \node (n2) at (61,6)  [Cgray] {};
  \node (n3) at (67,6)  [Cgray] {};
 \foreach \from/\to in {n1/n2,n4/n1}
    \draw[lightgray][dashed] (\from) -- (\to);

    \foreach \from/\to in {n4/n3,n4/n2,n3/n1}
    \draw[blue][->] (\from) -- (\to);


 \node (n4) at (74,1)  [Cgray] {};
  \node (n1) at (74,11) [Cgray] {};
  \node (n2) at (71,6)  [Cgray] {};
  \node (n3) at (77,6)  [Cgray] {};
 \foreach \from/\to in {n4/n2,n1/n4}
    \draw[lightgray][dashed] (\from) -- (\to);

    \foreach \from/\to in {n4/n3,n3/n1,n1/n2}
    \draw[blue][->] (\from) -- (\to);

\end{tikzpicture}

\caption{Spanning trees $T$ and orientations $\Oc_T$ corresponding to $\tau_1,\tau_2,\ldots,\tau_8$}
\label{fig:spanning trees}
\end{center}
\end{figure}


\[\tau_1=\{y_{e_1},y_{e_3},y_{e_4},y_{e_5},y_{\bar{e}_2},y_{\bar{e}_4},y_{\bar{e}_5}\}, \ 
\tau_2=\{y_{e_1},y_{e_3},y_{e_4},y_{\bar{e}_1},y_{\bar{e}_2},y_{\bar{e}_4},y_{\bar{e}_5}\},\]
\[
\tau_3=\{y_{{e}_2},y_{e_3},y_{e_4},y_{\bar{e}_1},y_{\bar{e}_2},y_{\bar{e}_4},y_{\bar{e}_5}\}, \ 
\tau_4=\{y_{e_1},y_{e_3},y_{{e}_5},y_{\bar{e}_2},y_{\bar{e}_3},y_{\bar{e}_4},y_{\bar{e}_5}\},\]
\[
\tau_5=\{y_{e_1},y_{e_3},y_{\bar{e}_1},y_{\bar{e}_2},y_{\bar{e}_3},y_{\bar{e}_4}, y_{\bar{e}_5}\}, \ 
\tau_6=\{y_{e_2},y_{e_3},y_{\bar{e}_1},y_{\bar{e}_2},y_{\bar{e}_3},y_{\bar{e}_4},y_{\bar{e}_5}\},
\]
\[
\tau_7=\{y_{e_1},y_{e_5},y_{\bar{e}_1},y_{\bar{e}_2},y_{\bar{e}_3},y_{\bar{e}_4},y_{\bar{e}_5}\}, \
\tau_8=\{y_{e_2},y_{e_5},y_{\bar{e}_1},y_{\bar{e}_2},y_{\bar{e}_3},y_{\bar{e}_4},y_{\bar{e}_5}\}.
\]

See Proposition~\ref{prop:primdecom}(i) and Figure~\ref{fig:spanning trees}.

If we choose $\{e_1, e_3, e_4,\bar{e}_4\}$ as our distinguished set of incoming edges to vertices, we have 
\[
\Lc_{u_1}=\{y_{e_2}-y_{e_1}\} \ , \quad
\Lc_{u_2}=\{y_{\bar{e}_1}-y_{e_3}, y_{e_5}-y_{e_3}\} \ ,
\]
\[
\Lc_{u_3}=\{y_{\bar{e}_3}-y_{e_4}\} \ , \quad
\Lc_{u_4}=\{y_{\bar{e}_2}-y_{\bar{e}_4}, y_{\bar{e}_5}-y_{\bar{e}_4}\} \ .
\]
Therefore
\[
\Lc= \bigcup_{v \in V(G)} \Lc_v = \{y_{e_2}-y_{e_1}, y_{\bar{e}_1}-y_{e_3}, y_{e_5}-y_{e_3}, y_{\bar{e}_3}-y_{e_4}, y_{\bar{e}_2}-y_{\bar{e}_4}, y_{\bar{e}_5}-y_{\bar{e}_4} \} 
\]
and
\[
\Lc^{(q)}=\Lc\cup \{y_{\bar{e}_4} \} \ .
\]

Note that $|\Lc^{(q)}|=7=2 \times 5 - 4 +1$. The restrictions of linear forms of $\Lc^{(q)}$ to
$\tau_1$ are

\[
\Lc^{(q)} | _{\tau_1}= \{ -y_{e_1}, -y_{e_3}, y_{e_5}-y_{e_3}, -y_{e_4}, y_{\bar{e}_2}, y_{\bar{e}_5}, y_{\bar{e}_4} \} \ ,
\]
which span a vector space of dimension $|\tau_1|=7=2 \times 5 - 4 +1$. Similarly, the restrictions of the linear forms of $\Lc^{(q)}$ to the other $\tau_i$'s span a vector space of dimension $|\tau_i|$.
\end{Example}

\medskip

\subsection{Linear system of parameters for $\JJ_G$}

Next we use \cite[Proposition 15.15]{Eisenbud} to give a regular sequence for $\Sb/\JJ_G$. 

\begin{Proposition}\label{prop:regularJ}
The set $\Lc$ forms a regular sequence for $\Sb/\JJ_G$.
\end{Proposition}

\begin{proof}
Let $\lambda_q \in C_1(G,\RR)$ be the integral, non-negative weight functional defined in Definition~\ref{def:O_intwt}. 
Any element of $\Lc$ is of the form $g=y_e - y_{e_v}$ with $e_+=(e_v)_+=v$ for some $v \in V(G)$. Since $\lambda_q(e)=\lambda_q(e_v)$ depends only on $v$ by the construction in Proposition~\ref{prop:plusworks}, we obtain $\ini_{\lambda_q}(g)=g$ and $\tilde{g}=g$. Therefore $\{\ini_{\lambda_q}(g) : g \in \Lc\} = \Lc$ which is a regular sequence on $\Sb / \ini_{\lambda_q}(\JJ_G)=\Sb / \OO_G^q$ by Proposition~\ref{prop:reg}. So we may apply \cite[Proposition 15.15]{Eisenbud} to conclude that $\Lc$ is a $(\Sb/\JJ_G)$-regular sequence.
\end{proof}

\begin{Remark}
It follows from \cite[p.35]{Stanley96} and Proposition~\ref{prop:CM} that $\Lc$ also forms a (partial) l.s.o.p. for $\Sb/\JJ_G$.
\end{Remark}

\section{$\II_G$ from $\JJ_G$ and $\MM_G^q$ from $\OO_G^q$} 
\label{sec:relate}

A common and powerful technique in the theory of divisors on graphs and chip-firing games is to relate divisors to orientations. Given an orientation, one can form a divisor from the associated indegrees or outdegrees (see, e.g., \cite[Theorem~2.3]{Lovasz91}, \cite[Theorem~3.3]{BN1}, \cite{HopPerk}, \cite{FarbodFatemeh},  \cite{ABKS}, and \cite{Fatemehreliability}). Algebraically, there is a good justification for the strength of this method related to the regular sequences studied in \S\ref{sec:main}.

\medskip

Recall that $\Rb=K[\xb]$ denotes the polynomial ring in $n$ variables $\{x_v: v \in V(G)\}$ and $\Sb=K[\yb]$ denotes the polynomial ring in $2m$ variables $\{y_e: e \in \EE(G)\}$. There is a canonical surjective $K$-algebra homomorphism 
\[
\phi \colon \Sb \rightarrow \Rb
\]
defined by sending $y_e$ to $x_{e_+}$ for all $e \in \EE(G)$. The kernel of this map is precisely the ideal generated by $\Lc$ (defined in \eqref{eq:L}), which we denote by $\aa= \langle \Lc \rangle$. 
The induced isomorphism 
\[
\bar{\phi} \colon \Sb /\aa \xrightarrow{\sim} \Rb
\]
is the ``algebraic indegree map'', and it relates the ideals $\II_G$ and $\MM_G^q$ to the ideals $\JJ_G$ and $\OO_G^q$. 

\begin{Proposition}\label{prop:specialize}
\begin{itemize}
\item[]
\item[(i)] $\bar{\phi}(\JJ_G +\aa)  = \II_G$. In other words $\bar{\phi}$ induces an isomorphism $(\Sb / \JJ_G ) \otimes_{\Sb} (\Sb / \aa) \cong \Rb / \II_G$.
\item[(ii)] $\bar{\phi}(\OO_G^q +\aa )  = \MM_G^q$. In other words $\bar{\phi}$ induces an isomorphism $(\Sb / \OO_G^q) \otimes_{\Sb} (\Sb / \aa) \cong \Rb / \MM_G^q$.
\end{itemize}
\end{Proposition}
\begin{proof}
The map $\bar{\phi}$ sends $\prod_{e \in \EE(A^c,A)}{y_e} + \aa$ to $\xb^{D(A^c,A)}$. So the proposition immediately follows from examining the generating sets described in Theorem~\ref{thm:Cori} and in Proposition~\ref{prop:grobrelation}. 
\end{proof}

\begin{Remark} \label{rmk:Rtilde}
The variables $y_e$ with $e_+=q$ do not appear in the support of any element of $\OO_G^q$ (see Theorem~\ref{prop:grobrelation}(iii)). Likewise, the variable $x_q$ does not appear in the support of any element of $\MM_G^q$ (see Theorem~\ref{thm:Cori}). Therefore we also have an isomorphism $\bar{\phi}(\OO_G^q + \langle \Lc^{(q)} \rangle) = \bar{\phi}(\OO_G^q +\aa + \langle y_{e_q} \rangle)  \cong \MM_G^q+\langle x_q \rangle$. In other words $(\Sb / \OO_G^q) \otimes_{\Sb} (\Sb / \langle \Lc^{(q)}\rangle) \cong \tilde{\Rb} / \MM_G^q$, where $\tilde{\Rb} = K[\{x_v\}_{ v \ne q}]$.
\end{Remark}

\medskip

\begin{Theorem} \label{thm:betti_coincide}
\begin{itemize}
\item[]
\item[(i)] The polyhedral cell complex $\B_G^{q,c}$ (equivalently, $\A_G^{q,c}$) supports a $\Div(G)$-graded (and $\ZZ$-graded) minimal free resolution for $\MM_G^q$. 

\item[(ii)] The quotient labeled cell complex $\Del(L(G)) / L(G)$ supports a $\Pic(G)$-graded (and $\ZZ$-graded) minimal free resolution for $\II_G$.

\item[(iii)] The $\ZZ$-graded Betti diagrams of $\JJ_G$, $\II_G$, $\OO_G^q$, and $\MM_G^q$ coincide.
\end{itemize}
\end{Theorem}
\begin{proof}
(i) By Theorem~\ref{thm:bounded}, we know that $\B_G^{q,c}$ gives a $C^1(G,\ZZ)$-graded minimal free resolution for $\Sb / \OO_G^q$. The same statement is true if we replace $\B_G^{q,c}$ with $\A_G^{q,c}$ by the discussion in \S\ref{sec:grobrel}. 
By \cite[Lemma~3.15]{EisenbudSyz} (see also \cite[Proposition~1.1.5]{Bruns}) and Proposition~\ref{prop:reg}, if we replace all the labels $\mb_F$ with $\mb_F + \aa$, we obtain a minimal cellular free resolution for $(\Sb / \OO_G^q)/\otimes_{\Sb} (\Sb / \aa) \cong \Rb / \MM_G^q$ (see Proposition~\ref{prop:specialize}(ii)). Alternatively we could replace all labels $\mb_F$ with $\mb_F + \langle \Lc^{(q)} \rangle$ to obtain a minimal cellular free resolution for $(\Sb / \OO_G^q) \otimes_{\Sb} (\Sb / \langle \Lc^{(q)}\rangle) \cong \tilde{\Rb} / \MM_G^q$. The new labels are easily seen to be $\Div(G)$ and $\ZZ$-homogeneous, and the resulting minimal free resolution is $\Div(G)$ and $\ZZ$-graded.

\medskip

(ii) follows similarly from Theorem~\ref{thm:Jresol}, \cite[Lemma~3.15]{EisenbudSyz}, Proposition~\ref{prop:regularJ}, and Proposition~\ref{prop:specialize}(i). 

\medskip

(iii) The fact that the (ungraded) Betti numbers of $\JJ_G$ and $\OO_G^q$ coincide follows from Lemma~\ref{lem:bij}. By the labeling compatibility described in Lemma~\ref{lem:labels} the $\ZZ$-graded Betti numbers of $\JJ_G$ and $\OO_G^q$ coincide as well. Since all elements of $\Lc$ are homogeneous (linear) forms, the relabeling of cells described above (in passing from $\JJ_G$ to $\II_G$ and from $\OO_G^q$ to $\MM_G^q$) does not change the $\ZZ$-degrees. Therefore the $\ZZ$-graded Betti diagrams of all four ideals coincide.
\end{proof}

\begin{Remark}
Recall from Remark~\ref{sec:nopartorient} that the number of $i$-dimensional cells in $\B_G^{q,c}$ is equal to the number of acyclic partial orientations of $G$ with $(i+2)$ (connected) components having unique source at $q$. So one immediately obtains a combinatorial description of the (ungraded) Betti numbers in terms of acyclic partial orientations. This interpretation for the Betti numbers of $I_G$ was conjectured in \cite{perkinson} and was proved in \cite{FarbodFatemeh} and \cite{Madhu}. 
\end{Remark}

\medskip

\begin{Example}\label{exam:3}
We return to Examples~\ref{exam:1}. We described the sequence $\Lc^{(q)}$ in Example~\ref{exam:Delta}. 
For simplicity we let $x_i=x_{u_i}$. By sending $\{y_{e_2}, y_{e_1}\}$ to $x_1$, $\{y_{\bar{e}_1}, y_{e_5}, y_{e_3}\}$ to $x_{2}$,
and $\{y_{\bar{e}_3}, y_{e_4}\}$ to $x_{3}$, $\OO_G^q$ in \eqref{eq:OGex} is sent to the ideal
\[
\langle x_2^2 x_3,x_1 x_2^2,x_3^2,x_2^{3},x_1^{2},x_1x_2x_3\rangle
\]
which is precisely $\MM_G^q=\ini_{<_q}(\II_G)$ by Theorem~\ref{thm:Cori}(ii).
The minimal cellular free resolution of $\MM_G^q$ is obtained from the minimal cellular free resolution of $\OO_G^q$ (described in Examples~\ref{exam:1}) by ``relabeling'' (i.e. by replacing each $y_e$ with $x_{e_+}$). We first relabel the complex in Figure~\ref{fig:arrangement} to obtain Figure~\ref{fig:arrangement2}. The resulting labeled complex gives a minimal free resolution for $\MM_G^q$ which is precisely the minimal free resolution of $\OO_G^q$ ``relabeled''. Concretely, we first extend the labels $\mb'(\pb_i)$ on the vertices to the whole of $\B_G^q$ by the least common multiple construction. For example,  
\[
\mb_{E_2}=y_{\bar{e}_1}y_{\bar{e}_3} y_{e_4} y_{e_5} \mapsto \mb'_{E_2}= x_2^2 x_3^2 \ ,
\]
\[
\mb_{E_4}=y_{\bar{e}_1} y_{e_2} y_{e_4} y_{e_5} \mapsto \mb'_{E_4}= x_1x_2^2x_3\ ,
\]
\[
\mb_{E_5}=y_{e_2} y_{\bar{e}_3} y_{e_4} y_{e_5} \mapsto \mb'_{E_5}= x_1x_2x_3^2\ ,
\]
\[
\mb_{E_6}=y_{e_2}y_{e_3} y_{e_4} y_{e_5} \mapsto \mb'_{E_6}=x_1x_2^2x_3 \ ,
\]
\[
\mb_{F_2}= y_{\bar{e}_1} y_{e_2} y_{\bar{e}_3} y_{e_4}y_{e_5} \mapsto \mb'_{F_2}= x_1x_2^2x_3^2\ .
\] 
The minimal resolution of $\MM_G^q$ is as follows.
\[
 0 \rightarrow \bigoplus_{i=1}^4\Rb(-\mb'_{F_i}) \xrightarrow{\partial'_2} \bigoplus_{i=1}^9\Rb(-\mb'_{E_i}) \xrightarrow{\partial'_1} \bigoplus_{i=1}^6\Rb(-\mb'_{\pb_i}) \xrightarrow{\partial'_0} \Rb \twoheadrightarrow \Rb /\MM_G^q \ .
\]
Assume $[[F]]$ denotes the generator of $\Rb(-\mb'_F)$. The homogenized differential operator of the cell complex $(\partial'_0,\partial'_1, \partial'_2)$ is as described in \eqref{eq:differntials}. For example:

\[
\partial'_0([[\pb_i]])= \mb'_{\pb_i} = \mb'(\pb_i) \ ,
\]
\[
\partial'_1([[E_6]]) = x_2[[\pb_4]] - x_3[[\pb_4]]\ ,
\]
\[
\partial'_2([[F_2]]) =x_1[[E_2]] - x_3[[E_4]] + x_2[[E_5]] \ .
\]

Although $\JJ_G$ and $\II_G$ have the same Betti table as $\OO_G^q$ and $\MM_G^q$, it is not possible to read the minimal free resolutions for $\JJ_G$ or $\II_G$ directly from $\B_G^q$; one really needs to consider the cell decomposition of $L(G)_{\RR}/L(G)$ or of $\Div_{\RR}^0(G)/\Prin(G)$.
\end{Example}

\medskip

\setlength{\unitlength}{1.1pt}
\begin{figure}[h!]
\begin{center}
    \begin{picture}(100,195)(0,-55)

     \thicklines
     \put(60,0){\circle*{4}}
      \put(35,-13){$x_2^2x_3$}

 \put(26,66){\circle*{4}}
      \put(-13,60){$x_1x_2^2$}

      \put(160,0){\circle*{4}}
      \put(140,-13){$x_3^2$}

 \put(-40,0){\circle*{4}}
      \put(-75,-13){$x_2^3$}

      \put(60,100){\circle*{4}}
      \put(40,98){$x_1^2$}

      \put(60,50){\circle*{4}}
      \put(20,45){$x_1x_2x_3$}

      \put(60,0){\line(0,-1){30}}
\put(60,50){\line(2,-1){100}}  \put(60,50){\line(-2,1){65}}     \put(60,100){\line(0,1){30}}
      \put(60,-20){\vector(-1,0){10}}
      \put(56,-43){$H_{e_3}$}

 \put(175,0){\vector(0,1){10}}
      \put(180,10){$H_{e_2}$}

 \put(-50,-10){\vector(1,-1){10}}
      \put(-60,-43){$H_{e_4}$}

      \put(30,130){\line(1,-1){20}}
      \put(160,0){\line(1,-1){20}}
      \put(170,-10){\vector(-2,-3){7}}
      \put(193,-40){$H_{e_5}$}

  \put(14,73){\vector(2,3){7}}
      \put(-30,80){$H_{e_1}$}

      \thicklines
      \put(-40,0){\line(1,0){230}}
      \put(-40,0){\line(-1,0){30}}
      \put(60,0){\line(-1,0){80}}
      \put(60,0){\line(0,1){130}}

\put(160,0){\line(3,-2){35}}
\put(-40,0){\line(1,1){130}} \put(-40,0){\line(-1,-1){30}}
\put(160,0){\line(-1,1){130}} \put(160,0){\line(1,-1){30}}

\put(-40,-10){$\pb_1$}
\put(64,-10){$\pb_2$}
\put(155,10){$\pb_3$}
\put(63,55){$\pb_4$}
\put(21,74){$\pb_5$}
\put(65,98){$\pb_6$}

    \end{picture}
    \caption{The relabeled bounded complex $\B_G^{q,c}$ giving a minimal free resolution of $\MM_G^{q}$}
    \label{fig:arrangement2}
\end{center}
\end{figure}
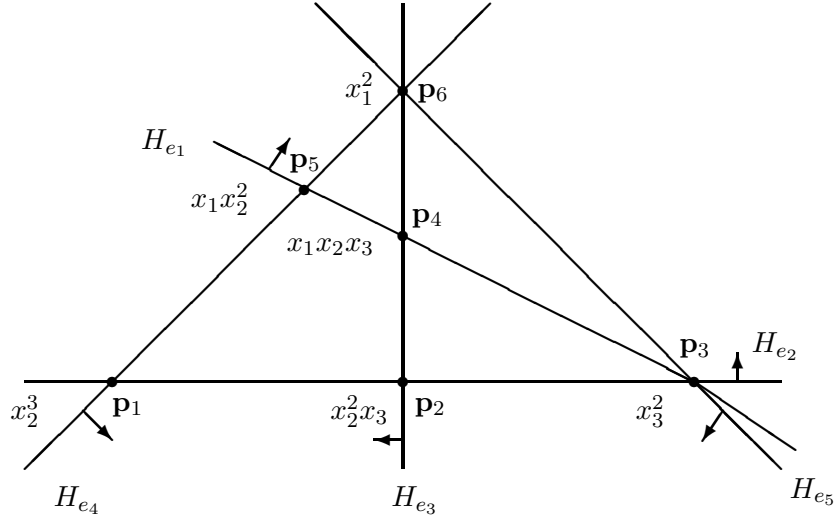

\begin{Remark} \label{rmk:ResolRelation}
There is an isometry between the principal lattice $(\Prin(G) , \langle \cdot, \cdot \rangle_{\en})$ and the cut lattice $(L(G) , \langle \cdot , \cdot \rangle )$ (Remark~\ref{rmk:isometry}). So the Delaunay decompositions $\Del(\Prin(G))$ and $\Del(L(G))$ are combinatorially equivalent (compare Figure~\ref{fig:PrinLattice} with Figure~\ref{fig:CutLattice}) and the relabeling of cells in $\Del(L(G))$ described above correspond to the labels that were given to cells of $\Del(\Prin(G))$ in \S\ref{sec:labelDelPrin}. Therefore the resolution of $\II_G$ described in Theorem~\ref{thm:Ires} coincides with the resolution of $\II_G$ obtained from the resolution of $\JJ_G$ in Theorem~\ref{thm:Jresol} by ``relabeling'' as in Theorem~\ref{thm:betti_coincide}.  For example, the resolution of $\II_G$ described in Example~\ref{ex:PrinResol} can alternatively be obtained from the resolution of $\JJ_G$ described in Example~\ref{ex:CutResol}.

It is straightforward to give an alternate {\em proof} for Theorem~\ref{thm:Ures} and Theorem~\ref{thm:Ires} using these observations.
\end{Remark}


\medskip

\section{Some consequences of our main results}
\label{sec:conseq}
\subsection{Cohen-Macaulayness}
\label{sec:CM}
For a polynomial ring $S$, a term order $<$ and an ideal $I \subset S$, it is known that $S/I$ is Cohen-Macaulay if and only if $S/\ini_{<}(I)$ is Cohen-Macaulay (see, e.g., \cite[Corollary 3.3.5]{HHBook}).
\begin{Proposition}\label{prop:CM}
The modules $\Sb / \OO_G^q$, $\Rb / \MM_G^q$, $\Sb / \JJ_G$, and $\Rb / \II_G$ are all Cohen-Macaulay.
\end{Proposition}

\begin{proof}
By Proposition~\ref{prop:primdecom}(iv) we have that $\Sb/\OO_G^q$ is Cohen-Macaulay. 
For $\Rb / \MM_G^q$, first observe that by Theorem~\ref{thm:Cori} the variable $x_q$ does not appear in the support of any of the given monomial generators of $\MM_G^q$. This implies that $\depth(\Rb / \MM_G^q)\geq 1$. On the other hand, $\dim(\Rb/\MM_G^q)=1$. 
Therefore $\MM_G^q$ is also Cohen-Macaulay. 
Since $\ini_{\prec_q}(\JJ_G)=\OO_G^q$ and $\ini_{<_q}(\II_G)=\MM_G^q$, we immediately conclude that $\Sb/\JJ_G$ and $\Rb/\II_G$ are also Cohen-Macaulay.
\end{proof}

We remark that the Cohen-Macaulay property of $\Rb / \MM_G^q$ and $\Rb / \II_G$ also follows from the results of \cite{FarbodFatemeh} and \cite{Madhu}.
\medskip

\subsection{Multiplicities}
\label{sec:mult}
For a finitely generated (graded) module $M$ of dimension $d>0$ over a polynomial ring, the {\em multiplicity} of $M$ is defined to be the leading coefficient of the Hilbert polynomial of $M$ (i.e. the polynomial defining $i \mapsto \dim(M_i)$ for $i >> 0$). We will denote this quantity by $e(M)$. Since the Hilbert polynomial is completely determined by the Betti table (see, e.g., \cite[Theorem~8.20 and Proposition~8.23]{MillerSturmfels}), the multiplicity is also determined by the Betti table. The following result easily follows.

\begin{Theorem}
\[
e(\Sb/\OO_G^q)=e(\Sb/\JJ_G)=e(\Rb/\MM_G^q)=e(\Rb/\II_G)=\kappa(G)\ ,
\]
where $\kappa(G)$ denotes the number of spanning trees of $G$.
\end{Theorem}

\begin{proof}
All these ideals have the same Betti table and hence the same multiplicity. It suffices to compute  the multiplicity of $\Sb/\OO_G^q=K[\Sigma_G^q]$.
By Proposition~\ref{prop:primdecom}(ii), we have 
\[
\OO_G^q = \bigcap_{T} P_T \ ,
\]
the intersection being over all spanning trees of $G$. 
By Proposition~\ref{prop:primdecom}(iii), we have $\dim(\Sb/\OO_G^q)=2m-n+1$. 
Also, for each spanning tree $T$ we have $P_T=\langle y_e: e \in \Oc_T \rangle$ and therefore
\[
\dim(\Sb/P_T)=2m-n+1\quad \text{and} \quad e(\Sb/P_T)=1 \ .
\]

In this situation (see, e.g., \cite[Lemma 5.3.11]{Singular}) we have
\[
e(\Sb/\OO_G^q)=\sum_{T}e(\Sb/P_T) \ ,
\]
the sum being over all spanning trees of $G$. 
\end{proof}

For $\Rb / \II_G$, the multiplicity was recently computed in \cite{OCarroll} using a different method. There are other related ideals with the same multiplicity (see e.g. \cite[Proposition 3.7]{kateri2014family}).

\subsection{Alexander dual of $\MM_G^q$ and cocellular free resolution} \label{sec:alex}
In \cite{MadhuBernd}, Riemann-Roch theory for graphs is linked to Alexander duality for the ideal $\MM_G^q$.  Recall that $\MM_G^q \subset \tilde{\Rb} = K[\{x_v\}_{ v \ne q}]$ (see Remark~\ref{rmk:Rtilde}). Here we quickly study the Alexander dual of $\MM_G^q$ and use Theorem~\ref{thm:bounded} to obtain its minimal cocellular free resolution. 

\medskip

We define the divisor 
\[
\ab = \sum_{v \in V(G)}{(\deg(v))(v)} \ .
\]
It follows from Theorem~\ref{thm:Cori} and Theorem~\ref{thm:betti_coincide}(i) that:
\begin{itemize}
\item[(i)] $\MM_G^q$ is generated in degree preceding $\ab$.
\item[(ii)] $\MM_G^q+\langle \{x_v^{\ab(v)+1}\}_{v \ne q} \rangle = \MM_G^q$; this this is because for each $v\neq q$ in $V(G)$,  the star of the vertex $v$ forms a cut and therefore $x_v^{\deg(v)} \in \MM_G^q$.
\item[(iii)] All face labels in the labeled cell complex $B_G^{q,c}$ resolving $\MM_G^q$ (as in Theorem~\ref{thm:betti_coincide}(i)) divide $\xb^{\ab+\mathbf{1}}$. In fact a stronger statement is true; all vertex labels divide $\xb^{\ab}$. 
\end{itemize}
 Consider the cellular complex $\B_G^{q,c}$ with labels $\mb_F'$ for cells $F$ as in the proof of Theorem~\ref{thm:betti_coincide}(i). Relabel each cell $F$ with $\xb^{\ab+\mathbf{1}}/\mb_F'$. For simplicity, let us call $\B_G^{q,c}$ with its new labels $\Dc$. Let $\Dc_{\leq \ab}$ denote the subcomplex consisting of cells with labels dividing $\ab$. Let $(\MM_G^q)^{[\ab]}$ denote the Alexander dual of $\MM_G^q$ with respect to $\ab$ (\cite[Definition~5.20]{MillerSturmfels}). In this setting, \cite[Theorem~5.37]{MillerSturmfels} gives the following result: 

\begin{Proposition}
The polyhedral complex $(\Dc_G)_{\leq \ab}$ supports a minimal (cocellular) resolution for the ideal $(\MM_G^q)^{[\ab]}$.
\end{Proposition}

This observation has been made (independently) in \cite{Anton}. 


\subsection{Graphic matroid ideal and $h$-vectors}
\label{sec:hvector}
Let $\tilde{\Sb}=K[\zb]$ denote the polynomial ring in $m$ variables $\{z_e: e \in E(G)\}$. 
There is a surjective $K$-algebra homomorphism
\[
\pi \colon \Sb \rightarrow \tilde{\Sb}
\]
defined by sending both $y_e$ and $y_{\bar{e}}$ to $z_e$. The kernel of this map is the ideal generated by 
\[
\K=\{y_e-y_{\bar{e}} : e \in \Oc\}
\]
for some fixed orientation $\Oc$. We will denote this kernel by $\bff = \langle \K \rangle$. We get an induced isomorphism
\[
\tilde{\pi} \colon \Sb/\bff \rightarrow \tilde{\Sb} \ .
\]
We define the (unoriented) {\em graphic matroid ideal} $\Mat_G \subset \tilde{S}$ to be the image of $\OO_G^q+\bff$ under this isomorphism. Concretely, $\Mat_G$ is obtained from $\OO_G^q$ by identifying the variables $y_e$ and $y_{\bar{e}}$ and replacing them with $z_e$.

\begin{Lemma} \label{lem:Mat}
\begin{itemize}
\item[]
\item[(i)] $\K$ forms a regular sequence for $\Sb / \OO_G^q$.
\item[(ii)] $\B_G^{q,c}$ (equivalently $\A_G^{q,c}$) supports a minimal free resolution for $\Mat_G$. 
\item[(iii)] $\Mat_G$ is independent of the choice of $q$.
\item[(iv)] The $\ZZ$-graded Betti diagram of $\Mat_G$ coincides with the $\ZZ$-graded Betti diagrams of $\JJ_G$, $\II_G$, $\OO_G^q$, and $\MM_G^q$.
\end{itemize}
\end{Lemma}
\begin{proof}
(i) follows from \cite[Corollary~2.7]{novik2002syzygies}. Alternatively, by the explicit description of the facets in Proposition~\ref{prop:primdecom}(i), the restriction of each linear form in $\K$ spans a vector space of dimension $1$ and therefore the result follows from \cite[Lemma 2.4]{Stanley96}. 

(ii) follows from (i) and \cite[Lemma~3.15]{EisenbudSyz}.

There are several ways to see (iii). For example, it follows from (ii) and the discussion in \S\ref{sec:grobrel} (e.g. Proposition~\ref{prop:grobrelation}) that $\Mat_G$ is minimally generated by monomials 
\begin{equation}\label{eq:MatGens}
\{\prod_{e \in E(A^c,A)}{z_e} : A \subsetneq V(G), G[A] \text{ and } G[A^c] \text{ are connected} \} 
\end{equation}
 where $E(A^c,A)$ denotes the set of (unoriented) edges connecting $G[A]$ and $G[A^c]$. This description is independent of the choice of the base vertex $q$.

(iv) follows from Theorem~\ref{thm:betti_coincide} and \cite[Lemma~3.15]{EisenbudSyz}.
\end{proof}

It is a fact, essentially due to Hilbert, that the Hilbert series of a module is completely determined by its graded Betti table and its dimension. The numerator of the Hilbert series is called the $h$-polynomial. Its coefficients are obtained from the Betti numbers as an alternating sum and they form the $h$-vector (see, e.g, \cite[Theorem~8.20 and Theorem~8.23]{MillerSturmfels}). So we immediately obtain, from Lemma~\ref{lem:Mat}(iv), the following result.
\begin{Lemma} \label{lem:h}
The $h$-vectors of $\Sb / \JJ_G$, $ \Rb / \II_G$, $\Sb / \OO_G^q$, $\Rb / \MM_G^q$, and $\tilde{\Sb} / \Mat_G$ coincide.
\end{Lemma}

The ideal $\Mat_G$ has been extensively studied in the literature (see, e.g., \cite{StanleyCM}, \cite[Section~III.3]{Stanley96}, \cite[Section~3]{novik2002syzygies}). A more well known presentation of this ideal is by its prime decomposition; for each spanning tree $T$ of $G$, let 
$I_T=\langle z_e: e \in T \rangle$. The minimal prime decomposition of $\Mat_G$ is
\begin{equation} \label{eq:MatPrim}
\Mat_G = \bigcap_{T} I_T \ ,
\end{equation}
the intersection being over all spanning trees of $G$. This can be proved the same way as Proposition~\ref{prop:primdecom}(ii) (or can be deduced from it). 

\medskip

From \eqref{eq:MatPrim} it is evident that $\Mat_G$ is the Stanley-Reisner ideal of the simplicial complex $\Sigma$ of independent sets of the cographic matroid (i.e. the matroid whose bases are the complements of spanning trees of $G$). Therefore the $h$-polynomial of $\tilde{\Sb} / \Mat_G$ is precisely $T(1,y)$, where $T(x,y)$ is the Tutte polynomial of the graph (\cite[page 236]{Bjorner}). By Lemma~\ref{lem:h}, we obtain the following result:

\begin{Corollary} \label{cor:h}
$T(1,y)$ is the $h$-polynomial for $\Sb / \JJ_G$, $ \Rb / \II_G$, $\Sb / \OO_G^q$, $\Rb / \MM_G^q$, and $\tilde{\Sb} / \Mat_G$.
\end{Corollary}
Postnikov and Shapiro in \cite{PostnikovShapiro04} prove this result for $\tilde{\Rb} / \MM_G^q$ (equivalently, for $\Rb / \MM_G^q$) by a combinatorial argument. Merino's work in \cite{Merino} proves this result for $ \Rb / \II_G$ using deletion-contraction methods. A bijective proof of Merino's result was later presented in \cite{CoriLeBorgne} (see also \cite{FarbodMatt12}). We believe that Corollary~\ref{cor:h} gives a unified and more conceptual proof of these results. Moreover, Merino's theorem (stating that $T(1,y)$ is the generating function for the number of $q$-reduced divisors in various degrees) is a straightforward consequence of Corollary~\ref{cor:h} and Theorem~\ref{thm:Cori}. 


\begin{Remark}
It follows from our descriptions of the minimal free resolutions that the Castelnuovo-Mumford regularity of $\Sb / \JJ_G$, $ \Rb / \II_G$, $\Sb / \OO_G^q$, $\Rb / \MM_G^q$, and $\tilde{\Sb} / \Mat_G$ is equal to $g = m-n + 1$. Moreover, $g$ is equal to the projective dimension of the squarefree Alexander dual of $\OO_G^q$ (called
spanning tree ideal of $G$ in the literature). The minimal free resolution
of this ideal is explicitly given in \cite{Fatemehreliability}, and the degenerations from the toric ideals associated to graphs to the spanning tree ideals are studied in \cite{kateri2014family} in the context of statistical models.
\end{Remark}


\section*{Acknowledgments}
The first author is grateful to Volkmar Welker for his support and many helpful conversations during this project. She is also grateful to J\"urgen Herzog for helpful discussions related to regular sequences and flat families. The second author would like to thank his advisor Matthew Baker for his support throughout this project and for many useful conversations. We would like to thank Bernd Sturmfels for useful discussions. We would like to thank Winfried Bruns for valuable comments on an earlier draft of this paper. We also thank the anonymous referee for valuable suggestions. 





\end{document}